\documentclass{article}

\usepackage{geometry,chngpage}

\usepackage[english]{babel}
\usepackage{amsmath,amsthm,amssymb}
\usepackage{tikz-cd}
\usepackage{dsfont}
\usepackage{picture,float,graphicx,listings}
\usepackage{hyperref,verbatim}

\usepackage{pgf,tikz}
\usepackage{mathrsfs}
\usetikzlibrary{arrows}

\usepackage{todonotes}

\newtheorem{theorem}{Theorem}[section]
\newtheorem{lemma}[theorem]{Lemma}
\newtheorem{corollary}[theorem]{Corollary}
\newtheorem{proposition}[theorem]{Proposition}
\theoremstyle{remark}
\newtheorem*{remark}{Remark}

\theoremstyle{definition}
\newtheorem{definition}[theorem]{Definition}

\newtheorem{theoremIntro}{Theorem}
\newtheorem{propositionIntro}[theoremIntro]{Proposition}
\newtheorem*{definitionIntro}{Definition}

\newcommand{\Z}{\mathbb{Z}}

\newcommand{\R}{\mathbb{R}}
\newcommand{\C}{\mathbb{C}}

\renewcommand{\P}{\mathbb{P}}

\newcommand{\tend}[2]{\underset{#1 \rightarrow #2}{\longrightarrow}}

\newcommand{\rk}{\mathrm{rk}}
\newcommand{\im}{\mathrm{im}}
\newcommand{\coker}{\mathrm{coker}}
\newcommand{\coim}{\mathrm{coim}}
\newcommand{\Hom}{\mathrm{Hom}}
\newcommand{\End}{\mathrm{End}}

\newcommand{\Aut}{\mathrm{Aut}}
\newcommand{\GL}{\mathrm{GL}}

\newcommand{\tr}{\mathrm{tr}}

\newcommand{\Vol}{\mathrm{Vol}}

\newcommand{\OX}{\mathcal{O}}

\newcommand{\Lie}{\mathrm{Lie}}
\newcommand{\Stab}{\mathrm{Stab}}

\newcommand{\Id}{\mathrm{Id}}
\newcommand{\LC}{\mathrm{LC}}

\let\isanspoint\i
\renewcommand{\i}{\boldsymbol{\mathrm{i}}}
\newcommand{\e}{\boldsymbol{\mathrm{e}}}
\newcommand{\scal}[2]{\left< #1,#2 \right>}
\newcommand{\abs}[1]{\left| #1 \right|}
\newcommand{\norme}[1]{\left\| #1 \right\|}

\newcommand{\E}{\mathcal{E}}
\newcommand{\F}{\mathcal{F}}
\newcommand{\G}{\mathcal{G}}

\newcommand{\mB}{\mathcal{B}}
\newcommand{\mC}{\mathcal{C}}
\newcommand{\mP}{\mathcal{P}}
\newcommand{\mX}{\mathcal{X}}
\newcommand{\mZ}{\mathcal{Z}}
\newcommand{\mU}{\mathcal{U}}
\newcommand{\mL}{\mathcal{L}}
\newcommand{\dbar}{\overline{\partial}}

\newcommand{\ch}{\mathrm{ch}}

\newcommand{\sym}{\mathrm{sym}}

\newcommand{\Lo}{Łojasiewicz}
\newcommand{\Sz}{Székelyhidi}

\renewcommand{\leq}{\leqslant}
\renewcommand{\geq}{\geqslant}

\newcommand{\fonction}[4]{\left\{ \begin{array}{rcl}
		\displaystyle #1 & \longrightarrow & \displaystyle #2\\
		\displaystyle #3 & \longmapsto & \displaystyle #4
	\end{array} \right.}

\title{Perturbations of a Vector Bundle whose \\ Curvature Form Solves a Polynomial Equation}
\author{Rémi Delloque}
\date{}

\begin{document}

\maketitle
\begin{abstract}
    We investigate the behaviour of local perturbations of a wide class of geometric PDEs on holomorphic Hermitian vector bundles over a compact complex manifold. Our main goal is to study the existence of solutions near an initial solution under small deformations of both the holomorphic structure of the bundle and the parameters of the equation.
    
    Inspired by techniques from geometric invariant theory and the moment map framework, under suitable assumptions on the initial solution, we establish a local Kobayashi--Hitchin correspondence. A perturbed bundle admits a solution to the equation if and only if it satisfies a local polystability condition.

    We also show additional results, such as continuity and uniqueness of solutions when they exist, and a local version of the Kempf--Ness theorem. We also provide a local version of the Jordan--Hölder and Harder--Narasimhan filtrations.
\end{abstract}

\tableofcontents

\section{Introduction}

\paragraph{Motivations.}

Let $X$ be a connected compact complex manifold of dimension $n$. Several important complex geometric partial differential equations (PDEs) can be written as a polynomial in its unknown (with respect to the wedge product), a closed $(1,1)$-form. It is the case, for example, of the complex Monge--Ampère equation,
$$
\chi^n = \eta, \qquad \chi > 0, \qquad \chi \in [\chi_0],
$$
the $J$-equation,
$$
c\chi^n - \omega \wedge \chi^{n - 1} = 0, \qquad \chi > 0, \qquad \chi \in [\chi_0],
$$
or the deformed Hermitian Yang--Mills (dHYM) equation,
$$
\Im\left(\e^{-\i\theta}(\omega + \i\chi)^n\right) = 0, \qquad \chi \in [\chi_0].
$$
Moreover, if there is a complex line bundle $L \rightarrow X$ such that $c_1(L) = [\chi_0]$, by Chern--Weil theory, the condition $\chi \in [\chi_0]$ is equivalent to saying that $\chi$ is the reduced curvature form of an integrable Chern connection $\nabla$ on $L$ \textit{i.e.}
$$
\chi = -\frac{1}{2\i\pi}\nabla \circ \nabla \in \Omega^{1,1}(X,\R).
$$

Similarly, the Hermitian Yang--Mills (HYM) equation is
$$
\omega^{n - 1} \wedge \hat{F} = c\omega^n,
$$
where $\hat{F}$ is the reduced curvature form of a given holomorphic Hermitian vector bundle $\E = (E,h,\dbar) \rightarrow X$. With that in mind, a higher rank version of the three equations above has been introduced recently~: by Collins and Yau for the dHYM equation \cite[Paragraph 8.1]{Collins_Yau_2018}, by Pingali for the Monge--Ampère equation \cite{Pingali} and by Takahashi for the $J$-equation \cite{Takahashi}. In general, following \cite{DNST}, when we say that $\E = (E,h,\dbar)$ admits a solution to a polynomial equation in its curvature, we mean that there is a Dolbeault operator $\dbar'$ on the gauge orbit of $\dbar$ such that the curvature of the Chern connection of $(E,h,\dbar')$ solves the equation. Here, we consider the gauge group
$$
\G^\C(E) = \{f \in \Omega^0(X,\End(E))|\forall x \in X, f_x \in \GL(E_x)\},
$$
acting on Dolbeault operators as follows,
$$
f \cdot \dbar = f \circ \dbar \circ f^{-1}.
$$

The ultimate goal with these equations is to be able to give a necessary and sufficient condition for the existence of solutions in terms of numerical invariants. On line bundles, such a condition was found for the Monge--Ampère equation by Yau \cite{Yau} and for the $J$-equation by G. Chen \cite[Theorem 1.1]{Chen} before being improved by Song \cite{Song}. On general vector bundles, the only known numerical condition of this kind concerns the HYM equation and is imposed by geometric invariant theory (GIT). This condition is the slope stability, a result known as the Kobayashi--Hitchin correspondence, proven by Donaldson on surfaces \cite{Donaldson} and by Uhlenbeck and Yau in higher dimensions \cite{Uhlenbeck_Yau}. See \cite{Lübke_Teleman} for a survey.

For the dHYM equation (on line bundles), the expected numerical condition is the Collins--Yau--Jacob conjecture \cite{Collins_Yau_Jacob}. It was solved by G. Chen in the super-critical regime \cite[Theorem 1.7]{Chen} (after a slight reformulation of the conjecture) but it does not hold outside of this regime \cite[Remark 1.10]{Chen}. G. Chen's correspondence was later improved by Chu, Lee and Takahashi \cite[Theorem 1.2]{CLT} and by Ballal in the projective case \cite[Corollary 1]{Ballal}.

For a general polynomial equation on the curvature, especially in higher rank, it seems too ambitious to expect such a numerical condition to hold. The idea of this paper is to start from a Hermitian holomorphic vector bundle $\E_0 \rightarrow X$ whose curvature satisfies a polynomial equation in which the coefficients are closed forms and to study local perturbations of this bundle and of the coefficients of the equation. We study the impact of these deformations on the existence of solution and their behaviour when they exist. We call such equations the \textit{$P$-critical equations}.

This idea is mostly inspired from the work of Dervan, McCarthy and Sektnan \cite{DMS}, and Buchdahl and Schumacher \cite{Buchdahl_Schumacher}. In \cite{Buchdahl_Schumacher}, the equation solved by $\E_0$ is the HYM equation and only the complex structure of the bundle is perturbed. Buchdahl and Schumacher associate to each small perturbation of $\E_0$ a point in a finite dimensional vector space on which the automorphisms group of $\E_0$ acts linearly. The origin represents $\E_0$. It is shown that any small enough perturbation of $\E_0$ admits a solution to the HYM equation if and only if the associated orbit is closed, providing a link between a geometric PDE and a GIT condition. This technique is inspired from \Sz\ on cscK equation \cite{Szkelyhidi}.

In \cite{DMS}, Dervan, McCarthy and Sektnan consider a slope semi-stable bundle $\E$. Thus, the sheaf $\E_0 = \mathrm{Gr}(\E)$ (which is assumed to be locally free) solves the HYM equation with the right choice of complex structure. This means that $\E$ is an infinitesimal perturbation of $\E_0$. In other words, the orbit built by Buchdahl and Schumacher contains the origin in its closure. But the coefficients of the equation vary, and the HYM equation is seen as a "large volume limit" of a wide class of equations~: the $Z$-critical equations. Dervan, McCarthy and Sektnan give a necessary and sufficient numerical condition for $\E$ to admit solutions to these equations when the volume is large enough~: the asymptotic $Z$-stability. This is an asymptotic stability condition for $\E$ with respect to its sub-bundles which destabilise $\E_0$. Dervan, McCarthy and Sektnan generalise a previous work by Leung \cite{Leung} where the asymptotic $Z$-stability corresponds to the Gieseker stability \cite[Sub-section 4.2]{Keller_Scarpa}, \cite[Example 3.13]{Delloque_2025}.

Outside of the large volume limit, there are a few known results on general polynomial equations. In \cite{Keller_Scarpa}, Keller and Scarpa find necessary conditions for the existence of sub-solutions and of solutions on vector bundles of rank $2$ on surfaces ; namely $P$-positivity and $P$-stability \cite[Theorem 1.1]{Keller_Scarpa}. They conjecture that the converse should hold, at least under suitable additional assumptions on the bundle \cite[Conjecture 1.4]{Keller_Scarpa}. In \cite{DNST}, Napame, Scarpa, Tipler and I study $P$-positivity and $P$-stability of equivariant sheaves on $G$-varieties. We show that in some cases, $P$-positivity and $P$-stability are equivalent to their equivariant analogue \cite[Theorem 1.9, Theorem 1.11]{DNST}. We also study how $P$-positivity of bundles is affected by blow-ups of points \cite[Theorem 1.13]{DNST}.

\paragraph{Method.}

Atiyah and Bott noticed that the HYM equation has an infinite dimensional moment map interpretation \cite{Atiyah_Bott}. It is not the only geometric PDE which admits such an interpretation. For example, Fujiki shows a moment map interpretation for the constant scalar curvature (cscK) equation \cite{Fujiki}, see also \cite{Donaldson_97}. X. Chen and Sun use this infinite dimensional moment map interpretation in this context \cite{Chen_Sun} to study small deformations of cscK manifolds. Later, Georgoulas, Robbin and Salamon wrote an excellent survey \cite{GRS} on the properties of moment maps and their flows in a general compact finite dimensional setting. Most of the GIT results of this paper either come from this book, or are proved by adapting a proof from this book.

Dervan and Hallam gave a large class of geometric PDEs that admit a moment map interpretation \cite{Dervan_Hallam}. However, in this class, the closed $(1,1)$-form associated to the moment map may be degenerate. In the case of the $P$-critical equation, the same moment map interpretation exists with the same degeneration issue. In \cite{DMS}, the points (here, the Dolbeault operators) where this $(1,1)$-form is non-degenerate are called sub-solutions. We use the same terminology here. This sub-solution condition is open (with the right topology) so in our local perturbation problem, the $(1,1)$-form will be a Kähler form, yielding a usual moment map interpretation. This sub-solution condition is crucial so we require it in our definition of $P$-critical bundles.

The HYM equations have this particularity that all Dolbeault operators $\dbar$ are sub-solutions. It is probably the reason why they are the most understood among $P$-critical equations. Actually, the Kobayashi--Hitchin correspondence can be seen as an infinite dimensional version of the Hilbert--Mumford criterion in the polystable case \cite[Theorem 12.5]{GRS}.

Moreover, as in \cite{Buchdahl_Schumacher}, we use the Kuranishi slice to see the set of small perturbations of $\E_0$ as a germ of analytic subset of a finite dimensional vector space. This enables us to avoid the complications caused by infinite dimensional constructions. Moreover, we deform the Kuranishi slice as in \cite{Donaldson_08,Szkelyhidi,Sektnan_Tipler,Ortu_Sektnan,Ortu} in the context of the cscK equation or \cite{Clarke_Tipler,Delloque_2024} in the context of the HYM equation, to make sure that the zeroes of the moment map we build yield solutions to the $P$-critical equation.

\paragraph{Main results.}

Let $\E_0 = (E,h,\dbar_0) \rightarrow X$ be a holomorphic Hermitian vector bundle, $\nabla$ its Chern connection and
$$
\hat{F} = -\frac{1}{2\i\pi}\nabla \circ \nabla
$$
its reduced curvature form. Assume that it satisfies a polynomial equation of the form
$$
\sum_{k = 0}^n \zeta_{0,k} \wedge \hat{F}^k = 0,
$$
where for all $k$, $\zeta_{0,k}$ is a closed real $(n - k,n - k)$-form. Assume moreover that the sub-solution condition is verified (see Definition \ref{DEF:Sous-solution P}). We call this equation the \textit{$P_{\zeta_0}$-critical equation} to emphasize the dependency in $\zeta$. We also say that $\E_0$ is $P_{\zeta_0}$-critical. The formalism of $P$-critical equations was first introduced by Scarpa \cite[Sub-section 1.1]{DNST}.

We first show a result which is well-known in the case of the HYM equation.

\begin{propositionIntro}[Proposition \ref{PRO:Polysimplicité}]\label{PRO:1}
    We have an orthogonal and holomorphic decomposition
    $$
    \E_0 = \bigoplus_{k = 1}^l \G_{0,k}^{m_k},
    $$
    where the $m_k \geq 1$ are integers and the $\G_{0,k}$ are simple $P_{\zeta_0}$-critical Hermitian holomorphic vector bundles. Moreover, for all $i \neq j$, $\Hom(\G_{0,i},\G_{0,j}) = 0$.
\end{propositionIntro}

We show that the existence of a sub-solution on a bundle implies the existence of a positive $(n - 1,n - 1)$ class in $X$, see Proposition \ref{PRO:X équilibrée}. In particular, there is a metric on $X$ such that the associated $(1,1)$-form $\omega$ is balanced \textit{i.e.} $d\omega^{n - 1} = 0$. Let $V = H^{0,1}(X,\End(\E_0))$ be the complex finite dimensional space of harmonic $(0,1)$-forms with respect to $\omega$. Then, the automorphisms group of $\E_0$ acts linearly on $V$ and any small deformation of $\E_0$ has a corresponding orbit in $V$. Let $\E_b = (E,\dbar_b)$ be the bundle represented by $b \in V$. Unless it is mentioned, the Dolbeault operator is \textit{a priori} not integrable, so $\E_b$ may not be a holomorphic bundle, but the results of this paper still apply. Let for all $\zeta = \sum_{k = 0}^n \zeta_k$ and all smooth complex vector bundle $F$,
$$
P_\zeta(F) = \sum_{k = 0}^n k![\zeta_k] \cup \ch_k(F).
$$
Chern--Weil theory implies that $P_\zeta(E) = 0$ (see Lemma \ref{LEM:P(E) = 0}). We then define a notion analogous to the slope (semi-)(poly)stability : the \textit{local $P_\zeta$-(semi-)(poly)stability}.

\begin{definitionIntro}[Definition \ref{DEF:Stabilité P}]
    $\E_b$ is locally $P_\zeta$-semi-stable if for all smooth bundles $0 \subsetneq F \subsetneq E$ such that $(F,\dbar_b) \subset \E_b$, $(F,\dbar_0) \subset \E_0$ and $(F^\bot,\dbar_0) \subset \E_0$ are holomorphic embeddings, we have
    \begin{equation}\label{EQ:Stabilité}
        P_\zeta(F) \leq 0.
    \end{equation}
    If is said to be locally $P_\zeta$-stable if (\ref{EQ:Stabilité}) is always strict. It is said to be locally $P_\zeta$-polystable if (\ref{EQ:Stabilité}) is strict or $(F^\bot,\dbar_b) \subset \E_b$ is holomorphic.
\end{definitionIntro}

The next theorem can be seen as a local Kobayashi--Hitchin like correspondence. It confirms a local version of a conjecture of Keller and Scarpa \cite[Conjecture 1.4]{Keller_Scarpa}. Their conjecture asserts that the existence of a sub-solution is related to a kind of positivity with respect to the sub-varieties of $X$, called $Z$-positivity \cite[Definition 1.2]{Keller_Scarpa}. Under this assumption, the existence of a solution is conjectured to be related to a kind of stability with respect to the sub-sheaves of the bundle, called $Z$-stability \cite[Definition 1.2]{Keller_Scarpa}. They use the formalism of $Z$-critical equations but their results apply to any $P$-critical equation.

\begin{theoremIntro}[Theorem \ref{THE:Déformation P-critique}]\label{THE:2}
    If $\zeta$ is close enough to $\zeta_0$ and $b$ is close enough to $0$, $\E_b$ admits a solution to the $P_\zeta$-critical equation close to $\dbar_0$ if and only if it is locally $P_\zeta$-polystable.
\end{theoremIntro}

We also give a characterisation involving a moment map defined on a neighbourhood of $0$ in $V$ which depends continuously on $\zeta$. We don't go into the details of what we mean by "close" here, but it implies in particular that all the solutions we consider are sub-solutions. See Theorem \ref{THE:Déformation P-critique} for more details. The topological considerations in Banach spaces are mainly discussed in Section \ref{SEC:Préliminaires}.

Moreover, the solutions of this equation, when they exist, are locally unique modulo a unitary gauge transform where the unitary gauge group is
$$
\G(E,h) = \{u \in \G^\C(E)|uu^\dagger = \Id_E\} \subset \G^\C(E).
$$
And the $\G(E,h)$-orbit of the solution, when it exists, varies continuously with $\zeta$ and $b$.

\begin{theoremIntro}[Theorem \ref{THE:Unicité solutions}]\label{THE:3}
    If $\zeta$ is close enough to $\zeta_0$ and $\dbar$ and $\dbar'$ are two solutions to the $P_\zeta$-critical equation close enough to $\dbar_0$ and in the same $\G^\C(E)$-orbit, then, they are in the same $\G(E,h)$-orbit.
\end{theoremIntro}

\begin{theoremIntro}[Theorem \ref{THE:Continuité des opérateurs P-critiques}]\label{THE:4}
    When it exists, the $\G(E,h)$-orbit of any $P_\zeta$-critical operator $\dbar$ in the gauge orbit of $\dbar_b$ varies continuously with respect to $\zeta$ and $b$.
\end{theoremIntro}

Once more, the topologies we use to define the continuity are given precisely in Theorem \ref{THE:Continuité des opérateurs P-critiques}. Theorems \ref{THE:2} and \ref{THE:4} are powerful tools to build canonical metrics as deformations of an initial special metric. They generalise \cite[Theorem 1.4]{Takahashi} and \cite[Theorem 1.6]{Takahashi} on the $J$-equation, \cite[Theorem 4.1]{Pingali} on the Monge--Ampère equation, \cite[Theorem 1.1]{DMS} on the asymptotic $Z$-critical equations (which are deformations of the HYM equation), \cite[Theorem 1.2]{Clarke_Tipler} and \cite[Theorem 3]{Delloque_2024} on the HYM equation and \cite[Proposition 3.16]{Keller_Scarpa} on the $Z$-critical equation.

Then, in the case where $\zeta = \zeta_0$, we obtain a generalisation of the results of Buchdahl and Schumacher \cite{Buchdahl_Schumacher}. We still assume that $b$ is close to $0$.

\begin{theoremIntro}[Theorem \ref{THE:Cas zeta = zeta_0}]\label{THE:5}
    If $\zeta = \zeta_0$, $\E_b$ is locally $P_{\zeta_0}$-semi-stable. It is locally $P_{\zeta_0}$-polystable if and only if the orbit of $b$ in $V$ with respect to the action of the automorphisms group of $\E_0$ is closed.
\end{theoremIntro}

When they exist, the solutions to the $P_\zeta$-critical equation are built as limits of the negative gradient flow of the moment maps on $V$. We call these flows the \textit{moment map flows}. Still under the assumptions of $\zeta$ close to $\zeta_0$ and $b$ close to $0$, if $\E_b$ is only locally $P_\zeta$-semi-stable, it can be seen as a small deformation of a locally $P_\zeta$-polystable bundle arising from a Jordan--Hölder filtration.

\begin{theoremIntro}[Theorem \ref{THE:FJH admissible}]\label{THE:6}
    If $\E_b$ is locally $P_\zeta$-semi-stable, there is a $b_\infty$ close to $0$, which is unique modulo the action of the group of unitary automorphisms of $\E_0$ such that $\E_b$ is an infinitesimal deformation of $\E_{b_\infty}$ and $\E_{b_\infty}$ is locally $P_\zeta$-polystable.
    
    Moreover, $b_\infty$ can be reached as the limit of the moment map flows and $\E_{b_\infty}$ can be constructed from a Jordan--Hölder filtration of $\E_b$. We call $\mathrm{Gr}(\E_b) = \E_{b_\infty}$ the \textit{graded object of $\E_b$}, which is unique up to an isomorphism.
\end{theoremIntro}

In the general case, we can also build a unique Harder--Narasimhan filtration. See Theorem \ref{THE:FHN admissible} for a precise definition.

\begin{theoremIntro}[Theorem \ref{THE:FHN admissible}]\label{THE:7}
    $\E_b$ admits a unique Harder--Narasimhan filtration. Moreover, the limit point $b_\infty$ of the moment map flows represents a bundle $\E_{b_\infty}$ which is the direct sum of the graded objects of the locally $P_\zeta$-semi-stable quotients of the Harder--Narasimhan filtration of $\E_b$.
\end{theoremIntro}

Finally, we show a local version of the Kempf--Ness theorem, which is usually a homeomorphism between a GIT quotient and a symplectic quotient. Let $L_{\mathrm{ps}}^{\mathrm{int}}$ be the set of all $(\zeta,b)$ close to $(0,0)$ such that $\E_b$ is locally $P_\zeta$-polystable and $\dbar_b$ is integrable. Recall that $\Aut(\E_0)$ acts on $b \in V$.

\begin{theoremIntro}[Theorem \ref{THE:Kempf--Ness}]\label{THE:8}
    For each $\zeta$ close to $\zeta_0$, there is a germ of homeomorphism at the origin,
    $$
    \{b|(\zeta,b) \in L_{\mathrm{ps}}^{\mathrm{int}}\}/\Aut(\E_0) \longrightarrow \{\dbar|\dbar \textrm{ is integrable and }(E,h,\dbar) \textrm{ is $P_\zeta$-critical}\}/\G(E,h).
    $$
    Moreover, they vary continuously with $\zeta$.
\end{theoremIntro}

A more complete statement can be found later (see Theorem \ref{THE:Kempf--Ness}) but it requires to introduce many definitions and notations to be stated properly. We restrict ourselves to this partial presentation to avoid having a too cumbersome introduction.

\paragraph{Organisation of the paper.}

In Section \ref{SEC:Préliminaires}, we define the notions of sub-solutions and solutions to the $P$-critical equation. We study the properties of a $P$-critical bundle. In particular, we show Proposition \ref{PRO:1}. We also study the properties of the linearisation of the $P$-critical equation. We show that under the sub-solution condition, it is a non-negative real symmetric (for a suitable Hermitian product) elliptic operator of order $2$ and its kernel is exactly the space of global holomorphic sections of $\E$.

In Section \ref{SEC:Exemples}, we give more detailed examples of polynomial equations in the curvature and their appearance in the literature.

In Section \ref{SEC:Kuranishi}, we give the properties of the Kuranishi slice and we construct the deformed Kuranishi slices. We also define the associated moment maps, including the infinite dimensional moment map. We end this section with the precise statements of Theorems \ref{THE:2} and \ref{THE:5}. We also give the topological structure of the locally $P_\zeta$-(semi-)(poly)stable loci with Proposition \ref{PRO:Topologies lieux stables} but we prove it in Section \ref{SEC:Preuve du théorème}.

In Section \ref{SEC:Résultats GIT locaux}, we prove all the intermediate results and technical tools we need for the next section. They are mostly GIT results as the ones we can find in \cite{GRS} but with the constraint that we work on an open ball of $V$, which is non-compact.

In Section \ref{SEC:Preuve du théorème}, we complete the proofs of Theorems \ref{THE:2} and \ref{THE:5}. We also prove Theorems \ref{THE:3}, \ref{THE:4} and \ref{THE:6}.

In Section \ref{SEC:Kempf--Ness}, we prove Theorems \ref{THE:7} and \ref{THE:8}.

Finally, in Section \ref{SEC:Exemple avec la J-équation}, we apply Theorem \ref{THE:2} on a concrete case. We build a bundle $\E_0$ of rank $2$ which solves the $J$-equation, and two infinitesimal simple deformations $\E_1$ and $\E_2$ of $\E_0$, on the blow-up of a point in $\P^2$. We find a necessary and sufficient numerical condition for them to admit solutions to the dHYM equation with respect to a varying Kähler form, in a low volume regime.

\paragraph{Acknowledgements.}

I thank my PhD advisor Carl Tipler for suggesting this interesting problem to me, and for his guidance and support toward its resolution. I thank Carlo Scarpa for useful discussions about the $Z$-critical/$P$-critical equations. In particular, the idea of considering any polynomial equation in the curvature instead of the ones arising from central charges only, comes back to him. I thank Annamaria Ortu for useful discussions about the moment map picture for the cscK equation. I thank Ruadha\'\isanspoint\ Dervan and Lars Martin Sektnan for useful discussions about deformations of geometric PDEs which admit a moment map interpretation and their explanations of Dervan--Hallam theory \cite{Dervan_Hallam}. I thank Ruadha\'\isanspoint\ Dervan, Julien Keller, Annamaria Ortu, Carlo Scarpa, Lars Martin Sektnan and Carl Tipler for their feedbacks on this paper and useful references. I thank Mats Andersson and Richard Lärkäng for interesting discussions in an attempt to generalise the results of this paper to non locally free sheaves.

\section{Preliminaries}\label{SEC:Préliminaires}

\subsection{Framework}

Let $X$ be a connected compact complex manifold of dimension $n$ endowed with a complex Hermitian vector bundle $(E,h) \rightarrow X$.

When $\E = (E,\dbar,h)$ is a Hermitian vector bundle on $X$ endowed with a (non necessarily integrable) Dolbeault operator $\dbar : \Omega^0(X,E) \rightarrow \Omega^{0,1}(X,E)$, we set $\hat{F}(E,\dbar,h) = -\frac{1}{2\i\pi}\nabla \circ \nabla \in \Omega^2(X,\End(E))$ to be the associated reduced curvature form, where $\nabla$ is the Chern connection associated with $\dbar$ and $h$. We set
$$
\mP_\zeta(E,\dbar,h) = \sum_{k = 0}^n \zeta_k \wedge \hat{F}(E,\dbar,h)^k \in \Omega^{n,n}(X,\End(E)),
$$
where each $\zeta_k$ is a closed $(n - k,n - k)$-form with values in $\R$. This operator was first introduced by Scarpa \cite[Sub-section 1.1]{DNST}. When there is no ambiguity on $E$ and $h$, we may drop them from the notation. We then define
$$
P_\zeta(E) = [\tr(\mP_\zeta(E,\dbar,h))] = \sum_{k = 0}^n k![\zeta_k] \cup \ch_k(E) \in \R,
$$
which only depends on the topology of $E$ and is additive on short exact sequences.

We need now to define the formal derivative of $\mP_\zeta(E,\dbar,h)$ with respect to $\hat{F}(E,\dbar,h)$. For this, let us introduce the symmetric product of endomorphism valued forms. Let $d_1,\ldots,d_m$ be integers. When $\sigma \in \frak{S}_m$ is a permutation, we denote by $\mathrm{grsgn}(\sigma) \in \Z/2\Z$ its graded sign \textit{i.e.} the integer modulo $2$ such that for all forms $\alpha_1,\ldots,\alpha_m$ with values in $\C$ and homogenous degree $\deg(\alpha_k) = d_k$,
$$
\alpha_{\sigma(1)} \wedge \cdots \wedge \alpha_{\sigma(m)} = (-1)^{\mathrm{grsgn}(\sigma)}\alpha_1 \wedge \cdots \wedge \alpha_m.
$$
It is uniquely defined by the parities of the $d_k$. Notice that it always equals $1$ if the $d_k$ are all even and it equals the usual sign of a permutation if they are all odd. When $A_1,\ldots,A_m$ are forms with values in $\End(E)$ with homogenous degree $\deg(A_k) = d_k$, let
$$
[A_1,\ldots,A_m]_\sym = \frac{1}{m!}\sum_{\sigma \in \frak{S}_m} (-1)^{\mathrm{grsgn}(\sigma)}A_{\sigma(1)} \wedge \cdots \wedge A_{\sigma(m)}.
$$
Notice that it is the usual product of the $A_k$ if they commute. We can now define the formal derivative of $\mP_\zeta(E,\dbar,h)$ with respect to $\hat{F}(E,\dbar,h)$ as a linear map $\Omega^{1,1}(X,\End(E)) \rightarrow \Omega^{n,n}(X,\End(E))$ by
$$
\mP_\zeta'(E,\dbar,h) : \kappa \mapsto \sum_{k = 1}^n k\zeta_k \wedge [\underbrace{\hat{F}(E,\dbar,h),\ldots,\hat{F}(E,\dbar,h)}_{k - 1 \textrm{ times}},\kappa]_\sym.
$$
We abusively write $[\mP_\zeta'(E,\dbar,h) \wedge \kappa]_\sym$ for $(\mP_\zeta'(E,\dbar,h))(\kappa)$. Similarly, we define a bilinear map\\
$\Omega^{1,0}(X,\End(E)) \times \Omega^{0,1}(X,\End(E)) \rightarrow \Omega^{n,n}(X,\End(E))$ by
$$
\mP_\zeta'(E,\dbar,h) : (\alpha,\beta) \mapsto \sum_{k = 1}^n k\zeta_k \wedge [\underbrace{\hat{F}(E,\dbar,h),\ldots,\hat{F}(E,\dbar,h)}_{k - 1 \textrm{ times}},\alpha,\beta]_\sym.
$$
We abusively write $[\mP_\zeta'(E,\dbar,h) \wedge \beta \wedge \gamma]_\sym$ for $(\mP_\zeta'(E,\dbar,h))(\beta,\gamma)$. We hope no confusion occurs about the fact that in general, $[\mP_\zeta'(E,\dbar,h) \wedge \beta \wedge \gamma]_\sym \neq [\mP_\zeta'(E,\dbar,h) \wedge (\beta \wedge \gamma)]_\sym$. These quite convoluted expressions will appear in the linearisation of the $P$-critical equation \cite[Section 2.3.1]{DMS}.

\begin{definition}[{\cite[Definition 2.35]{DMS}, \cite[Definition 2.6]{Keller_Scarpa}, \cite[Definition 2.1]{DNST}}]\label{DEF:Sous-solution P}
    We say that $\dbar$ is a \textit{sub-solution of the $P$-critical equation} if for all points $x \in X$ and all non-zero tangent vectors $v \in T_x^{1,0}X \otimes \End(E_x)\backslash\{0\}$,
    $$
    \tr([\mP_\zeta'(E,\dbar,h) \wedge \i v \wedge v^\dagger]_\sym) > 0,
    $$
    where $v^\dagger \in T_x^{0,1}X \otimes \End(E_x)$ is the adjoint of the endomorphism $v$ with respect to the Hermitian product $h_x$.
\end{definition}

\begin{definition}[{\cite[Sub-section 1.1]{DNST}}]
    We say that $\dbar$ is a \textit{solution of the $P$-critical equation} if it is a sub-solution and
    $$
    \mP_\zeta(E,\dbar,h) = 0.
    $$
    We also say that $\E$ is $P$-critical in this case, or $P_\zeta$-critical to emphasize the dependence in $\zeta$.
\end{definition}

\begin{remark}
In some papers like \cite{DMS,Keller_Scarpa}, being a sub-solution is not a requirement for being a solution. In this paper, it is a crucial assumption so we add it to the definition of $P$-critical bundles. Soon enough, we will consider a solution $\dbar_0$ to this equation and look for other solutions in a close neighbourhood of $\dbar_0$. Therefore, the sub-solution condition, which is open, will always be satisfied.
\end{remark}

Notice that this equality imposes a condition on the topology of $E$ and on the de Rham classes of the $\zeta_k$.

\begin{lemma}\label{LEM:P(E) = 0}
    If $(E,h)$ admits a solution to the $P$-critical equation, then $P_\zeta(E) = 0$.
\end{lemma}
\begin{proof}
Indeed, when we integrate the trace of the equality $\mP_\zeta(E,\dbar,h) = 0$, we obtain
$$
0 = \int_X \tr(\mP_\zeta(E,\dbar,h)) = P_\zeta(E).
$$
\end{proof}

Notice also that the existence of a sub-solution on $E$ imposes a condition on $X$ and a topological condition on $E$.

\begin{proposition}[{\cite[Lemma 3.3]{Keller_Scarpa}}]\label{PRO:X équilibrée}
    If $(E,h)$ admits a sub-solution $\dbar$ to the $P$-critical equation, then $X$ is balanced \textit{i.e.} it admits a Hermitian metric $g$ whose associated $(1,1)$-form $\omega$ verifies $d\omega^{n - 1} = 0$. Moreover, we may choose it so
    $$
    \omega^{n - 1} \in \sum_{k = 1}^n k![\zeta_k] \cup \ch_{k - 1}(E).
    $$
    In particular, this class is positive.
\end{proposition}
\begin{proof}
Set
$$
\Theta = \tr\left(\sum_{k = 1}^n k\zeta_k \wedge \hat{F}(E,\dbar,h)^{k - 1}\right).
$$
$\Theta$ is closed and if $\dbar$ is a sub-solution, for all $x \in X$ and all $v \in T_x^{1,0}X\backslash\{0\}$,
\begin{align*}
    \Theta \wedge \i v \wedge \overline{v} & = \tr\left(\sum_{k = 1}^n k\zeta_k \wedge \hat{F}(E,\dbar,h)^{k - 1} \wedge \i v\Id_E \wedge (v\Id_E)^\dagger\right)\\
    & = \tr([\mP'(E,\dbar,h) \wedge \i v\Id_E \wedge (v\Id_E)^\dagger]_\sym) > 0.
\end{align*}
We used the fact that homotheties commute with any endomorphism. We deduce that $\Theta$ is a positive $(n - 1,n - 1)$-form. By the usual theorems on balanced manifolds \cite{Michelsohn}, there is a unique Hermitian metric $g$ on $X$ whose associated $(1,1)$-form $\omega$ verifies $\omega^{n - 1} = \Theta$, hence $g$ is balanced. Clearly, $[\Theta] = \sum_{k = 1}^n k![\zeta_k] \cup \ch_{k - 1}(E)$.
\end{proof}

From now on, fix a balanced metric $g$ on $X$ and call $\omega$ the associated $(1,1)$-form and $\Vol = \frac{\omega^n}{n!}$ the associated volume form.

In general, when given a Hermitian vector bundle equipped with a (not necessarily integrable) Dolbeault operator $(E,\dbar,h)$, we are interested by Dolbeault operators $\dbar'$ which are solutions to the $P$-critical equation and such that $(E,\dbar')$ is isomorphic to $(E,\dbar)$. This is equivalent to saying that they lie in the same orbit under the action of the complex gauge group of $E$
$$
\G^\C(E) = \{f \in \Omega^0(X,\End(E))|\forall x \in X, f_x \in \GL(E_x)\},
$$
given by,
$$
f \cdot \dbar = f \circ \dbar \circ f^{-1} = \dbar + f\dbar(f^{-1}).
$$
This action preserves integrable connections. We also introduce the unitary gauge group of $(E,h)$,
$$
\G(E,h) = \{u \in \G^\C(E)|uu^\dagger = \Id_E\},
$$
where $u^\dagger$ is the adjoint of $u$ with respect to the metric $h$. $\G^\C(E)$ can be seen as a formal complexification of $\G(E,h)$.

\subsection{$L^2$ Hermitian products and elliptic linearisation}\label{SEC:Produits scalaires}

From now on, we fix a Hermitian complex vector bundle $(E,h)$ on $X$.

When $W$ is a space of smooth sections, we call $L_d^p(W)$ its $L_d^p$ Sobolev completion, $\mC^d(W)$ its $\mC^d$ completion and $\mC^{d,\alpha}(W)$ its Hölder $\mC^{d,\alpha}$ completion. We can use any background metrics to define the associated norms as they are all equivalent. Let $\mX_\zeta$ be the set of all Dolbeault operators $\dbar$ which are sub-solutions to the $P_\zeta$-critical equation. It is a subset of an affine space of direction $\Omega^{0,1}(X,\End(E))$. The map $\dbar \mapsto \hat{F}(\dbar)$ is continuous from $L_d^p$ to $L_{d - 1}^p$ if $d \geq 1$. If moreover, $(d - 1)p > 2n$, $L_{d - 1}^p \subset \mC^0$ continuously so $\hat{F}(\dbar)$ varies continuously in $\mC^0$ thus $\mP_\zeta'(\dbar)$ too. Since the required condition to be a sub-solution is a pointwise positivity condition, it makes being a sub-solution an open condition for this topology. Similarly, for all $d \geq 1$ and all $0 < \alpha \leq 1$, $\mX_\zeta$ is open for the $\mC^d$ and the $\mC^{d,\alpha}$ topologies.

We also define the completions $L_d^p(\mX_\zeta)$, $\mC^d(\mX_\zeta)$ and $\mC^{d + \alpha}(\mX_\zeta)$ as the set of all $\dbar$ with the right regularity such that the operator $\mP_\zeta'(\dbar)$ satisfies the positivity condition pointwise. It makes sense because the induced curvature is continuous in this case. Moreover, for any Dolbeault operator $\dbar_0$ on $E$,
$$
L_d^p(\mX_\zeta) \subset \dbar_0 + L_d^p(\Omega^{0,1}(X,\End(E))) \textrm{ is open if } (d - 1)p > 2n.
$$
And similarly,
$$
\mC^d(\mX_\zeta) \subset \dbar_0 + \mC^d(\Omega^{0,1}(X,\End(E))) \textrm{ is open if } d \geq 1,
$$
$$
\mC^{d,\alpha}(\mX_\zeta) \subset \dbar_0 + \mC^{d,\alpha}(\Omega^{0,1}(X,\End(E))) \textrm{ is open if } d \geq 1.
$$
Finally, we complete the gauge groups. When $(d - 1)p > n$, let
$$
L_{d + 1}^p(\G^\C(E)) = \{g \in L_{d + 1}^p(\Omega^0(X,\End(E)))|\forall x \in X, g_x \in \GL(E_x)\},
$$
$$
L_{d + 1}^p(\G(E,h)) = \{u \in L_{d + 1}^p(\G^\C(E))|uu^\dagger = \Id_E\}.
$$
Similarly, we define their $\mC^{d + 1}$ and their $\mC^{d + 1,\alpha}$ completions when $d \geq 1$ and $0 < \alpha \leq 1$.

In the rest of this section, we endow the spaces of smooth sections/forms with values in $\End(E)$ with $L^2$ Hermitian products. First of all, we endow the space $\Omega^0(X,\End(E))$ of smooth sections of $\End(E)$ with the usual $L^2$ Hermitian product.
$$
\scal{f}{g}_{L^2} = \int_X h_{\End(E)}(f,g) \, \Vol = \int_X \tr(fg^\dagger) \, \Vol.
$$
Recall that $g^\dagger$ is the adjoint of $g$ with respect to $h$. For the space $\Omega^{1,0}(X,\End(E))$ of smooth $(1,0)$-forms of $\End(E)$, sub-solutions provide natural Hermitian products. Indeed, let us define
$$
\scal{\alpha}{\beta}_{L^2(\dbar,\zeta)} = \int_X \tr([\mP_\zeta'(\dbar) \wedge \i\alpha \wedge \beta^\dagger]_\sym).
$$
It is a sesquilinear form and when $\dbar$ is a sub-solution to the $P$-critical equation (with parameter $\zeta$), this product is definite positive. Moreover, it induces a norm which is equivalent to the usual $L^2$ norms because the positivity condition on $\mP_\zeta'$ is pointwise. Similarly, we define the scalar product on $\Omega^{0,1}(X,\End(E))$ as
$$
\scal{\alpha}{\beta}_{L^2(\dbar,\zeta)} = \overline{\scal{\alpha^\dagger}{\beta^\dagger}_{L^2(\dbar,\zeta)}} = \overline{\int_X \tr([\mP_\zeta'(\dbar) \wedge \i\alpha \wedge \beta^\dagger]_\sym)} = \int_X \tr([\mP_\zeta'(\dbar) \wedge \i\beta^\dagger \wedge \alpha]_\sym).
$$
Let $\nabla = \partial + \dbar$ be the Chern connection associated with $(\dbar,h)$. Let us introduce operators,
$$
\dbar_\zeta^* : \fonction{\Omega^{0,1}(X,\End(E))}{\Omega^0(X,\End(E))}{\alpha}{-\frac{1}{\Vol}\i[\mP_\zeta'(\dbar) \wedge \partial\alpha]_\sym},
$$
$$
\partial_\zeta^* : \fonction{\Omega^{1,0}(X,\End(E))}{\Omega^0(X,\End(E))}{\alpha}{\frac{1}{\Vol}\i[\mP_\zeta'(\dbar) \wedge \dbar\alpha]_\sym},
$$
$$
\nabla_\zeta^* = \partial_\zeta^* + \dbar_\zeta^* : \Omega^1(X,\End(E)) \rightarrow \Omega^0(X,\End(E)).
$$
The following lemma justifies the notation.

\begin{lemma}\label{LEM:Adjoints d et dbar}
    For all $\alpha \in \Omega^{1,0}(X,\End(E))$ (resp. $\Omega^{0,1}(X,\End(E))$, $\Omega^1(X,\End(E))$) and all $f \in \Omega^0(X,\End(E))$,
    $$
    \scal{\partial f}{\alpha}_{L^2(\dbar,\zeta)} = \scal{f}{\partial_\zeta^*\alpha}_{L^2} \quad (\textrm{resp. } \scal{\dbar f}{\alpha}_{L^2(\dbar,\zeta)} = \scal{f}{\dbar_\zeta^*\alpha}_{L^2}, \quad \scal{\nabla f}{\alpha}_{L^2(\dbar,\zeta)} = \scal{f}{\nabla_\zeta^*\alpha}_{L^2}).
    $$
\end{lemma}
\begin{proof}
The proof for the second case is the same as the one for the first case, and the third case is simply the sum of the first two so we only give a proof of the first equality. Let $\alpha \in \Omega^{1,0}(X,\End(E))$ and $f \in \Omega^0(X,\End(E))$.

{\allowdisplaybreaks
\begin{align*}
    \scal{\partial f}{\alpha}_{L^2(\dbar,\zeta)} & = \int_X \tr([\mP_\zeta'(\dbar) \wedge \i\partial f \wedge \alpha^\dagger]_\sym)\\
    & = \sum_{k = 1}^n \int_X k\zeta_k \wedge \tr([\hat{F}(\dbar),\ldots,\hat{F}(\dbar),\i\partial f,\alpha^\dagger]_\sym)\\
    & = -\sum_{k = 1}^n \int_X k\zeta_k \wedge \tr([\hat{F}(\dbar),\ldots,\hat{F}(\dbar),\i f,\partial\alpha^\dagger]_\sym)\\
    & \textrm{by integration by parts and Bianchi's identity,}\\
    & = -\sum_{k = 1}^n \int_X k\zeta_k \wedge \tr(f[\hat{F}(\dbar),\ldots,\hat{F}(\dbar),\i\partial\alpha^\dagger]_\sym)\\
    & \textrm{by commutativity inside the trace,}\\
    & = \int_X \tr(f(\i[\mP_\zeta'(\dbar) \wedge \dbar\alpha]_\sym)^\dagger)\\
    & = \int_X \tr(f(\partial_\zeta^*\alpha)^\dagger) \, \Vol\\
    & = \scal{f}{\partial_\zeta^*\alpha}_{L^2}.
\end{align*}}
\end{proof}

Then, we define symmetric order $2$ differential operators on smooth sections of endomorphisms of $E$,
$$
\Delta_{\zeta,\partial} = \partial_\zeta^*\partial, \qquad \Delta_{\zeta,\dbar} = \dbar_\zeta^*\dbar, \qquad \Delta_{\zeta,\nabla} = \nabla_\zeta^*\nabla = \Delta_{\zeta,\partial} + \Delta_{\zeta,\dbar}.
$$
Notice that since $\partial f^\dagger = (\dbar f)^\dagger$ on sections, we have $\partial_\zeta^*\alpha^\dagger = (\dbar_\zeta^*\alpha)^\dagger$ on $(0,1)$-forms so $\Delta_{\zeta,\partial}f^\dagger = (\Delta_{\zeta,\dbar}f)^\dagger$ on sections and similarly, $\Delta_{\zeta,\dbar}f^\dagger = (\Delta_{\zeta,\partial}f)^\dagger$ thus $\Delta_{\zeta,\nabla}f^\dagger = (\Delta_{\zeta,\nabla}f)^\dagger$ meaning this third Laplacian is a real operator.

Moreover, $\Delta_{\zeta,\nabla}$ is the linearisation of the $P$-critical equation at $\dbar$ with respect to Hermitian gauge transformations.

\begin{proposition}[{\cite[Lemma 2.36]{DMS}}]\label{PRO:Linéarisation Z-critique}
    For all Dolbeault operator $\dbar$, if $(d - 1)p > n$, the maps
    $$
    \fonction{L_{d + 1}^p(\G^\C(E))}{L_{d - 1}^p(\Omega^0(X,\End(E)))}{s}{\frac{1}{\Vol}\mP_\zeta(\e^s \cdot \dbar)},
    $$
    are smooth. Similarly with the $\mC^d$ and the $\mC^{d,\alpha}$ topologies. For all $v \in \Lie(\G^\C(E)) = \Omega^0(X,\End(E))$,
    $$
    \frac{d}{ds}_{|s = 0} \frac{1}{\Vol}\mP_\zeta(\e^s \cdot \dbar)v = \frac{1}{2\pi}\Delta_{\zeta,\nabla}\Re(v) + \i\left[\Im(v),\frac{1}{\Vol}\mP_\zeta(\dbar)\right].
    $$
    The real and imaginary part of $v$ has to be understood as the Hermitian and the skew-Hermitian part (divided by $\i$) of $v$. Moreover, if $\dbar$ is a sub-solution, $\Delta_{\zeta,\nabla}$ is elliptic.
\end{proposition}
\begin{proof}
The map
$$
(s,L_{d + 1}^p(\G^\C(E))) \mapsto (\hat{F}(\e^s \cdot \dbar),L_{d - 1}^p(\G^\C(E))),
$$
is smooth. Therefore, $s \mapsto \frac{1}{\Vol}\mP_\zeta(\e^s \cdot \dbar)$ is smooth too with the same topologies because $L_{d - 1}^p$ spaces are stable by products when $(d - 1)p > n$ \cite[Theorem 4.12]{Adams_Fournier}. With the same method as in \cite[Section 2.3.1]{DMS}, the formula for the derivative at $0$ is true when $v$ is Hermitian. When $v$ is skew-Hermitian, $\e^v$ is unitary thus $\hat{F}(\e^v \cdot \dbar) = \e^v\hat{F}(\dbar)\e^{-v}$ hence $\mP_\zeta(\e^v \cdot \dbar) = \e^v\mP_\zeta(\dbar)\e^{-v}$. The formula for the derivative at $v$ easily follows. The proofs for the $\mC^d$ and the $\mC^{d,\alpha}$ topologies are similar.

The proof of the ellipticity of $\Delta_{\zeta,\nabla}$ is the same as in \cite[Lemma 2.36]{DMS}, \cite[Lemma 3.3]{Takahashi}.
\end{proof}

Finally, we want these operators to satisfy the same kind of regularity as Laplacians. Let
$$
\Omega^{*,*}(X,\R) = \bigoplus_{k = 0}^n \Omega^{k,k}(X,\C) \cap \Omega^{2k}(X,\R),
$$
be the space of real smooth forms of degrees $(k,k)$. As before, we may complete it with $L_d^p$, $\mC^d$ or $\mC^{d,\alpha}$ topologies.

\begin{proposition}\label{PRO:Régularité Z laplaciens}
    The operators $\Delta_{\zeta,\partial}$, $\Delta_{\zeta,\dbar}$ and $\Delta_{\zeta,\nabla}$ extend continuously as
    $$
    \Delta_{\zeta,\partial}, \Delta_{\zeta,\dbar}, \Delta_{\zeta,\nabla} : L_{d + 1}^p(\Omega^0(X,\End(E))) \rightarrow L_{d - 1}^p(\Omega^0(X,\End(E))),
    $$
    as long as $(d - 1)p > 2n$, $\zeta$ is $L_{d - 1}^p$ and $\dbar$ is $L_d^p$. Moreover, the maps
    \begin{adjustwidth}{-2cm}{-2cm}
    \begin{equation}\label{EQ:Régularité en dbar du laplacien}
        \fonction{L_{d - 1}^p(\Omega^{*,*}(X,\R)) \times (\dbar_0 + L_d^p(\Omega^{0,1}(X,\End(E))))}{\Hom(L_{d + 1}^p(\Omega^0(X,\End(E))),L_{d - 1}^p(\Omega^0(X,\End(E))))}{(\zeta,\dbar)}{\Delta_{\zeta,\partial}, \Delta_{\zeta,\dbar}, \Delta_{\zeta,\nabla}}
    \end{equation}
    \end{adjustwidth}
    are smooth. Similarly, we can extend them continuously as
    $$
    \Delta_{\zeta,\partial}, \Delta_{\zeta,\dbar}, \Delta_{\zeta,\nabla} : \mC^{d + 1}(\Omega^0(X,\End(E))) \rightarrow \mC^{d - 1}(\Omega^0(X,\End(E))),
    $$
    and
    $$
    \Delta_{\zeta,\partial}, \Delta_{\zeta,\dbar}, \Delta_{\zeta,\nabla} : \mC^{d + 1,\alpha}(\Omega^0(X,\End(E))) \rightarrow \mC^{d - 1,\alpha}(\Omega^0(X,\End(E))).
    $$
    as long as $d \geq 1$, $0 < \alpha \leq 1$, $\zeta$ is $\mC^{d - 1}$ (resp. $\mC^{d - 1,\alpha}$) and $\dbar$ is $\mC^d$ (resp. $\mC^{d,\alpha}$) and we have the same smooth maps as in (\ref{EQ:Régularité en dbar du laplacien}). Moreover, if $\zeta$ is closed and $\dbar$ is a sub-solution, these operators induce isomorphisms from their coimage toward their image with the same topologies as above.
\end{proposition}
\begin{proof}
We have nine cases to verify but the proofs are exactly the same so we only prove the proposition for
$$
\Delta_{\zeta,\partial} : L_{d + 1}^p(\Omega^0(X,\End(E))) \rightarrow L_{d - 1}^p(\Omega^0(X,\End(E))).
$$
First of all, let $\dbar_0$ be a smooth Dolbeault operator on $E$ and $\mathcal{D} = \dbar_0 + \Omega^{0,1}(X,\End(E))$ the affine space of Dolbeault operators, and let us introduce
$$
q : \fonction{L_d^p(\mathcal{D})}{\Hom(L_{d + 1}^p(\Omega^0(X,\End(E))),L_{d - 1}^p(\Omega^{1,1}(X,\End(E))))}{\dbar = \dbar_0 + \alpha}{\dbar\partial = (\dbar_0 + [\alpha,\cdot])(\partial_0 - [\alpha^\dagger, \cdot])}
$$
We must be careful that the $\dbar$ on the left is a Dolbeault operator on $E$ and the $\partial$ and the $\dbar$ on the right are operators on $\End(E)$ hence the appearance of the Lie bracket with $\alpha$.

Recall that the condition $dp > 2n$ implies that the $L_d^p$ spaces are stable by wedge product \cite[Theorem 4.12]{Adams_Fournier}. For all $f \in L_{d + 1}^p(\Omega^0(X,\End(E)))$ and all $\dbar = \dbar_0 + \alpha \in L_d^p(\mX)$,
$$
q(\dbar)f = \dbar_0\partial_0f + [\alpha,\partial_0f] - \dbar_0[\alpha^\dagger,f] - [\alpha,[\alpha^\dagger,f]].
$$
We easily verify that this $(1,1)$-form is indeed $L_{d - 1}^p$ so $q$ is well defined. Moreover, $q$ is quadratic in $\alpha$ so it is smooth if and only if it is continuous at $\dbar_0$ and
\begin{align*}
    \norme{q(\dbar_0 + \alpha) - q(\dbar_0)}_{\Hom(L_{d + 1}^p,L_{d - 1}^p)} & = \sup_{\norme{f}_{L_{d + 1}^p} = 1}\norme{q(\dbar_0 + \alpha)f - q(\dbar_0)f}_{L_{d - 1}^p}\\
    & = \sup_{\norme{f}_{L_{d + 1}^p} = 1}\norme{[\alpha,\partial_0f] - \dbar_0[\alpha^\dagger,f] - [\alpha,[\alpha^\dagger,f]]}_{L_{d - 1}^p}\\
    & = \mathrm{O}\!\left(\sup_{\norme{f}_{L_{d + 1}^p} = 1} \norme{[\alpha,\partial_0f]}_{L_d^p} + \norme{[\alpha^\dagger,f]}_{L_d^p} + \norme{[\alpha,[\alpha^\dagger,f]]}_{L_d^p}\right)\\
    & = \mathrm{O}\!\left(\sup_{\norme{f}_{L_{d + 1}^p} = 1} \norme{\alpha}_{L_d^p}\norme{f}_{L_{d + 1}^p} + \norme{\alpha}_{L_d^p}\norme{f}_{L_d^p} + \norme{\alpha}_{L_d^p}^2\norme{f}_{L_d^p}\right)\\
    & = \mathrm{O}\!\left(\norme{\alpha}_{L_d^p} + \norme{\alpha}_{L_d^p}^2\right).
\end{align*}

This proves continuity. Notice that we only used $dp > 2n$ for this part. Then,
$$
\dbar \in L_d^p(\mathcal{D}) \mapsto \hat{F}(\dbar) \in L_{d - 1}^p(\Omega^{1,1}(X,\End(E)))
$$
is smooth and $(d - 1)p > 2n$ so $L_{d - 1}^p$ spaces are stable by wedge products. Therefore,
\begin{adjustwidth}{-2cm}{-2cm}
$$
(\zeta,\dbar) \in L_{d - 1}^p(\Omega^{*,*}(X,\R)) \times L_d^p(\mathcal{D}) \mapsto \mP_\zeta'(\dbar) \in \Hom(L_{d - 1}^p(\Omega^{1,1}(X,\End(E))),L_{d - 1}^p(\Omega^{n,n}(X,\End(E))))
$$
\end{adjustwidth}
is smooth too. Up to a constant smooth $(n,n)$-form, the operator $\Delta_{\zeta,\partial}$ is the composition if this map with $q$ thus,
\begin{adjustwidth}{-2cm}{-2cm}
$$
(\zeta,\dbar) \in L_{d - 1}^p(\Omega^{*,*}(X,\R)) \times L_d^p(\mathcal{D}) \mapsto \Delta_{\zeta,\partial} \in \Hom(L_{d + 1}^p(\Omega^0(X,\End(E))),L_{d - 1}^p(\Omega^0(X,\End(E))))
$$
\end{adjustwidth}
is well-defined and smooth.

All we have left to prove is that for all fixed $(\zeta,\dbar)$, if $\zeta$ is closed and $\dbar \in L_d^p(\mX_\zeta)$,
$$
\Delta_{\zeta,\partial} : L_{d + 1}^p(\coim(\Delta_{\zeta,\partial})) \rightarrow L_{d - 1}^p(\im(\Delta_{\zeta,\partial}))
$$
is an isomorphism. By Lemma \ref{LEM:Adjoints d et dbar}, $\Delta_{\zeta,\partial}$ is symmetric and elliptic (in particular, Fredholm of index $0$). By Fredholm alternative, it is bijective from $L_{d + 1}^p(\coim(\Delta_{\zeta,\partial}))$ to $L_{d - 1}^p(\im(\Delta_{\zeta,\partial}))$, and by the open mapping theorem, it is an isomorphism.
\end{proof}

\subsection{Decomposition of $P$-critical bundles}

In this sub-section, we show that $P$-critical bundles are, as HYM bundles, decomposable as a direct sum of simple $P$-critical bundles which share no holomorphic morphisms.

\begin{lemma}\label{LEM:Aut réductif}
    If $\E = (E,\dbar,h)$ is a $P$-critical vector bundle, all its automorphisms are parallel (equivalently, automorphisms are stable by $f \mapsto f^\dagger$).
\end{lemma}
\begin{proof}
Let $f \in \Aut(\E)$ such that $\dbar f = 0$. By Proposition \ref{PRO:Linéarisation Z-critique}, we have
$$
\frac{d}{dt}_{|t = 0} \frac{1}{\Vol}\mP_\zeta(\e^{tf} \cdot \dbar) = \frac{1}{2\pi}\Delta_{\zeta,\nabla}\Re(f) + \i\left[\Im(f),\frac{1}{\Vol}\mP_\zeta(\dbar)\right] = \frac{1}{2\pi}\Delta_{\zeta,\nabla}\Re(f),
$$
where we used here the fact that $\mP_\zeta(\dbar) = 0$ by hypothesis. Moreover, $f$ is holomorphic thus for all $t$, $\e^{tf} \cdot \dbar = \dbar$. We deduce that $\frac{1}{2\pi}\Delta_{\zeta,\nabla}\Re(f) = 0$ as it is the derivative at $0$ of a constant function. Therefore, by Lemma \ref{LEM:Adjoints d et dbar},
$$
\norme{\nabla\Re(f)}_{L^2(\dbar,\zeta)} = \scal{\Delta_{\zeta,\nabla}\Re(f)}{\Re(f)}_{L^2} = 0
$$
and $\scal{\cdot}{\cdot}_{L^2(\dbar,\zeta)}$ is definite positive because $\dbar$ is a sub-solution so $\nabla\Re(f) = 0$. In particular, $\Re(f)$ is holomorphic. So are $\Im(f) = -\i(f - \Re(f))$ and $f^\dagger = \Re(f) - \i\Im(f)$.
\end{proof}

Lemma \ref{LEM:Aut réductif} can be seen as the analogue in GIT of the fact that the stabiliser of a polystable point is reductive. Indeed, it implies that the complex Lie group $\Aut(\E)$ of automorphisms of $\E$ is the complexification of the compact real Lie group $\Aut(\E,h)$ of unitary automorphisms with respect to $h$. In particular, $\Aut(\E)$ is reductive. It also has for consequence a partial result of uniqueness of solutions.

\begin{proposition}\label{PRO:caractérisation unicité}
    If $\dbar_1$ and $\dbar_2$ are solutions to the $P$-critical equation in the same $\G^\C(E)$-orbit, we have equivalence between,
    \begin{enumerate}
        \item $\dbar_1$ and $\dbar_2$ are in the same $\G(E,h)$-orbit.
        \item $\dbar_1 \oplus \dbar_2$ is a sub-solution to the $P$-critical equation on $F = (E,h) \oplus (E,h)$.
    \end{enumerate}
\end{proposition}
\begin{proof}\ \\
\noindent\framebox{$1 \Rightarrow 2$} Assume $\dbar_2 = u \cdot \dbar_1$ for some unitary transform $u$. Set $\hat{F}_1 = \hat{F}(E,\dbar_1)$ and $\hat{F}_2 = \hat{F}(E,\dbar_2)$. Then, $\hat{F}_2 = u\hat{F}_1u^\dagger$ so for all $v = \begin{pmatrix} v_1 & v_2 \\ v_3 & v_4 \end{pmatrix}$ in $T_x^{1,0}X \otimes \End(F_x)$ and all $i$, $j$ and $k$,
\begin{align*}
    & \tr(\hat{F}(F,\dbar_1 \oplus \dbar_2)^i \wedge \i v \wedge \hat{F}(F,\dbar_1 \oplus \dbar_2)^j \wedge v^\dagger \wedge \hat{F}(F,\dbar_1 \oplus \dbar_2)^k)\\
    =\ & \tr\left(\begin{pmatrix} \hat{F}_1^i & 0 \\ 0 & \hat{F}_2^i \end{pmatrix} \wedge \begin{pmatrix} \i v_1 & \i v_2 \\ \i v_3 & \i v_4 \end{pmatrix} \wedge \begin{pmatrix} \hat{F}_1^j & 0 \\ 0 & \hat{F}_2^j \end{pmatrix} \wedge \begin{pmatrix} v_1^\dagger & v_3^\dagger \\ v_2^\dagger & v_4^\dagger \end{pmatrix} \wedge \begin{pmatrix} \hat{F}_1^k & 0 \\ 0 & \hat{F}_2^k \end{pmatrix}\right)\\
    =\ & \tr(\hat{F}_1^i \wedge \i v_1 \wedge \hat{F}_1^j \wedge v_1^\dagger \wedge \hat{F}_1^k) + \tr(\hat{F}_1^i \wedge \i v_2 \wedge u\hat{F}_1^ju^\dagger \wedge v_2^\dagger \wedge \hat{F}_1^k)\\
    +\ & \tr(u\hat{F}_1^iu^\dagger \wedge \i v_3 \wedge \hat{F}_1^j \wedge v_3^\dagger \wedge u\hat{F}_1^ku^\dagger) + \tr(u\hat{F}_1^iu^\dagger \wedge \i v_4 \wedge u\hat{F}_1^ju^\dagger \wedge v_4^\dagger \wedge u\hat{F}_1^ku^\dagger)\\
    =\ & \tr(\hat{F}_1^i \wedge \i v_1 \wedge \hat{F}_1^j \wedge v_1^\dagger \wedge \hat{F}_1^k) + \tr(\hat{F}_1^i \wedge (\i v_2u) \wedge \hat{F}_1^j \wedge (v_2u)^\dagger \wedge \hat{F}_1^k)\\
    +\ & \tr(\hat{F}_1^i \wedge (\i u^\dagger v_3) \wedge \hat{F}_1^j \wedge (u^\dagger v_3)^\dagger \wedge \hat{F}_1^k) + \tr(\hat{F}_1^i \wedge (\i u^\dagger v_4u) \wedge \hat{F}_1^j \wedge (u^\dagger v_4u)^\dagger \wedge \hat{F}_1^k).
\end{align*}

We deduce that
\begin{align*}
    & [\mP_\zeta'(F,\dbar_1 \oplus \dbar_2,h \oplus h) \wedge \i v \wedge v^\dagger]_\sym\\
    =\ & [\mP_\zeta'(E,\dbar_1,h) \wedge \i v_1 \wedge v_1^\dagger]_\sym + [\mP_\zeta'(E,\dbar_1,h) \wedge \i (v_2u) \wedge (v_2u)^\dagger]_\sym\\
    +\ & [\mP_\zeta'(E,\dbar_1,h) \wedge \i (u^\dagger v_3) \wedge (u^\dagger v_3)^\dagger]_\sym + [\mP_\zeta'(E,\dbar_1,h) \wedge \i (u^\dagger v_4u) \wedge (u^\dagger v_4u)^\dagger]_\sym.
\end{align*}
The result immediately follows. Notice that we also have $\mP_\zeta(F,\dbar_1 \oplus \dbar_2,h \oplus h) = 0$ so $\dbar_1 \oplus \dbar_2$ is a solution to the $P$-critical equation.\\

\noindent\framebox{$2 \Rightarrow 1$} Let $\dbar_2 = f \cdot \dbar_1$ be both $P$-critical. Let $\E_1 = (E,\dbar_1,h)$, $\E_2 = (E,\dbar_2,h)$ and $\F = \E_1 \oplus \E_2$ endowed with the direct sum metric and complex structure. By assumption, $\dbar_1 \oplus \dbar_2$ is a sub-solutions and the equality $\mP_\zeta(F,\dbar_1 \oplus \dbar_2) = 0$ is trivial. It makes $\F$ a $P$-critical bundle. Let
$$
g =
\begin{pmatrix}
    0 & 0\\
    f & 0
\end{pmatrix}.
$$
Since $f : \E_1 \rightarrow \E_2$ is holomorphic, so is $g : \F \rightarrow \F$. By Lemma \ref{LEM:Aut réductif}, $g^\dagger$ is holomorphic and
$$
g^\dagger =
\begin{pmatrix}
    0 & f^\dagger\\
    0 & 0
\end{pmatrix},
$$
because the decomposition $\F = \E_1 \oplus \E_2$ is orthogonal. It implies that $f^\dagger : \E_2 \rightarrow \E_1$ is holomorphic. We deduce that $\sqrt{f^\dagger f} : \E_1 \rightarrow \E_1$ is holomorphic too thus $\dbar_2 = f \cdot \dbar_1 = f\sqrt{f^\dagger f}^{-1} \cdot \dbar_1$ and $f\sqrt{f^\dagger f}^{-1} : E \rightarrow E$ is unitary.
\end{proof}

The following proposition means that a $P$-critical bundle admits a decomposition as a direct sum of $P$-critical simple components. If two of these components are isomorphic, they are orthogonally isomorphic, else, there are no holomorphic morphisms from one to an other.

\begin{proposition}\label{PRO:Polysimplicité}
    Let $\E$ be $P$-critical. It admits an orthogonal and holomorphic decomposition
    $$
    \E = \bigoplus_{i = 1}^l \G_i^{m_i},
    $$
    where the $m_i \geq 1$ are integers and for all $i \neq j$, $\Hom(\G_i,\G_j) = 0$.
\end{proposition}
\begin{proof}\ \\
\noindent\textbf{Step 1 :} $\E$ decomposes as an orthogonal and holomorphic sum of simple sub-bundles.

Let us prove it by induction on the rank of $E$. It is trivial if $\rk(E) = 1$. If $E$ has rank $\rk(E) \geqslant 2$, it is also trivial if $\E$ is simple, else, $\End(\E)$ admits a holomorphic section $\xi$ which is not a homothety. Using Lemma \ref{LEM:Aut réductif}, up to replacing $\xi$ by its real part or its imaginary part, we may assume that $\xi$ is Hermitian and is not a homothety.

The coefficients of the characteristic polynomial of $\xi$ in function of $x \in X$ are holomorphic thus constant. We deduce that the eigenvalues of $\xi$ and their multiplicities are constant so each eigenspace $\ker(\lambda\Id_E - \xi)$ is a holomorphic sub-bundle of $\E$. Let $\Pi_\lambda \in \Omega^0(X,\End(E))$ be the orthogonal projection onto this space. By the spectral theorem, we have $\mathrm{Sp}(\xi) \subset \R$ and
\begin{equation}\label{EQ:Diaginalisation xi}
    \xi = \sum_{\lambda \in \mathrm{Sp}(\xi)} \lambda\Pi_\lambda.
\end{equation}
As each $\ker(\lambda\Id_E - \xi)$ is a holomorphic sub-bundle of $\E$, $\dbar\Pi_\lambda$ takes values in $\ker(\lambda\Id_E - \xi)$. In other words, we have $\Pi_\lambda\dbar\Pi_\lambda = \dbar\Pi_\lambda$ and for all $\nu \neq \lambda$, $\Pi_\nu\dbar\Pi_\lambda = 0$. When we differentiate (\ref{EQ:Diaginalisation xi}), we obtain
$$
0 = \sum_{\lambda \in \mathrm{Sp}(\xi)} \lambda\dbar\Pi_\lambda,
$$
so for all $\nu \in \mathrm{Sp}(\xi)$,
$$
0 = \sum_{\lambda \in \mathrm{Sp}(\xi)} \lambda\Pi_\nu\dbar\Pi_\lambda = \nu\dbar\Pi_\nu.
$$
Recall that $\xi$ is not a homothety so we can choose $\nu \neq 0$ and $\Pi_\nu \neq 0,\Id_E$. In this case, $\dbar\Pi_\nu = 0$ so $\Pi_\nu$ is holomorphic, meaning that we have an exact sequence which splits both orthogonally and holomorphically,
$$
\begin{tikzcd}
    0 \ar{r} & \ker(\nu\Id_E - \xi) \ar["\subset",shift left]{r} & \E \ar["\Pi_\nu",shift left]{l}\ar["\Id_E - \Pi_\nu",shift left]{r} & \ker(\nu\Id_E - \xi)^\bot \ar["\supset",shift left]{l}\ar{r} & 0
\end{tikzcd}.
$$
Therefore, $\E = \ker(\nu\Id_E - \xi) \oplus \ker(\nu\Id_E - \xi)^\bot$ and we conclude by induction. Let us write
$$
\E = \bigoplus_{i = 1}^l \G_i,
$$
where the $\G_i$ are the simple components of $\E$.\\

\noindent\textbf{Step 2 :} Each $\G_i$ is $P$-critical.

Let for all $i$, $G_i$ be the smooth structure of $\G_i$. In the decomposition $\E = \bigoplus_{i = 1}^l \G_i$, we have
$$
\hat{F}(E,\dbar) = \begin{pmatrix}
    \hat{F}(G_1,\dbar) & 0 & \hdots & 0\\
    0 & \hat{F}(G_2,\dbar) & \hdots & 0\\
    \vdots & \vdots & \ddots & \vdots \\
    0 & 0 & \hdots & \hat{F}(G_l,\dbar)
\end{pmatrix},
$$
because the decomposition is both orthogonal and holomorphic. Therefore if $x \in X$ and $v \in T_x^{1,0}X \otimes \End(G_{i,x})\backslash\{0\}$, let $\tilde{v}$ be its trivial extension to the whole $\End(E)$. Since the decomposition of $\E$ into the sum of the $\G_i$ is orthogonal, $\widetilde{(v^\dagger)} = (\tilde{v})^\dagger$ and
$$
0 < \tr([\mP_\zeta'(E,\dbar,h) \wedge \i \tilde{v} \wedge \tilde{v}^\dagger]_\sym) = \tr([\mP_\zeta'(G_i,\dbar,h) \wedge \i v \wedge v^\dagger]_\sym),
$$
so $\dbar$ (restricted to $G_i$) is a sub-solution of the $P$-critical equation on $(G_i,h)$ and
$$
0 = \mP_\zeta(E,\dbar,h) = \begin{pmatrix}
    \mP_\zeta(G_1,\dbar,h) & 0 & \hdots & 0\\
    0 & \mP_\zeta(G_2,\dbar,h) & \hdots & 0\\
    \vdots & \vdots & \ddots & \vdots \\
    0 & 0 & \hdots & \mP_\zeta(G_l,\dbar,h)
\end{pmatrix}.
$$
We deduce that for all $i$, $\mP_\zeta(G_i,\dbar,h) = 0$ thus $\G_i$ is $P$-critical.\\

\noindent\textbf{Step 3 :} If $\Hom(\G_i,\G_j) \neq 0$, $\G_i$ and $\G_j$ are orthogonally isomorphic.

Let $1 \leq i,j \leq l$. Assume there is a non-zero holomorphic $f : \G_i \rightarrow \G_j$. Using the natural inclusion $\Hom(\G_i,\G_j) \rightarrow \End(\E)$ and Lemma \ref{LEM:Aut réductif}, $f^\dagger : \G_j \rightarrow \G_i$ is holomorphic. We deduce that $f^\dagger f : \G_i \rightarrow \G_i$ is holomorphic. Since $f \neq 0$, $f^\dagger f \neq 0$ and by simplicity of $\G_i$, $f^\dagger f = a\Id_{\G_i}$ for some $a \in \C^*$.

Moreover, $f^\dagger f$ is Hermitian positive so $a \in \R_+^*$. We deduce that $\frac{f}{\sqrt{a}} : \G_i \rightarrow \G_j$ is a unitary isomorphism.\\

\noindent\textbf{Step 4 :} Conclusion.

Let $1 \leq i,j \leq m$. If $\G_i$ is isomorphic to $\G_j$, they are orthogonally isomorphic by Step 3. Therefore, $\G_i \oplus \G_j = \G_i^2$ orthogonally and holomorphically. Else, $\Hom(\G_i,\G_j) = \Hom(\G_j,\G_i) = 0$, still by Step 3. We conclude by regrouping together the pairwise isomorphic simple components of $\E$.
\end{proof}

\begin{remark}
    By Lemma \ref{LEM:P(E) = 0}, it automatically implies that for all $i$, $P_\zeta(G_i) = 0$.
\end{remark}

\section{Examples}\label{SEC:Exemples}

\subsection{The $Z$-critical equation}

Assume $X$ is Kähler and let $\omega$ be a Kähler form on it. Let $U \in H^{*,*}(X,\R)$ with $U_0 = 1$ and $\rho = (\rho_0,\ldots,\rho_n)$ be a vector of complex numbers. Let for all $i$, $u_i \in U_i$ a smooth closed form. Then, we can introduce following operators \cite[Definition 2.22]{DMS},
$$
\mathcal{Z}(E,\dbar,h) = \left(\sum_{d = 0}^n \rho_d\omega^d \wedge u \wedge \e^{\hat{F}(E,\dbar,h)}\right)_{(n,n)},
$$
$$
Z(E) = \int_X \tr(\mathcal{Z}(E,\dbar,h)) = \left(\sum_{d = 0}^n \rho_d[\omega]^d \wedge U \wedge \ch(E)\right)_{(n,n)}.
$$
Here, $\alpha_{(n,n)}$ denotes the component of bidegree $(n,n)$ of a form $\alpha$. We say that $Z$ is the associated \textit{central charge}. Assume that $E$ verifies $Z(E) \neq 0$ and write $Z(E) = \abs{Z(E)}\e^{\i\varphi(E)}$. The associated $P$-critical operator is
$$
\mP_\zeta(E,\dbar,h) = \Im(\e^{-\i\varphi(E)}\mathcal{Z}(E,\dbar,h)).
$$
The equation $\mP_\zeta(E,\dbar,h) = 0$ is the \textit{$Z$-critical equation}. By definition of $\varphi(E)$, we have $P_\zeta(E) = \Im(\e^{-\i\varphi(E)}Z(E)) = 0$ and we can compute that the closed real forms $\zeta_k$ are given by
\begin{equation}\label{EQ:Expression zeta}
    \zeta_k = \frac{1}{\abs{Z(E)}k!}\sum_{i + j \leq n}\sum_{l = 0}^{n - k} \Im(\overline{\rho_i}\rho_l)\left(\int_X [\omega]^i \cup U_j \cup \ch_{n - i - j}(E)\right)\omega^l \wedge u_{n - k - l}.
\end{equation}
The name "$P$-critical equation" is inspired from "$Z$-critical equation". All of the other examples below can be obtained as $Z$-critical equations whose parameters $(\rho,\omega,u)$ don't depend on the particular bundle $E$ we are working on. However, we see in (\ref{EQ:Expression zeta}) that the parameter $\zeta$ of the associated $P$-critical equation depends on the topology of $E$, which comes from the necessary condition $P_\zeta(E) = 0$.

The only known results about $Z$-critical equations on a bundle of rank $2$ or more and outside of asymptotic regimes come from Keller and Scarpa on surfaces \cite[Theorem 1.1]{Keller_Scarpa}. We can notice that any polynomial equation in the curvature can be recovered from a central charge. In particular, Keller and Scarpa's results apply to any $P$-critical equation. They show that the existence of a sub-solution implies a form of stability of the bundle with respect to sub-varieties of dimension $1$ \cite[Definition 1.2]{Keller_Scarpa}. If moreover, $E$ is indecomposable of rank $2$, they show that the existence of a solution implies a form of stability with respect to sub-bundles of rank $1$. They conjecture that the converse holds \cite[Conjecture 1.4]{Keller_Scarpa}.

Notice that under the condition $\Im\left(\frac{\rho_{n - 1}}{\rho_n}\right) > 0$, if we replace $\omega$ by $\epsilon^{-1}\omega$ with $\epsilon > 0$ small, the equation we obtain converges toward the HYM equation. This specific asymptotic regime know as the large volume limit, was studied by Dervan, McCarthy and Sektnan \cite{DMS}.

\subsection{The Hermitian Yang--Mills equation}

Let $g$ be a balanced metric on $X$ \textit{i.e.} the associated $(1,1)$-form $\omega$ verifies $d\omega^{n - 1} = 0$. In this case, we can set
$$
\mP_\zeta(E,\dbar,h) = \omega^{n - 1} \wedge \hat{F}(E,\dbar,h) - \frac{[\omega^{n - 1}] \cup c_1(E)}{\rk(E)[\omega^n]}\omega^n\Id_E.
$$
We immediately see that $P_\zeta(E) = 0$. Moreover, if $x \in X$ and $v \in T_x^{1,0}X \otimes \End(E_x)\backslash\{0\}$,
$$
\tr([\mP_\zeta'(E,\dbar,h) \wedge \i v \wedge v^\dagger]_\sym) = \omega^{n - 1} \wedge \tr(\i v \wedge v^\dagger) > 0,
$$
so the sub-solution condition is always verified. $\dbar$ is a solution to the $P$-critical equation if and only if it is HYM.

In this case, solutions are unique modulo the unitary gauge group and exist if and only if the bundle is slope polystable, as shown by Donaldson, Uhlenbeck and Yau \cite{Donaldson,Uhlenbeck_Yau,Lübke_Teleman}, a famous result known as the Kobayashi--Hitchin correspondence. Theorem \ref{THE:2} can be seen as a local version of this correspondence.

In the HYM case, Theorem \ref{THE:Cas zeta = zeta_0} (\textit{i.e.} when the equation does not vary), was already shown by Buchdahl and Schumacher in \cite{Buchdahl_Schumacher}. Notice however that their results involve a coarser topology than here, because they don't need to worry about working with sub-solutions since all $\dbar$ are sub-solutions to the HYM equation.

If we consider variations of the balanced metric $\omega^{n - 1}$, which can be translated as a variation of the associated form $\zeta$, the continuity result given by Theorem \ref{THE:4} generalises \cite{Delloque_2024}.

\subsection{The complex Monge--Ampère equation}

Let $\eta \in n!\ch_n(E)$ be a $(n,n)$-form on $X$ and
$$
\mP_\zeta(E,\dbar,h) = \hat{F}(E,\dbar,h)^n - \eta.
$$
In this case, $P_\zeta(E) = 0$ because $[\eta] = n!\ch_n(E)$. For all $x \in X$ and $v \in T_x^{1,0}X \otimes \End(E_x)$, if we set $\hat{F} = \hat{F}(E,\dbar,h)$ for some $\dbar$ and $h$ on $E$,
$$
\tr([\mP_\zeta'(E,\dbar,h)) \wedge \i v \wedge v^\dagger]_\sym) = \sum_{k = 1}^n \tr\left(\hat{F}(E,\dbar,h)^{k - 1} \wedge \i v \wedge \hat{F}(E,\dbar,h)^{n - 1 - k} \wedge v^\dagger \right).
$$
This equation is the complex Monge--Ampère equation, originally on line bundles. On a line bundle $L$, the sub-solution condition is equivalent to $\hat{F}(L,h,\dbar)$ being a Kähler form (thus $L$ must be ample). The condition of existence is simply the ampleness of $L$ (or more generally, the ampleness of the $(1,1)$-class where we look for a solution). It is the Calabi conjecture, solved by Yau \cite{Yau}. See \cite{Collins_Shi} for a survey.

It was later generalised by Datar and Pingali in \cite{Datar_Pingali} where the sub-solution condition is called the "cone condition". They show an equivalence between the existence of a solution, the existence of a sub-solution and a numerical condition \cite[Theorem 1.1]{Datar_Pingali}.

In the case of bundles of rank $2$ or more, The Monge--Ampère equation was introduced by Pingali \cite{Pingali}. He calls sub-solutions MA-positively curved \cite[page 2]{Pingali}. The positivity properties of vector bundles which satisfy the Monge--Ampère equation are studied by Ballal and Pingali in \cite{Ballal_Pingali}.

\subsection{The $J$-equation}

Assume $X$ is Kähler and set $\omega$ a Kähler form on $X$. Assume also that $n \geq 2$ and $E$ is a vector bundle such that $\ch_n(E) > 0$ and $[\omega] \cup \ch_{n - 1}(E) > 0$ and let
$$
\mP_\zeta(E,\dbar,h) = \frac{[\omega] \cup \ch_{n - 1}(E)}{n\ch_n(E)}\hat{F}(E,\dbar,h)^n - \omega \wedge \hat{F}(E,\dbar,h)^{n - 1}.
$$
In this case, $P_\zeta(E) = 0$ once again and for all $x \in X$ and $v \in T_x^{1,0}X \otimes \End(E_x)$, if we set $\hat{F} = \hat{F}(E,\dbar,h)$ for some $\dbar$ and $h$ on $E$,
\begin{align*}
    & \tr([\mP_\zeta'(E,\dbar,h)) \wedge \i v \wedge v^\dagger]_\sym)\\
    =\ & \tr\left(\frac{[\omega] \cup \ch_{n - 1}(E)}{n\ch_n(E)}\sum_{k = 0}^{n - 1} \hat{F}^k \wedge \i v \wedge \hat{F}^{n - 1 - k} \wedge v^\dagger - \omega \wedge \sum_{k = 0}^{n - 2} \hat{F}^k \wedge \i v \wedge \hat{F}^{n - 2 - k} \wedge v^\dagger\right).
\end{align*}
This equation is the $J$-equation. It was originally proposed by Donaldson \cite{Donaldson_J_eq} in the case of line bundles. G. Chen proved a Nakai--Moishezon like criterion \cite[Theorem 1.1]{Chen} which was later improved by Song \cite[Corollary 1.2]{Song}.

The $J$-equation was extended in the case of bundles of any rank by Takahashi \cite{Takahashi}. He calls sub-solutions $J$-positive \cite[Definition 3.1]{Takahashi}.

\subsection{The deformed Hermitian Yang--Mills equation}

Assume once again that $X$ is Kähler and set $\omega$ a Kähler form on $X$. Let
$$
\mP_\zeta(E,\dbar,h) = \Im(\e^{-\i\theta}(\omega\Id_E + \i\hat{F}(E,\dbar,h))^n).
$$
It is the dHYM equation. The constant $\theta \in \R$ must be the argument of
\begin{align*}
    Z_{\mathrm{dHYM}}(E) & = \int_X \tr((\omega\Id_E + \i\hat{F}(E,\dbar,h))^n)\\
    & = n!\int_X \i^n\tr(\e^{-\i\omega\Id_E + \hat{F}(E,\dbar,h)})\\
    & = n!\left(\i^n\e^{-\i[\omega]} \cup \ch(E)\right)_{(n,n)},
\end{align*}
at least when this quantity does not vanish, so the equality $P_\zeta(E) = 0$ is verified.

Once more, the usual equation is on line bundles. It comes from SYZ mirror symmetry as the dual of the special Lagrangian equation on a Calabi--Yau manifold \cite{LYZ,MMMS}. It was generalised by Collins and Yau \cite[Paragraph 8.1]{Collins_Yau_2018} where it is expected to be the dual of the special Lagrangian of a multi-section \cite{MMMS,LYZ,Minasian_Tomasiello,Grange_Minasian}. Notice that they also require that
$$
\Re(\tr((\omega\Id_E + \i\hat{F}(E,\dbar,h))^n)) > 0.
$$

Even in the line bundle case, the algebraic condition of existence of a solution is still a conjecture, known as the Collins--Yau--Jacob conjecture \cite[Conjecture 1.4]{Collins_Yau_Jacob}, which was only verified in particular cases. For example, in the super-critical phase, it was solved by G. Chen \cite[Theorem 1.7]{Chen} and later improved by Chu, Lee and Takahashi \cite[Theorem 1.2]{CLT} and by Ballal in the projective case \cite[Corollary 1]{Ballal}. It does not hold outside of the super-critical phase \cite[Remark 1.10]{Chen} where it remains an open problem. See \cite{Collins_Shi} for a survey.

\section{Moment maps and Kuranishi slices}\label{SEC:Kuranishi}

We know from Lemma \ref{LEM:P(E) = 0} that a necessary topological condition for $E$ to admit a solution to the $P$-critical equation is that $P_\zeta(E) = 0$. Let us introduce
$$
\mZ = \{\zeta \in \Omega^{*,*}(X,\R)|d\zeta = 0, P_\zeta(E) = 0\}.
$$
Notice that once we assume $\zeta$ is closed, the equality $P_\zeta(E) = 0$ is linear in $[\zeta] \in H^{*,*}(X,\R)$. In particular, $\mZ$ is a real sub-space of $\Omega^{*,*}(X,\R)$ which contains $\zeta_0$.

Let us fix $\zeta_0 \in \mC^{d - 1}(\mZ)$, with $d$ a positive integer. Let $\E_0 = (E,h,\dbar_0)$ be a $P_{\zeta_0}$-critical bundle. Assume moreover that $\dbar_0^2 = 0$ so $\E_0$ is actually a holomorphic vector bundle. By Proposition \ref{PRO:Polysimplicité}, we can decompose it both orthogonally and holomorphically as
$$
\E_0 = \bigoplus_{i = 1}^l \G_{0,i}^{m_i},
$$
where the $\G_{0,i}$ are simple and pairwise non-isomorphic. Moreover, for all $i \neq j$, $\Hom(\G_{0,i},\G_{0,j}) = 0$. Let $G = \Aut_0(\E_0)$ be the group of automorphisms of $\E_0$ of determinant $1$ and $K = \Aut_0(\E_0,h)$ the compact sub-group of unitary automorphisms of determinant $1$. Let $\frak{g} = \Lie(G)$ and $\frak{k} = \Lie(K)$. $\frak{g}$ is the complex space of trace-free holomorphic endomorphisms of $\E_0$ and $\frak{k} \subset \frak{g}$ is the real space of trace-free skew-Hermitian holomorphic endomorphisms of $\E_0$. By Lemma \ref{LEM:Aut réductif}, $\frak{g}$ is the complexification of $\frak{k}$, thus $G$ is reductive with maximal compact sub-group $K$.

This space inherits from the $L^2$ Hermitian product on $\Omega^0(X,\End(E))$. We simply call $\scal{\cdot}{\cdot}$ the restriction. This Hermitian product is $K$-invariant. In particular, balls in $V$ are preserved by the action of $K$.

Since there is no non-zero morphisms between two distinct $\G_{0,i}$, we have a commutative diagram
$$
\begin{tikzcd}
    K \ar[leftrightarrow,"\rotatebox{90}{$\sim$}"]{d}\ar[hookrightarrow]{r} & G \ar[leftrightarrow,"\rotatebox{90}{$\sim$}"]{d}\\
    \prod_{i = 1}^l U(m_i) \cap \mathrm{SL}_{m_1 + \cdots + m_l}(\R) \ar[hookrightarrow]{r} & \prod_{i = 1}^l \GL_{m_i}(\C) \cap \mathrm{SL}_{m_1 + \cdots + m_l}(\C)
\end{tikzcd}
$$
where the morphisms are the obvious ones. We impose this condition of having a determinant $1$/being trace-free because homotheties act trivially on Dolbeault operators.

\subsection{Infinite dimensional moment map}

The purpose of this article is to give necessary and sufficient conditions for small deformations of $\E_0$ to have solutions to the $P$-critical equations in a neighbourhood of $\dbar_0$ by using the moment map interpretation of this equation. We also make the parameter $\zeta$ of the equation vary near $\zeta_0$. Recall that for all $\zeta$,
$$
\mX_\zeta = \{\dbar : \Omega^0(X,E) \rightarrow \Omega^{0,1}(X,E)|\dbar \textrm{ is a sub-solution to the $P_\zeta$-critical equation}\}.
$$
And recall that $L_d^p(\mX_\zeta)$ is the set of all $\dbar \in \dbar_0 + L_d^p(\Omega^{0,1}(X,\End(E)))$ which are sub-solutions to the $P_\zeta$-critical equation. Similarly with the $\mC^d$ and the $\mC^{d + \alpha}$ topologies. Since $\zeta$ varies near $\zeta_0$, we need first to understand how it affects sub-solutions.

\begin{lemma}\label{LEM:X_zeta ouverts}
    If $(d - 1)p > 2n$, the inclusion
    $$
    \bigsqcup_{\zeta \in L_{d - 1}^p(\mZ)} (\zeta,L_d^p(\mX_\zeta)) \subset L_{d - 1}^p(\mZ) \times (\dbar_0 + L_d^p(\Omega^{0,1}(X,\End(E))))
    $$
    is open. Similarly, if $d \geq 1$, the inclusions
    $$
    \bigsqcup_{\zeta \in \mC^{d - 1}(\mZ)} (\zeta,\mC^d(\mX_\zeta)) \subset \mC^{d - 1}(\mZ) \times (\dbar_0 + \mC^d(\Omega^{0,1}(X,\End(E)))),
    $$
    $$
    \bigsqcup_{\zeta \in \mC^{d - 1,\alpha}(\mZ)} (\zeta,\mC^{d,\alpha}(\mX_\zeta)) \subset \mC^{d - 1,\alpha}(\mZ) \times (\dbar_0 + \mC^{d,\alpha}(\Omega^{0,1}(X,\End(E))))
    $$
    are open.
\end{lemma}
\begin{proof}
For the first one, the inequality $(d - 1)p > 2n$ implies that the wedge product
$$
\wedge : \Omega^{p,q}(X,\End(E)) \times \Omega^{p',q'}(X,\End(E)) \rightarrow \Omega^{p + p',q + q'}(X,\End(E))
$$
is continuous when these spaces are all endowed with the $L_{d - 1}^p$ topology. Moreover, $L_{d - 1}^p(\mZ) \subset \mC^0(\mZ)$ is continuous too. We deduce that $\mP_\zeta'(\dbar)$ being positive is an open condition on $\zeta \in L_d^p(\mZ)$ and $\dbar \in \dbar_0 + L_d^p(\Omega^{0,1}(X,\End(E)))$. We deduce the first result.

For the $\mC^d$ and the $\mC^{d,\alpha}$ topologies, the argument is the same.
\end{proof}

From now on, let us fix $d \geq 1$. We shall only work with the $\mC^d$ topologies for simplicity but all our results also hold for $L_d^p$ and $\mC^{d,\alpha}$ topologies as long as $(d - 1)p > 2n$ (resp. $d \geq 1$) in order to satisfy Proposition \ref{PRO:Régularité Z laplaciens} and Lemma \ref{LEM:X_zeta ouverts}. The reason why we choose this topology is that it is the coarsest for $d = 1$. Indeed, by Sobolev embedding theorems, $L_d^p \subset \mC^1$ if $(d - 1)p > 2n$ and clearly $\mC^{d,\alpha} \subset \mC^{1,\alpha} \subset \mC^1$.

By Lemma \ref{LEM:X_zeta ouverts} and by definition of the product topology, there is an open neighbourhood $\mU_d \subset \mC^{d - 1}(\mZ)$ of $\zeta_0$ and an open neighbourhood $\mX_d$ of $\dbar$ in $\dbar_0 + \mC^d(\Omega^{0,1}(X,\End(E)))$ such that
$$
\mU_d \times \mX_d \subset \bigsqcup_{\zeta \in \mC^{d - 1}(\mZ)} (\zeta,\mC^d(\mX_\zeta)).
$$
In other words, for all $(\zeta,\dbar) \in \mU_d \times \mX_d$, $\dbar$ is a sub-solution to the $P$-critical equation with parameter $\zeta$. From now on, we only consider $\zeta \in \mU_d$.

Let for all $\zeta$,
$$
\mu_{\infty,\zeta} : \fonction{\mX_d}{\mC^{d - 1}(\Lie(\G(E,h)))}{\dbar}{-\frac{2\i\pi}{\Vol}\mP_\zeta(\dbar)}.
$$
Notice that $\Lie(\G(E,h)) = \Omega^0(X,\i\End_H(E,h))$ where $\End_H(E,h)$ denotes the real vector bundle of Hermitian endomorphisms of $E$ (for the metric $h$), so $\i\End_H(E,h)$ is the bundle of skew-Hermitian endomorphisms of $E$.

By Propositions \ref{PRO:Linéarisation Z-critique} and \ref{PRO:Régularité Z laplaciens}, $\mu_{\infty,\zeta}$ is smooth on each germ of $\mC^{d + 1}(\G^\C(E))$-orbit and for all $\dbar \in \mX_d$,
$$
\frac{d}{ds}|_{s = 0}\mu_{\infty,\zeta}(\e^s \cdot \dbar) : v \mapsto -\i\Delta_{\zeta,\nabla}\Re(v) + 2\pi\left[\Im(v),\mP_\zeta(\dbar)\right].
$$
The problem of finding $P$-critical connections now reduces to finding zeroes of $\mu_{\infty,\zeta}$.

In Sub-section \ref{SEC:Produits scalaires}, we introduced $L^2$ Hermitian products on $\Omega^{0,1}(X,\End(E))$ depending on parameters $\zeta$ and sub-solutions $\dbar$. Let $\Omega_{\infty,\zeta}$ be the associated Kähler form,
$$
\Omega_{\infty,\zeta,\dbar}(\alpha,\beta) = \Im(\scal{\alpha}{\beta}_{L^2(\dbar,\zeta)}) = \int_X \Re\left(\tr([\mP_\zeta'(\dbar) \wedge \alpha \wedge \beta^\dagger]_\sym))\right).
$$
Using the same arguments as in the proof of Proposition \ref{PRO:Régularité Z laplaciens}, the map
$$
\fonction{\mU_d \times \mX_d}{\Hom(\mC^d(\Omega^{0,1}(X,\End(E))) \otimes \mC^d(\Omega^{0,1}(X,\End(E))),\mC^{d - 1}(\Omega^0(X,\End(E))))}{(\zeta,\dbar)}{\mP_\zeta'(\dbar)}
$$
is well-defined ans smooth. Moreover, it is $\mC^{d + 1}(\G(E,h))$-invariant. We deduce that $\Omega_{\infty,\zeta}$ extends as a Kähler form on $\mX_d$ and it varies smoothly with $\zeta \in \mC^{d - 1}(\mZ)$. However, it does not make it locally complete because the induced norm is an $L^2$ norm. The following Proposition comes from \cite[Theorem 2.45]{DMS}. See also \cite[Theorem 6.7]{Dervan_Hallam} for reference.

\begin{proposition}\label{PRO:Application moment dimension infinie}
    The (germ of) action of $\mC^{d + 1}(\G(E,h))$ on $(\mX_d,\Omega_{\infty,\zeta})$ is Hamiltonian with associated $\mC^{d + 1}(\G(E,h))$-equivariant moment map $\mu_{\infty,\zeta}$.
\end{proposition}

\begin{remark}
    We could have used Dervan--Hallam theory \cite[Theorem 6.7]{Dervan_Hallam} directly instead of constructing $\Omega_{\infty,\zeta}$ but it wouldn't have given us to continuity with respect to $\zeta$ with the involved topologies. We use here the method of \cite{DMS} to define $\Omega_{\infty,\zeta}$ thanks to the formal derivative of $\mP_\zeta$. It has the advantage of giving a more explicit definition of this Kähler form and to grant us the wanted continuity and smoothness results.
\end{remark}

\subsection{Kuranishi slices and finite dimensional moment maps}

An important tool in deformation theory is the Kuranishi slice, which is basically a way to parametrise all complex deformations of a complex object (here, a holomorphic vector bundle) by a finite dimensional analytic space. It can be found in various forms in literature \cite{Kuranishi}, \cite[Section 7.3]{Kobayashi}, \cite[Theorem 1]{Buchdahl_Schumacher}. We give here a small informal exposition.

The gauge group $\G^\C(E)$ acts on the space of $(0,1)$-forms with values in $\End(E)$ by conjugation,
$$
f \cdot \alpha = f\alpha f^{-1}.
$$
The closed positive $(n - 1,n - 1)$-form $\frac{\omega^{n - 1}}{(n - 1)!}$ on $X$ defines an $L^2$ norm on the space of $(0,1)$-forms with values in the endomorphisms of $E$, on which $\G(E,h)$ acts by unitary actions,
$$
\norme{\alpha}_{L^2} = \int_X \tr(\i\alpha^\dagger \wedge \alpha) \wedge \frac{\omega^{n - 1}}{(n - 1)!}.
$$
Let $\Lambda$ be the contraction with respect to $\omega$ on $(1,1)$-forms,
$$
\Lambda\alpha = \frac{\alpha \wedge \frac{\omega^{n - 1}}{(n - 1)!}}{\frac{\omega^n}{n!}}.
$$
For this norm, the adjoint of the $\dbar_0$ (still on $\End(E)$) operator is
$$
\dbar_0^\bullet = -\star\partial_0\star,
$$
where $\star$ is the Hodge star. We use the notation $\bullet$ instead of $*$ to avoid confusion with the adjoints of $\dbar_0$ with respect to the $\scal{\cdot}{\cdot}_{L^2(\dbar_0,\zeta)}$ introduced in Sub-section \ref{SEC:Produits scalaires}. Moreover, on the space of sections,
$$
\dbar_0^\bullet = -\i\Lambda\partial_0 \textrm{ on } \Omega^0(X,\End(E)).
$$
It is a well-known fact on Kähler manifolds and it remains true for balanced manifolds at degree $0$. See \cite[Lemma 2.1]{Delloque_2024} or Lemma \ref{LEM:Adjoints d et dbar} in the balanced HYM case.

Let $V = H^{0,1}(X,\End(\E))$ be the space of harmonic $(0,1)$-forms \textit{i.e.} the space of\\
$b \in \Omega^{0,1}(X,\End(E))$ such that $\dbar_0b = \dbar_0^\bullet b = 0$. It is preserved by the action of $G \subset \G^\C(E)$. We endow it with the induced $L^2$ norm so $G$ acts on $V$ and $K$ acts on $V$ by unitary actions.

Hodge theory tells us that $V$ is finite dimensional and we have an $L^2$ orthogonal decomposition \cite[Theorem 7.1]{Demailly},
\begin{equation}\label{EQ:Hodge}
    \Omega^{0,1}(X,\End(E)) = \lefteqn{\overbrace{\phantom{\im(\dbar_0) \oplus V}}^{= \ker(\dbar_0)}} \im(\dbar_0) \oplus \underbrace{V \oplus \im(\dbar_0^\bullet)}_{= \ker(\dbar_0^\bullet)}.
\end{equation}
When we apply a gauge transform of the form $\e^s$ to $\dbar_0$, we obtain $\e^s \cdot \dbar_0 = \dbar_0 + \e^s\dbar_0(\e^{-s})$, which is close to $\dbar_0 - \dbar_0s$ when $s$ is close to $0$. Therefore, the direction tangent to the gauge orbit of $\dbar_0$ at $\dbar_0$ is $\im(\dbar_0)$.

Consequently, the small deformations of $\dbar_0$ can be chosen in the orthogonal direction \textit{i.e.} $\ker(\dbar_0^\bullet)$. Moreover, $\dbar = \dbar_0 + \gamma$ is integrable if and only if $\dbar_0\gamma + \gamma \wedge \gamma = 0$. We see that in this case, if $\gamma$ is small, $\dbar_0\gamma = -\gamma \wedge \gamma$ is small with quadratic order in $\gamma$. We deduce that small integrable deformations of $\dbar_0$ can be chosen in the direction of $\ker(\dbar_0) \cap \ker(\dbar_0^\bullet) = V$, which is finite dimensional.

The \textit{Kuranishi slice} is the rigorous statement of this observation.

\begin{proposition}\label{PRO:Propriétés Kuranishi}
    There is and open ball $B \subset V$ centred at $0$ and a holomorphic embedding $\Phi : (B,0) \rightarrow (\mC^\infty(\Omega^{0,1}(X,\End(E))),0)$ such that,
    \begin{enumerate}
        \item $\Phi$ takes values in $\ker(\dbar_0^\bullet)$.
        \item $\Phi$ is $G$ invariant \textit{i.e.} for all $b \in B$ and all $g \in G$ such that $g \cdot b \in B$, $\Phi(g \cdot b) = g \cdot \Phi(b)$.
        \item For all $b$, $b = \Phi(b) + \dbar_0^\bullet\mathrm{Green}(\Phi(b) \wedge \Phi(b))$ where $\mathrm{Green}$ is the Green operator of the Laplace--Beltrami operator $\dbar_0\dbar_0^\bullet + \dbar_0^\bullet\dbar_0$. In particular, $d\Phi(0) : v \mapsto v$.
        \item For all $\dbar$ which is $\mC^d$ close enough to $\dbar_0$ and integrable, there is a unique $s \in \ker(\dbar_0)^\bot$ whose $\mC^{d + 1}$ norm is bounded by a multiple of the $\mC^d$ norm of $\dbar - \dbar_0$ such that $\e^s \cdot \dbar \in \Phi(B)$.
        \item The set of all $b \in B$ such that $\dbar_0 + \Phi(b)$ is integrable is an analytic subset of $B$ whose equation is given by $\Pi(\Phi(b) \wedge \Phi(b)) = 0$, where $\Pi : \Omega^{0,2}(X,\End(E)) \rightarrow H^{0,2}(X,\End(\E_0))$ is the $L^2$ orthogonal projection onto the space of harmonic $(0,2)$-forms.
        \item The map
        $$
        G \cdot b \mapsto \mC^{d + 1}(\G^\C(E)) \cdot \Phi(b)
        $$
        is a well-defined and continuous. Moreover, it induces a local homeomorphism from the space of $b$ modulo $G$ such that $\dbar_0 + \Phi(b)$ is integrable to the space of integrable $\mC^d$ Dolbeault operators modulo $\mC^{d + 1}(\G^\C(E))$.
    \end{enumerate}
\end{proposition}

Moreover, by openness of $\mX_d$, we may assume, up to shrinking $B$, that $\Phi(B) \subset \mX_d$ so the closed $(1,1)$-forms
$$
\Omega_\zeta = \Phi^*\Omega_{\infty,\zeta}
$$
are well-defined and positive on the whole $B$. They define Kähler metrics which vary smoothly with $\zeta$. From now on, when $b \in B$, we set
$$
\dbar_b = \dbar_0 + \Phi(b) \in \mX_d.
$$
Let $\Pi_{\frak{k}} : \Omega^0(X,\End(E)) \rightarrow \frak{k}$ be the $L^2$ orthogonal projection on $\frak{k}$ and
$$
\mu_\zeta : \fonction{B}{\frak{k}}{b}{\Pi_{\frak{k}}\mu_{\infty,\zeta}(\dbar_b)}.
$$
Then, $\mu$ is a smooth map between finite dimensional spaces by composition and it can also be interpreted as a moment map. It varies smoothly with $\zeta$ too.

\begin{proposition}
    The action of $K$ on $(B,\Omega_\zeta)$ is Hamiltonian with associated $K$-equivariant moment map $\mu_\zeta$.
\end{proposition}
\begin{proof}
It is a direct consequence of Proposition \ref{PRO:Application moment dimension infinie}. See also \cite[Theorem 6.8]{Dervan_Hallam}.
\end{proof}

This moment map is defined between finite dimensional vector spaces so we will be able to use classical GIT results \cite{DMS}. The fact that $B$ is non compact will be compensated by the fact that we are only looking for local results around $0$. Therefore, finding which $G$-orbits contain zeroes of $\mu_\zeta$ will be easier than finding which $\G^\C(E)$-orbits contain zeroes of $\mu_{\infty,\zeta}$. However, the zeroes $b$ of $\mu_\zeta$ only give operators $\dbar_b$ on $E$ such that $\mP_\zeta(\dbar_b) \in \i\frak{k}^\bot\Vol$, which is weaker than being $0$.

Nevertheless, our main theorem (Theorem \ref{THE:Déformation P-critique}) will ensure that finding zeroes of $\mu_\zeta$ is more or less equivalent to finding zeroes of $\mu_{\infty,\zeta}$ so the complicated infinite dimensional problem can be reduced to a more simple finite dimensional one.

In order to prove Theorem \ref{THE:Déformation P-critique}, we still need of course to be able to find zeroes of $\mu_{\infty,\zeta}$. For this, we introduce an other finite dimensional moment map whose zeroes will correspond to actual zeroes of $\mu_{\infty,\zeta}$. The following Proposition is analogous to \cite[Proposition 3.2]{Clarke_Tipler}, \cite[Proposition 4.1]{Delloque_2024}.

\begin{proposition}\label{PRO:Tranche déformée P}
    Up to shrinking $B$ and $\mU_d$, there exists a smooth map
    $$
    \sigma : \mU_d \times B \rightarrow \mC^{d + 1}(\Omega^0_0(X,\End(E))),
    $$
    that verifies $\sigma(\zeta_0,0) = 0$, $\frac{\partial}{\partial b}_{|b = 0}\sigma(\zeta_0,b) = 0$, all $\sigma(\zeta,b)$ are Hermitian and orthogonal to $\i\frak{k}$ and the following maps are well-defined and smooth.
    $$
    \tilde{\Phi} : \fonction{\mU_d \times B}{\mX_d - \dbar_0}{(\zeta,b)}{\e^{\sigma(\zeta,b)} \cdot \dbar_b - \dbar_0}, \qquad \tilde{\mu} : \fonction{\mU_d \times B}{\mC^{d - 1}(\Omega^0(X,\End(E)))}{(\zeta,b)}{\mu_{\infty,\zeta}(\dbar_0 + \tilde{\Phi}(b))}.
    $$
    Moreover, they verify,
    \begin{enumerate}
        \item For all $\zeta$ and $b$, $\dbar_{\zeta,b} = \dbar_0 + \tilde{\Phi}(\zeta,b)$ is gauge equivalent to $\dbar_b$.
        \item For all $\zeta$, $b \mapsto \dbar_{\zeta,b}$ is $K$-equivariant.
        \item For all $\zeta$ and $b$, $\tilde{\mu}_\zeta(b) \in \frak{k}$. Reciprocally, if $\dbar = \e^s \cdot \dbar_b$ verifies $\mu_{\infty,\zeta}(\dbar) \in \frak{k}$, $b \in B$ and $s \in \mC^{d + 1}(\i\frak{k}^\bot)$ is small enough, then $s = \sigma(\zeta,b)$, so $\dbar = \dbar_{\zeta,b}$.
        \item $\frac{\partial}{\partial b}_{|b = 0}\tilde{\Phi}(\zeta_0,b) : v \mapsto v$. In particular, up to shrinking $B$ and $\mU_d$, each $\tilde{\Phi}(\zeta,\cdot)$ is a smooth embedding.
    \end{enumerate}
\end{proposition}
\begin{proof}
Let $\Pi_{\frak{k}^\bot} : \Omega^0(X,\End_H(E,h)) \rightarrow \i\frak{k}^\bot$ be the orthogonal projection onto $\i\frak{k}^\bot$ with respect $\scal{\cdot}{\cdot}_{L^2}$. Notice that we write abusively $\i\frak{k}^\bot$ to denote Hermitian global sections of $E$ which are orthogonal to $\i\frak{k}$. Let
$$
\Psi : \fonction{\mU_d \times B \times \mC^{d + 1}(\i\frak{k}^\bot)}{\mC^{d - 1}(\i\frak{k}^\bot)}{(\zeta,b,s)}{\Pi_{\frak{k}^\bot}\i\mu_{\infty,\zeta}(\e^s \cdot \dbar_b)}.
$$
With the same arguments as in the proof of Proposition \ref{PRO:Régularité Z laplaciens}, we see that $\Psi$ is smooth and by Proposition \ref{PRO:Linéarisation Z-critique}, $\frac{\partial}{\partial s}_{|s = 0}\Psi(\zeta_0,0,s) = \Pi_{\frak{k}^\bot}\Delta_{\zeta_0,\nabla_0}$. $\Delta_{\zeta_0,\nabla_0} = \nabla_{0,\zeta_0}^*\nabla_0$ is symmetric positive so $\coker(\Delta_{\zeta_0,\nabla_0}) = \ker(\Delta_{\zeta_0,\nabla_0}) = \ker(\nabla_0) = \i\frak{k}$ by Lemma \ref{LEM:Aut réductif}. We deduce that $\frac{\partial}{\partial s}_{|s = 0}\Psi(\zeta_0,0,s) = \Delta_{\zeta_0,\nabla_0} : \mC^{d + 1}(\i\frak{k}^\bot) \rightarrow \mC^{d - 1}(\i\frak{k}^\bot)$ is an isomorphism by Proposition \ref{PRO:Régularité Z laplaciens}.

By the implicit functions theorem, up to shrinking $B$ and $\mU_d$, there is a unique smooth $\sigma : \mU_d \times B \rightarrow \mC^{d + 1}(\i\frak{k}^\bot)$ such that $\sigma(\zeta_0,0) = 0$ and for all $\zeta$ and $b$, $\Psi(\zeta,b,\sigma(\zeta,b)) = 0$. When we set $\tilde{\Phi}(\zeta,b) = \e^{\sigma(\zeta,b)} \cdot \dbar_b - \dbar_0$, the first point is immediate and the third point is verified. For the second one, notice that for all $u \in K$ and all $\zeta$ and $b$,
\begin{align*}
    \Psi(\zeta,b,u^\dagger\sigma(\zeta,u \cdot b)u) & = \Pi_{\frak{k}^\bot}\i\mu_{\infty,\zeta}(\e^{u^\dagger\sigma(\zeta,u \cdot b)u} \cdot \dbar_b)\\
    & = \Pi_{\frak{k}^\bot}\i\mu_{\infty,\zeta}(u^\dagger \cdot (\e^{\sigma(\zeta,u \cdot b)} \cdot \dbar_{u \cdot b}))\\
    & = u^\dagger\Pi_{\frak{k}^\bot}\i\mu_{\infty,\zeta}(\e^{\sigma(\zeta,u \cdot b)} \cdot \dbar_{u \cdot b})u\\
    & = u^\dagger\Psi(\zeta,u \cdot b,\sigma(\zeta,u \cdot b))u\\
    & = 0.
\end{align*}
By uniqueness of $\sigma$, it implies that $u^\dagger\sigma(\zeta,u \cdot b)u = \sigma(\zeta,b)$. Therefore,
$$
\dbar_{\zeta,u \cdot b} = \e^{\sigma(\zeta,u \cdot b)} \cdot \dbar_{u \cdot b} = u\e^{\sigma(\zeta,b)}u^\dagger u\dbar_bu^\dagger = u \cdot \dbar_{\zeta,b}.
$$
It proves the second point. For the fourth point, we have for all $b$ close to $0$,
\begin{align*}
    \Psi(\zeta_0,b,0) & = \Pi_{\frak{k}^\bot}\i\mu_{\infty,\zeta_0}(\dbar_0 + \Phi(b))\\
    & = \Pi_{\frak{k}^\bot}\i\nabla_{0,\zeta_0}^*(\Phi(b) - \Phi(b)^\dagger) + \mathrm{o}(\Phi(b))\\
    & = \Pi_{\frak{k}^\bot}\i\nabla_{0,\zeta_0}^*(d\Phi(0)b - (d\Phi(0)b)^\dagger) + \mathrm{o}(b).
\end{align*}
Therefore, $\frac{\partial}{\partial b}_{|b = 0}\Psi(\zeta_0,b,0)v = \Pi_{\frak{k}^\bot}\i\nabla_{0,\zeta_0}^*(d\Phi(0)v - (d\Phi(0)v)^\dagger)$. Now, recall that the image of $d\Phi(0) : v \mapsto v$ is the space $V$ of $(0,1)$-harmonic forms. In particular, it is included in the kernel of $\nabla_{0,\zeta_0}^*$. Similarly for $(d\Phi(0)\cdot)^\dagger$. Thus $\frac{\partial}{\partial b}_{|b = 0}\Psi(\zeta_0,b,0) = 0$. Therefore, when we differentiate the equality $\Psi(\zeta_0,b,\sigma(\zeta_0,b))$ at $b = 0$,
$$
0 = \frac{\partial}{\partial b}|_{b = 0}\Psi(\zeta_0,b,0) + \frac{\partial}{\partial s}|_{s = 0}\Psi(\zeta_0,0,s)\frac{\partial}{\partial b}_{|b = 0}\sigma(\zeta_0,b) = \frac{\partial}{\partial s}|_{s = 0}\Psi(\zeta_0,0,s)\frac{\partial}{\partial b}_{|b = 0}\sigma(\zeta_0,b).
$$
Since $\frac{\partial}{\partial s}|_{s = 0}\Psi(\zeta_0,0,s)$ is an isomorphism, we have $\frac{\partial}{\partial b}_{|b = 0}\sigma(\zeta_0,b) = 0$ hence
$$
\tilde{\Phi}(\zeta_0,b) = \e^{\sigma(\zeta_0,b)} \cdot \dbar_b - \dbar_0 = \e^{\sigma(\zeta_0,b)}\dbar_0(\e^{-\sigma(\zeta_0,b)}) + \e^{\sigma(\zeta_0,b)}\Phi(b)\e^{-\sigma(\zeta_0,b)} = \Phi(b) + \mathrm{o}(b) = b + \mathrm{o}(b).
$$
We deduce the fourth point and we showed along the way that $\frac{\partial}{\partial b}_{|b = 0}\sigma(\zeta_0,b) = 0$. Finally, all the $\sigma(\zeta_0,b)$ take values in $\i\frak{k}^\bot$ by definition of $\sigma$.
\end{proof}

Similarly as the $\mu_\zeta$, the $\tilde{\mu}_\zeta$ are moment maps on $B$. Let $\tilde{\Omega}_\zeta = \tilde{\Phi}(\zeta,\cdot)^*\Omega_{\infty,\zeta}$ which is a close symplectic form on $B$.

\begin{proposition}
    The action of $K$ on $(B,\tilde{\Omega}_\zeta)$ is Hamiltonian with associated $K$-equivariant moment map $\tilde{\mu}_\zeta$.
\end{proposition}
\begin{proof}
It is a direct consequence of Propositions \ref{PRO:Application moment dimension infinie} and \ref{PRO:Tranche déformée P}.
\end{proof}

\begin{remark}
    The fact that $\tilde{\Phi}(\zeta,\cdot)$ is only smooth (and not holomorphic) means that $\tilde{\Omega}_\zeta$ is not compatible with the complex structure of $B$. However, as a small deformation of $\Omega_{\zeta_0}$, the form
    $$
    (v,w) \mapsto -\tilde{\Omega}_\zeta(\i v,w)
    $$
    will be positive (see Lemma \ref{LEM:Omega tilde positive}). However, the incompatibility of $\tilde{\Phi}(\zeta,\cdot)$ with the complex structure of $B$ implies that the above bilinear form need not be symmetric.
\end{remark}

\begin{lemma}\label{LEM:Omega tilde positive}
    Up to shrinking $B$ and $\mU_d$, there is a positive constant $c$ such that for all $b \in B$ and $v \in T_bB = V$,
    $$
    -\tilde{\Omega}_{\zeta,b}(\i v,v) \geq c\norme{v}_{L^2}^2.
    $$
\end{lemma}
\begin{proof}
By Proposition \ref{PRO:Tranche déformée P}, we have $\frac{\partial}{\partial b}_{|b = 0}\tilde{\Phi}(\zeta_0,b) : v \mapsto v$. Thus, at $\zeta = \zeta_0$ and $b = 0$,
$$
-\tilde{\Omega}_{\zeta_0,0}(\i v,v) = -\tilde{\Omega}_{\infty,\zeta_0,0}(\i v,v) = \norme{v}_{L^2(\dbar_0,\zeta_0)}^2 \geq 2c\norme{v}_{L^2}^2
$$
for some $c > 0$. We deduce the wanted equality by continuity of $\tilde{\Omega}_{\zeta,b}$ in $\zeta$ and $b$.
\end{proof}

\subsection{Statement of the main theorem}

Let us define some different notions of stability, all related to stability in GIT, that we will show to be equivalent.
\begin{definition}[{\cite[Definition 7.1]{GRS}}]
    A point $b \in B$ is said to be \textit{$\mu_\zeta$-semi-stable} in $B$ if $\mu_\zeta$ vanishes on the closure $\overline{G \cdot b} \cap B$ of the orbit of $b$.
    
    It is said to be \textit{$\mu_\zeta$-polystable} in $B$ if $\mu_\zeta$ vanishes in the orbit $G \cdot b \cap B$ of $b$.
    
    It is said to be \textit{$\mu_\zeta$-stable} in $B$ if it is $\mu_\zeta$-polystable and the stabiliser $\Stab(b) \subset G$ of $b$ is trivial.
\end{definition}

\begin{remark}
    Usually, in GIT, a point is stable when it is polystable and has a \textit{discrete} stabiliser. Here, we will see that the stabiliser of a point corresponds to the group of holomorphic automorphisms of a vector bundle. These are always connected so in this context, discrete is equivalent to trivial.
\end{remark}

Similarly, we introduce $\mu_{\infty,\zeta}$-(semi-)(poly)stability of a vector bundle $\E_b = (E,\dbar_b)$ whose Dolbeault operator is in the Kuranishi slice.
\begin{definition}
    Let $\mB_d$ be an open neighbourhood of $\dbar_0$ in $\mX_d$ and $b \in B$.
    
    $\E_b$ is said to be \textit{$\mu_{\infty,\zeta}$-semi-stable} in $\mB_d$ if $\mu_{\infty,\zeta}$ vanishes on $\overline{\mC^{d + 1}(\G^\C(E)) \cdot \dbar_b} \cap \mB_d$.
    
    It is said to be \textit{$\mu_{\infty,\zeta}$-polystable} in $\mB_d$ if $\mu_{\infty,\zeta}$ vanishes on $\mC^{d + 1}(\G^\C(E)) \cdot \dbar_b \cap \mB_d$.
    
    It is said to be \textit{$\mu_{\infty,\zeta}$-stable} if it is $\mu_{\infty,\zeta}$-polystable and if $\E_b$ is simple.
\end{definition}

Finally, we introduce local $P$-stability of a bundle $\E_b = (E,\dbar_b)$ which basically says that $\E$ does not have sub-bundles which are deformations of direct sums of the simple components of $\E_0$.
\begin{definition}
    We say that a smooth sub-bundle $F$ of $E$ is an \textit{admissible} sub-bundle of $\E_b$ if $(F,\dbar_b) \subset \E_b$, $(F,\dbar_0) \subset \E_0$ and $(F^\bot,\dbar_0) \subset \E_0$ are holomorphic sub-bundles.
\end{definition}
In particular, if $F$ is admissible, $(F,\dbar_0)$ is isomorphic to a direct sum of $\G_{0,i}$s (the simple components of $\E_0$). It implies that there is a finite number of them up to isomorphism (with respect to $\dbar_0$).
\begin{definition}\label{DEF:Stabilité P}
    Let $0 \subsetneq F \subset E$ an admissible sub-bundle of $\E_b$.
    
    $(F,\dbar_b)$ is said to be \textit{locally $P_\zeta$-semi-stable} if for all $0 \subsetneq G \subsetneq F$ which are admissible in $\E_b$, $\frac{P_\zeta(G)}{\rk(G)} \leq \frac{P_\zeta(F)}{\rk(F)}$.
    
    It is said to be \textit{locally $P_\zeta$-polystable} if for all $0 \subsetneq G \subsetneq F$ which are admissible in $\E_b$, $\frac{P_\zeta(G)}{\rk(G)} \leq \frac{P_\zeta(F)}{\rk(F)}$ and in case of equality, the inclusion $(G^\bot,\dbar_b) \subset (F,\dbar_b)$ is holomorphic.
    
    It is said to be \textit{locally $P_\zeta$-stable} if for all $0 \subsetneq G \subsetneq F$ which are admissible in $\E_b$, $\frac{P_\zeta(G)}{\rk(G)} < \frac{P_\zeta(F)}{\rk(F)}$.
\end{definition}
By an immediate induction, we see that any locally $P_\zeta$-polystable $(F,\dbar_b) \subset \E_b$ can be decomposed as a direct sum of locally $P_\zeta$-stable simple admissible components $(G_i,\dbar_b)$ verifying $\frac{P_\zeta(G_i)}{\rk(G_i)} = \frac{P_\zeta(F)}{\rk(F)}$. Notice also that $\E_0$ is locally $P_{\zeta_0}$-polystable.

We can now state our main theorems.

\begin{theorem}[Local Kobayashi--Hitchin correspondence]\label{THE:Déformation P-critique}
    There is an open neighbourhood $\mB_d$ of $\dbar_0$ in $\mX_d$ and an open neighbourhood $\mU_d$ of $\zeta_0$ in $\mC^{d - 1}(\mZ)$ such that for all small enough open ball $B_2 \subset V$ of centre $0$, there is an open ball $B_1 \subset B_2$ of centre $0$ such that for all $\zeta \in \mU_d$ and $b \in B_1$, we have equivalence between,
    \begin{enumerate}
        \item $b$ is $\mu_\zeta$-semi-stable (resp. polystable, resp. stable) in $B_2$,
        \item $\E_b$ is $\mu_{\infty,\zeta}$-semi-stable (resp. polystable, resp. stable) in $\mB_d$,
        \item $\E_b$ is locally $P_\zeta$-semi-stable (resp. polystable, resp. stable).
    \end{enumerate}
\end{theorem}

Moreover, at $\zeta = \zeta_0$ (when we don't change the parameter of the equation), we have,

\begin{theorem}\label{THE:Cas zeta = zeta_0}
    With the same $B_2$ and $B_1$ as in Theorem \ref{THE:Déformation P-critique}, we have, for all $b \in B_1$,
    \begin{enumerate}
        \item $b$ is $\mu_{\zeta_0}$-semi-stable in $B_2$.
        \item $b$ is $\mu_{\zeta_0}$-polystable in $B_2$ if and only if $G \cdot b$ is closed in $V$.
    \end{enumerate}
\end{theorem}

Then, we give a topological description of sets of couples $(\zeta,b)$ making $\E_b$ $\mu_{\infty,\zeta}$-semi-stable, polystable and stable. Let
$$
L_{\mathrm{ss}} = \{(\zeta,b) \in \mU_d \times B_1|b \textrm{ is $\mu_\zeta$-semi-stable in } B_2\},
$$
$$
L_{\mathrm{ps}} = \{(\zeta,b) \in \mU_d \times B_1|\E_b \textrm{ is $\mu_\zeta$-polystable in } B_2\},
$$
$$
L_{\mathrm{s}} = \{(\zeta,b) \in \mU_d \times B_1|\E_b \textrm{ is $\mu_\zeta$-stable in } B_2\}.
$$
Clearly, $L_{\mathrm{s}} \subset L_{\mathrm{ps}} \subset L_{\mathrm{ss}} \subset \mU_d \times B_1$.

\begin{proposition}\label{PRO:Topologies lieux stables}
    For all $(\zeta,b) \in \mU_d \times B_1$, $(\zeta,b)$ belonging to $L_{\mathrm{s}}$ (resp. $L_{\mathrm{ps}}$, $L_{\mathrm{ss}}$) only depends on $[\zeta]$ and $G \cdot b \cap B_1$. Moreover,
    \begin{enumerate}
        \item $L_{\mathrm{s}} \subset \mU_d \times B_1$ is open.
        \item For all fixed $\zeta \in \mU_d$, $\{b \in B_1|(\zeta,b) \in L_{\mathrm{ss}}\} \subset B_1$ is open.
        \item For all fixed $b \in B_1$, $\{\zeta \in \mU_d|(\zeta,b) \in L_{\mathrm{ss}}\} \subset \mU_d$ is a closed polyhedral cone centred at $0$ defined by a finite number of inequalities on $[\zeta - \zeta_0]$ and $\{\zeta \in \mU_d|(\zeta,b) \in L_{\mathrm{s}}\} \subset \mU_d$ is its interior.
        \item For all fixed $b \in B_1$, if $(\zeta_0,b) \in L_{\mathrm{ps}}$, $\{\zeta \in \mU_d|(\zeta,b) \in L_{\mathrm{ps}}\} = \{\zeta \in \mU_d|(\zeta,b) \in L_{\mathrm{ss}}\}$.
        \item For all fixed $b \in B_1$, if $\E_b$ is simple, $\{\zeta \in \mU_d|(\zeta,b) \in L_{\mathrm{ps}}\} = \{\zeta \in \mU_d|(\zeta,b) \in L_{\mathrm{s}}\}$.
    \end{enumerate}
\end{proposition}

In Section \ref{SEC:Résultats GIT locaux}, we show some local GIT results in a local non-compact symplectic setting and we fix $\mB_d$. We also study some properties of the moment map flows including their global existence and convergence at $\zeta = \zeta_0$. This section gives all the tools to complete the proofs of the main theorems of the paper.

In the first half of Section \ref{SEC:Preuve du théorème}, we prove Theorem \ref{THE:Déformation P-critique} under an hypothesis of global existence of the flows. In the second half, we use the existence of the moment map flows at all time when $\zeta = \zeta_0$ and a version of the second Ness uniqueness theorem in the semi-stable case to deduce that the flows exist and converge at all time and for all $\zeta$. It will conclude the proof of Theorem \ref{THE:Déformation P-critique}. Proofs of Theorem \ref{THE:Cas zeta = zeta_0} and Proposition \ref{PRO:Topologies lieux stables} will easily follow. We also construct along the way Jordan--Hölder filtrations of locally $P_\zeta$-semi-stable bundles by admissible sub-bundles (Theorem \ref{THE:FJH admissible}) and we show the uniqueness of solutions of the $P_\zeta$-critical equation in $\mB_d$ modulo the unitary gauge group, under existence assumption (Theorem \ref{THE:Unicité solutions}).

In Section \ref{SEC:Kempf--Ness}, we construct Harder--Narasimhan filtrations of any small deformation of $\E_0$ by admissible sub-bundles (Theorem \ref{THE:FHN admissible}), we show a continuity result about the solutions of the $P_\zeta$-critical equations (Theorem \ref{THE:Continuité des opérateurs P-critiques}) and we deduce a local version of the Kempf--Ness theorem (Theorem \ref{THE:Kempf--Ness}).

\section{Local GIT results}\label{SEC:Résultats GIT locaux}

\subsection{Degeneration by $1$-parameter sub-groups}

\begin{lemma}\label{LEM:Sections dbar + alpha holomorphes}
    Let $\F = (F,h,\dbar_F)$ be a holomorphic Hermitian vector bundle (so $\dbar_F$ is integrable) such that all global holomorphic sections of $\F$ are parallel. For all $\gamma \in \Omega^{0,1}(X,\End(F))$ whose $\mC^0$ norm is small enough (depending on $\F$) such that $\dbar_F^\bullet\gamma = 0$, all $(\dbar_F + \gamma)$-holomorphic global sections of $F$ are $\dbar_F$-holomorphic.
\end{lemma}
\begin{proof}
The proof is the same as the proof of \cite[Proposition 4.5]{Buchdahl_Schumacher}. It uses the fact that the Kähler identity $\dbar_F^\bullet = -\i\Lambda\partial_F$ on smooth sections remains true on balanced manifolds.
\end{proof}

\begin{corollary}\label{COR:Sections dbar_b holomorphes}
    Up to shrinking $B$, for all $(b,b') \in B^2$, any holomorphic morphism $\xi : \E_b \rightarrow \E_{b'}$ lies in $\frak{g}$. In particular, $\E_b \cong \E_{b'}$ is and only if $b \in G \cdot b'$ and any holomorphic endomorphism of $\E_b$ is a holomorphic endomorphism of $\E_0$.
\end{corollary}
\begin{proof}
This is a consequence of the continuity of $\Phi$, the fact that it takes values in $\ker(\dbar_0^\bullet)$ (see Proposition \ref{PRO:Propriétés Kuranishi}), Lemma \ref{LEM:Aut réductif} and Lemma \ref{LEM:Sections dbar + alpha holomorphes} applied to $\F = \End(\E_0)$.
\end{proof}

To complete the proof of Theorem \ref{THE:Déformation P-critique}, we need to construct four open balls centred at $0$ with certain properties. We set $B_{4,0} \subset V$ such a ball and we choose it small enough to satisfy all the previous results. We still call $B$ a ball smaller than $B_{4,0}$ and we allow ourselves to shrink it later. However, $B_{4,0}$ is fixed until the end of this paper.

We now give a link between the existence of the limit $\lim_{t \rightarrow +\infty} \e^{t\xi} \cdot b$ for $b \in B$ close enough to $0$ and $\xi \in \i\frak{k}$ in terms of holomorphic sub-bundles of $\E_b$. For all $\xi \in \i\frak{k}$, $\xi$ is Hermitian and the coefficients of its characteristic polynomial are holomorphic thus constant. We deduce that the eigenvalues of $\xi$ are constant so we can write, by the spectral theorem,
$$
\xi = \sum_{\lambda \in \mathrm{Sp}(\xi)} \lambda\Pi_\lambda,
$$
where the $\Pi_\lambda \in \i\frak{k}$ are the orthogonal projection onto the eigenspaces of $\xi$. These eigenspaces define pairwise orthogonal smooth sub-bundles $G_{\xi,\lambda}$ of $E$ such that $(G_{\xi,\lambda},\dbar_0) \subset \E_0$ and $(G_{\xi,\lambda}^\bot,\dbar_0) \subset \E_0$ are holomorphic sub-bundles.

We also define
$$
F_{\xi,\lambda} = \bigoplus_{\nu \leq \lambda} G_{\xi,\nu}.
$$
These bundles also verify $(F_{\xi,\lambda},\dbar_0) \subset \E_0$ and $(F_{\xi,\lambda}^\bot,\dbar_0) \subset \E_0$ being holomorphic, and they define a filtration of $\E_0$ by holomorphic sub-bundles.

\begin{proposition}\label{PRO:Convergence exp(txi)b}
    Up to shrinking $B$ and $\mU_d$, for all $\zeta \in \mU_d$, all $b \in B$ and all $\xi = \sum_{\lambda \in \mathrm{Sp}(\xi)} \lambda\Pi_\lambda$ in $\i\frak{k}$, we have equivalence between,
    \begin{enumerate}
        \item For all $\lambda$, $(F_{\xi,\lambda},\dbar_b) \subset \E_b$ is holomorphic (in other words, they define a filtration of $\E_b$ by admissible sub-bundles),
        \item For all $t \geq 0$, $\e^{t\xi} \cdot b \in B_{4,0}$ and it converges in $B_{4,0}$ when $t \rightarrow +\infty$,
        \item There is are sequences $t_p \rightarrow +\infty$, $b_p \rightarrow b$ and $\xi_p \rightarrow \xi$ of elements of $\i\frak{k}$ such that $(\e^{t_p\xi_p} \cdot b_p)$ is bounded in $V$.
        \item $\norme{\e^{t\xi} \cdot \dbar_b}_{L^2}$ is bounded when $t \rightarrow +\infty$.
    \end{enumerate}
    Moreover, in this case, if we set $b_\infty = \lim_{t \rightarrow +\infty} \e^{t\xi} \cdot b \in B_{4,0}$, all the embeddings $(G_{\xi,\lambda},\dbar_{b_\infty}) \subset \E_{b_\infty}$ are holomorphic and for all $t$, $\norme{\e^{t\xi} \cdot b} \leq \norme{b}$. In particular, $b_\infty \in B$.
\end{proposition}
\begin{proof}\ \\
\noindent\framebox{$1 \Rightarrow 2$} Assume the first point. Let $\mathrm{Sp}(\xi) = \{\lambda_1,\ldots,\lambda_m\}$ with $\lambda_1 < \lambda_2 < \ldots < \lambda_m$. It means that in the orthogonal decomposition
$$
E = \bigoplus_{k = 1}^m G_{\xi,\lambda_k},
$$
we have
$$
\Phi(b) = \dbar_b - \dbar_0 =
\begin{pmatrix}
    \gamma_{11} & \gamma_{12} & \hdots & \gamma_{1m}\\
    0 & \gamma_{22} & \hdots & \gamma_{2m}\\
    \vdots & \vdots & \ddots & \vdots\\
    0 & 0 & \hdots & \gamma_{mm}
\end{pmatrix}
$$
because each $F_{\xi,\lambda_k} = \bigoplus_{i = 1}^k G_{\xi,\lambda_i} \subset E$ is preserved by both $\dbar_0$ and $\dbar_b$. Let $\Id_k$ be the identity of $G_{\xi,\lambda_k}$. We have, for all $t$,
\begin{align*}
    &\e^{t\xi} \cdot \Phi(b)\\
    =\ & \e^{t\xi}\Phi(b)\e^{-t\xi}\\
    =\ &
    \begin{pmatrix}
        \e^{\lambda_1t}\Id_1 & 0 & \hdots & 0\\
        0 & \e^{\lambda_2t}\Id_2 & \hdots & 0\\
        \vdots & \vdots & \ddots & \vdots\\
        0 & 0 & \hdots & \e^{\lambda_mt}\Id_m
    \end{pmatrix}
    \begin{pmatrix}
        \gamma_{11} & \gamma_{12} & \hdots & \gamma_{1m}\\
        0 & \gamma_{22} & \hdots & \gamma_{2m}\\
        \vdots & \vdots & \ddots & \vdots\\
        0 & 0 & \hdots & \gamma_{mm}
    \end{pmatrix}
    \begin{pmatrix}
        \e^{-\lambda_1t}\Id_1 & 0 & \hdots & 0\\
        0 & \e^{-\lambda_2t}\Id_2 & \hdots & 0\\
        \vdots & \vdots & \ddots & \vdots\\
        0 & 0 & \hdots & \e^{-\lambda_mt}\Id_m
    \end{pmatrix}\\
    =\ &
    \begin{pmatrix}
        \gamma_{11} & \e^{-(\lambda_2 - \lambda_1)t}\gamma_{12} & \hdots & \e^{-(\lambda_m - \lambda_1)t}\gamma_{1m}\\
        0 & \gamma_{22} & \hdots & \e^{-(\lambda_m - \lambda_2)t}\gamma_{2m}\\
        \vdots & \vdots & \ddots & \vdots\\
        0 & 0 & \hdots & \gamma_{mm}
    \end{pmatrix}\\
    \tend{t}{+\infty}\ &
    \begin{pmatrix}
        \gamma_{11} & 0 & \hdots & 0\\
        0 & \gamma_{22} & \hdots & 0\\
        \vdots & \vdots & \ddots & \vdots\\
        0 & 0 & \hdots & \gamma_{mm}
    \end{pmatrix}.
\end{align*}
The decomposition $E = \bigoplus_{k = 1}^m G_{\xi,\lambda_k}$ is orthogonal and $\dbar_0$-holomorphic thus the associated decomposition
$$
\Omega^{0,1}(X,\End(E)) = \bigoplus_{i,j = 1}^m \Omega^{0,1}(X,\Hom(G_{\xi,\lambda_j},G_{\xi,\lambda_i}))
$$
is orthogonal for the Hermitian product $\scal{\cdot}{\cdot}_{L^2}$. We deduce that $t \mapsto \norme{\e^{t\xi} \cdot \Phi(b)}_{L^2}$ is non-increasing. Let $B'$ be an open ball in $V$ such that $B \subsetneq B' \subsetneq B_{4,0}$. As $\Phi$ is an embedding and $\partial B'$ is compact and does not contain $0$,
$$
a = \inf_{\partial B'} \norme{\Phi(b)}_{L^2}
$$
is positive. Up to shrinking $B$, we may assume that for all $b \in B$, $\norme{\Phi(b)}_{L^2} < a$. In particular, if $b \in B$, for all $t \geq 0$, $\norme{\e^{t\xi} \cdot \Phi(b)}_{L^2} < a$. It implies that whenever $\e^{t\xi} \cdot b \in B_{4,0}$, it is not in $\partial B'$ (recall that $\Phi$ is $G$-invariant). Thus, it stays in the compact subset $\overline{B'}$ of $B_{4,0}$. In particular, it implies that for all non-negative time, $\e^{t\xi} \cdot b \in \overline{B'} \subset B_{4,0}$ and it converges because $\Phi(\e^{t\xi} \cdot b)$ converges and $\Phi_{|\overline{B'}}$ is injective and proper.

Moreover, the limit structure $\dbar_{b_\infty}$ is diagonal in the decomposition $E = \bigoplus_{k = 1}^m G_{\xi,\lambda_k}$, each $G_{\xi,\lambda_k}$ is preserved hence the holomorphic decomposition
$$
\E_{b_\infty} = \bigoplus_{k = 1}^m (G_{\xi,\lambda_k},\dbar_{b_\infty}).
$$\\

\noindent\framebox{$2 \Rightarrow 3$ and $4$} Trivial.\\

\noindent\framebox{$3 \Rightarrow 2$} First of all, the set of all $\eta$ which are diagonal in the decomposition
$$
\E_0 = \bigoplus_{k = 1}^m (G_{\xi,\lambda_k},\dbar_0)
$$
is the Lie algebra $\frak{t}$ of a complex torus $T$ in $G$. Therefore, it is contained in the Lie algebra $\frak{t}_{\max}$ of a maximal torus $T_{\max}$ of $G$. It can be obtained thanks to an orthogonal and $\dbar_0$-holomorphic decomposition of each $(G_{\xi,\lambda_k},\dbar_0)$ for example. By assumption, $\xi \in \frak{t} \subset \frak{t}_{\max}$. Since $\xi_p \rightarrow \xi$, we can diagonalise it as
$$
\xi_p = u_p\xi_p'u_p^\dagger,
$$
where $u_p \in K$ and $\xi_p' \in \frak{t}_{\max}$. Moreover, as $\xi_p \rightarrow \xi \in \frak{t}_{\max}$, $\xi_p' \rightarrow \xi$ and we may assume that $u_p \rightarrow \Id_E$. Let us write, in the orthogonal decomposition $E = \bigoplus_{k = 1}^m G_{\xi,\lambda_k}$,
$$
\xi_p' =
\begin{pmatrix}
    \lambda_{p,1} & 0 & \hdots & 0\\
    0 & \lambda_{p,2} & \hdots & 0\\
    \vdots & \vdots & \ddots & \vdots\\
    0 & 0 & \hdots & \lambda_{p,m}
\end{pmatrix}.
$$
There is no reason to believe that $\xi_p'$ is in $\frak{t}$ so each $\lambda_{p,k}$ is a holomorphic Hermitian endomorphism of $(G_{\xi,\lambda_k},\dbar_0)$ but it needs not be a homothety. In other words, it is block diagonal but not diagonal. As $\xi_p' \rightarrow \xi$, we have $\lambda_{p,k} \rightarrow \lambda_k\Id_k$ for all $k$. Let for all $p$, $b_p' = u_p^\dagger \cdot b_p$ and let us set
$$
b =
\begin{pmatrix}
    b_{11} & b_{12} & \hdots & b_{1m}\\
    b_{21} & b_{22} & \hdots & b_{2m}\\
    \vdots & \vdots & \ddots & \vdots\\
    b_{m1} & b_{m2} & \hdots & b_{mm}
\end{pmatrix}, \qquad
b_p' = 
\begin{pmatrix}
    b_{p,11} & b_{p,12} & \hdots & b_{p,1m}\\
    b_{p,21} & b_{p,22} & \hdots & b_{p,2m}\\
    \vdots & \vdots & \ddots & \vdots\\
    b_{p,m1} & b_{p,m2} & \hdots & b_{p,mm}
\end{pmatrix}.
$$
As $u_p \rightarrow \Id_E$, for all $i,j$, $b_{p,ij} \rightarrow b_{ij}$. We have, for all $p$,
$$
\e^{t_p\xi_p} \cdot b_p = u_p\e^{t_p\xi_p'} \cdot b_p' = u_p \cdot
\begin{pmatrix}
    \e^{t_p\lambda_{p,1}}b_{p,11}\e^{-t_p\lambda_{p,1}} & \e^{t_p\lambda_{p,1}}b_{p,12}\e^{-t_p\lambda_{p,2}} & \hdots & \e^{t_p\lambda_{p,1}}b_{p,1m}\e^{-t_p\lambda_{p,m}}\\
    \e^{t_p\lambda_{p,2}}b_{p,21}\e^{-t_p\lambda_{p,1}} & \e^{t_p\lambda_{p,2}}b_{p,22}\e^{-t_p\lambda_{p,2}} & \hdots & \e^{t_p\lambda_{p,2}}b_{p,2m}\e^{-t_p\lambda_{p,m}}\\
    \vdots & \vdots & \ddots & \vdots\\
    \e^{t_p\lambda_{p,m}}b_{p,m1}\e^{-t_p\lambda_{p,1}} & \e^{t_p\lambda_{p,m}}b_{p,m2}\e^{-t_p\lambda_{p,2}} & \hdots & \e^{t_p\lambda_{p,m}}b_{p,mm}\e^{-t_p\lambda_{p,m}}
\end{pmatrix}.
$$
By assumption, this sequence is bounded so for all $i,j$, $(\e^{t_p\lambda_{p,i}}b_{p,ij}\e^{-t_p\lambda_{p,j}})$ is bounded. Let $i > j$, so $\lambda_i > \lambda_j$. Set $\varepsilon_{p,i} = \lambda_{p,i} - \lambda_i\Id_i \rightarrow 0$ and $\varepsilon_{p,j} = \lambda_{p,j} - \lambda_j\Id_j \rightarrow 0$. We have
\begin{align*}
    \norme{\e^{t_p\lambda_{p,i}}b_{p,ij}\e^{-t_p\lambda_{p,j}}}_{L^2} & = \norme{\e^{t_p\varepsilon_{p,i}}\e^{(\lambda_i - \lambda_j)t_p}b_{p,ij}\e^{-t_p\varepsilon_{p,j}}}_{L^2}\\
    & \geq \norme{\e^{-t_p\varepsilon_{p,i}}}^{-1}\norme{\e^{t_p\varepsilon_{p,j}}}^{-1}\norme{\e^{(\lambda_i - \lambda_j)t_p}b_{p,ij}}_{L^2}\\
    & \geq \e^{-t_p\norme{\varepsilon_{p,i}}}\e^{-t_p\norme{\varepsilon_{p,j}}}\e^{(\lambda_i - \lambda_j)t_p}\norme{b_{p,ij}}_{L^2}\\
    & \geq \e^{\frac{1}{2}(\lambda_i - \lambda_j)t_p}\norme{b_{p,ij}}_{L^2} \textrm{ for $p$ large enough.}
\end{align*}
Since this sequence is bounded, we must have $\norme{b_{p,ij}}_{L^2} \rightarrow 0$ so $b_{ij} = 0$. It holds for all $i > j$ so $b$ is upper diagonal. Moreover, the decomposition
$$
\Omega^{0,1}(X,\End(E)) = \bigoplus_{i,j = 1}^m \Omega^{0,1}(X,\Hom(G_j,G_i))
$$
is orthogonal. It implies that
$$
t \mapsto \norme{\e^{t\xi} \cdot b}_{L^2} = \sum_{i \leq j} \e^{t(\lambda_i - \lambda_j)}\norme{b_{ij}}_{L^2}
$$
decreases and $b_\infty = \lim_{t \rightarrow +\infty} \e^{t\xi} \cdot b$ exists and is given by the diagonal part of $b$. In particular, for all $t$, $\norme{\e^{t\xi} \cdot b}_{L^2} \leq \norme{b}_{L^2}$ and $\norme{b_\infty}_{L^2} \leq \norme{b}_{L^2}$ so $\e^{t\xi} \cdot b$ and $b_\infty$ lie in $B \subset B_{4,0}$.\\

\noindent\framebox{$4 \Rightarrow 1$} Write, in the orthogonal decomposition $E = \bigoplus_{k = 1}^m G_{\xi,\lambda_k}$,
$$
\Phi(b) =
\begin{pmatrix}
\gamma_{11} & \gamma_{12} & \hdots & \gamma_{1m}\\
\gamma_{21} & \gamma_{22} & \hdots & \gamma_{2m}\\
\vdots & \vdots & \ddots & \vdots\\
\gamma_{m1} & \gamma_{m2} & \hdots & \gamma_{mm}
\end{pmatrix}.
$$
As before, we have, for all $t$,
$$
\e^{t\xi} \cdot \Phi(b) =
\begin{pmatrix}
\gamma_{11} & \e^{-(\lambda_2 - \lambda_1)t}\gamma_{12} & \hdots & \e^{-(\lambda_m - \lambda_1)t}\gamma_{1m}\\
\e^{(\lambda_2 - \lambda_1)t}\gamma_{21} & \gamma_{22} & \hdots & \e^{-(\lambda_m - \lambda_2)t}\gamma_{2m}\\
\vdots & \vdots & \ddots & \vdots\\
\e^{(\lambda_m - \lambda_1)t}\gamma_{m1} & \e^{(\lambda_m - \lambda_2)t}\gamma_{m2} & \hdots & \gamma_{mm}
\end{pmatrix}.
$$
We easily see that if $\norme{\e^{t\xi} \cdot \dbar_b}_{L^2}$ is bounded for large $t$, then $\gamma_{ij} = 0$ when $i > j$. We deduce that each $F_{\xi,\lambda_k}$ is preserved by $\dbar_b = \dbar_0 + \Phi(b)$ thus $(F_{\xi,\lambda_k},\dbar_b) \subset \E_b$ is holomorphic.
\end{proof}

The following lemma will be useful to show uniqueness results.

\begin{lemma}\label{LEM:Convexité}
    Up to shrinking $B$, if $b \in B$ and $\xi \in \i\frak{k}$ are such that $\e^\xi \cdot b \in B$, then, all $\e^{t\xi} \cdot b$ for $0 \leq t \leq 1$ lie in $B_{4,0}$.
\end{lemma}
\begin{proof}
Let $\mathrm{Sp}(\xi) = \{\lambda_1,\ldots,\lambda_m\}$ with $\lambda_1 < \lambda_2 < \ldots < \lambda_m$ so
$$
\xi = \sum_{k = 1}^m \lambda_k\Pi_{\lambda_k}.
$$
Similarly as in the proof of Proposition \ref{PRO:Convergence exp(txi)b}, we have an orthogonal decomposition
$$
E = \bigoplus_{k = 1}^m G_{\xi,\lambda_k}
$$
in which we can decompose $b$ as
$$
b =
\begin{pmatrix}
    b_{11} & b_{12} & \hdots & b_{1m}\\
    b_{21} & b_{22} & \hdots & b_{2m}\\
    \vdots & \vdots & \ddots & \vdots\\
    b_{m1} & b_{m2} & \hdots & b_{mm}
\end{pmatrix}.
$$
So for all $t$, we have,
$$
\e^{t\xi} \cdot b = \e^{t\xi}b\e^{-t\xi} =
\begin{pmatrix}
    b_{11} & \e^{(\lambda_1 - \lambda_2)t}b_{12} & \hdots & \e^{(\lambda_1 - \lambda_m)t}b_{1m}\\
    \e^{(\lambda_2 - \lambda_1)t}b_{21} & b_{22} & \hdots & \e^{(\lambda_2 - \lambda_m)t}b_{2m}\\
    \vdots & \vdots & \ddots & \vdots\\
    \e^{(\lambda_m - \lambda_1)t}b_{m1} & \e^{(\lambda_m - \lambda_2)t}b_{m2} & \hdots & b_{mm}
\end{pmatrix}.
$$
In particular, for $t = 1$,
$$
\e^\xi \cdot b =
\begin{pmatrix}
    b_{11} & \e^{\lambda_1 - \lambda_2}b_{12} & \hdots & \e^{\lambda_1 - \lambda_m}b_{1m}\\
    \e^{\lambda_2 - \lambda_1}b_{21} & b_{22} & \hdots & \e^{\lambda_2 - \lambda_m}b_{2m}\\
    \vdots & \vdots & \ddots & \vdots\\
    \e^{\lambda_m - \lambda_1}b_{m1} & \e^{\lambda_m - \lambda_2}b_{m2} & \hdots & b_{mm}
\end{pmatrix}.
$$
And for all $i,j$ and $0 \leq t \leq 1$,
$$
\norme{\e^{(\lambda_i - \lambda_j)t}b_{ij}}_{L^2}^2 = \e^{2(\lambda_i - \lambda_j)t}\norme{b_{ij}}_{L^2}^2 \leq \max\{\norme{b_{ij}}_{L^2}^2,\norme{\e^{\lambda_i - \lambda_j}b_{ij}}_{L^2}^2\} \leq \norme{b_{ij}}_{L^2}^2 + \norme{\e^{\lambda_i - \lambda_j}b_{ij}}_{L^2}^2.
$$
When we sum the last inequality for all $1 \leq i,j \leq m$, we obtain that
$$
\norme{\e^{t\xi} \cdot b}_{L^2}^2 \leq \norme{b}_{L^2}^2 + \norme{\e^\xi \cdot b}_{L^2}^2.
$$
Therefore, if the radius of $B$ is less than or equal to the radius of $B_{4,0}$ divided by $\sqrt{2}$, the result of the lemma holds, which we may assume up to shrinking $B$.
\end{proof}

From now on, we set $B_{3,0} \subset B_{4,0}$ which satisfies Proposition \ref{PRO:Convergence exp(txi)b} and Lemma \ref{LEM:Convexité}. Once again, it is fixed and we still call $B$ a smaller ball that we allow ourselves to shrink later.

\subsection{Construction of $\mB_d$}

In this sub-section, we show that that all zeroes of $\mu_{\infty,\zeta}$ which are $\mC^d$ close to $\dbar_0$ and in the gauge orbit of $\dbar_b$ for some $b \in B$ actually belong to the deformed slice $\tilde{\Phi}(\zeta,B_{3,0})$ up to a unitary gauge transform. We define here the open neighbourhood $\mB_d$ of $\dbar_0$ mentioned in Theorem \ref{THE:Déformation P-critique}.

\begin{lemma}\label{LEM:Constuction B' subset B''}
    Let $B'' \subset B_{3,0}$ be an open ball centred at $0$. There is a ball $B' \subset B''$ such that for all $\alpha \in \im(\dbar_0^\bullet)$ whose $\mC^0$ norm is small enough, for all $f \in \mC^2(\G^\C(E))$ and all $(b,b') \in {B'}^2$, if
    $$
    f \cdot \dbar_b = \dbar_{b'} + \alpha,
    $$
    then $f \in G$ and $f \cdot b \in B''$.
\end{lemma}
\begin{proof}
By Proposition \ref{PRO:Propriétés Kuranishi}, $\Phi$ takes values in $\ker(\dbar_0^\bullet)$ and since $\dbar_0$ is integrable, $\im(\dbar_0^\bullet) \subset \ker(\dbar_0^\bullet)$. When $b \in B_{3,0}$ and $\alpha \in \im(\dbar_0^\bullet)$, the complex structure of $\Hom((E,\dbar_b),(E,\dbar_{b'} + \alpha))$ is given by
$$
\dbar_{b,b',\alpha}f = (\dbar_{b'} + \alpha) \circ f - f \circ \dbar_b = \dbar_0f + (\Phi(b') + \alpha)f - f\Phi(b).
$$
Since $\Phi(b)$, $\Phi(b')$ and $\alpha$ are in $\ker(\dbar_0^\bullet)$, then $\dbar_{b,b',\alpha} - \dbar_0 \in \ker(\dbar_0^\bullet)$ (be careful, this time, we are in $\End(E)$, not in $E$). In particular, we may apply Lemma \ref{LEM:Sections dbar + alpha holomorphes} to $\F = \End(\E_0)$ so, when $b$, $b'$ and the $\mC^0$ norm of $\alpha$ are small enough, any $f \in \ker(\dbar_{b,b',\alpha})$ is in $\ker(\dbar_0)$.

Let $B' \subset B_{3,0}$ and $\varepsilon > 0$ such that this property is true when $(b,b') \in {B'}^2$ and $\norme{\alpha}_{\mC^0} \leq \varepsilon$. Assume without loss of generality that $B' \subset B''$.

By Proposition \ref{PRO:Propriétés Kuranishi}, $\Phi(b) = b + \mathrm{O}(\norme{b}_{L^2}^2)$ thus there is a constant $C > 0$ such that, up to shrinking $B'$, for all $b \in B'$,
\begin{equation}\label{EQ:Inégalité norme Phi(b) - b}
    \norme{\Phi(b) - b}_{L^2} \leq C\norme{b}_{L^2}.
\end{equation}
Moreover, $\Phi(b) - b \in \im(\dbar_0^\bullet) \subset V^\bot$ and the action of $G$ by conjugation preserves both $V$ and $V^\bot$ thus, for all $b \in B'$ and all $f \in G$,
\begin{equation}\label{EQ:Inégalité norme fPhi(b)f-1}
    \norme{fbf^{-1}}_{L^2} \leq \norme{f\Phi(b)f^{-1}}_{L^2}.
\end{equation}
Let $C' > 0$ be a constant such that, on $(0,1)$-forms,
\begin{equation}\label{EQ:Norme L2 dominée}
    \norme{\cdot}_{L^2} \leq C'\norme{\cdot}_{\mC^0}.
\end{equation}
Let $r' \leq r''$ the respective radii of $B'$ and $B''$. Assume up to reducing $\varepsilon$ and shrinking $B'$ that
\begin{equation}\label{EQ:Inégalité rayons}
    r' + C{r'}^2 + C'\varepsilon \leq r''.
\end{equation}

We can now prove the lemma. Let $(b,b') \in {B'}^2$, $\alpha \in \im(\dbar_0^\bullet)$ with $\norme{\alpha}_{\mC^0} \leq \varepsilon$ and $f \in \G^\C(E)$ such that $f \cdot \dbar_b = \dbar_{b'} + \alpha$. We have $f \in \ker(\dbar_{b,b',\alpha})$ so by construction of $B'$ and $\varepsilon$, $f \in G$ and we have
$$
f\Phi(b)f^{-1} = \Phi(b') + \alpha.
$$
Therefore,
\begin{align*}
    \norme{fbf^{-1}}_{L^2} & \leq \norme{f\Phi(b)f^{-1}}_{L^2} \textrm{ by (\ref{EQ:Inégalité norme fPhi(b)f-1}),}\\
    & = \norme{\Phi(b') + \alpha}_{L^2}\\
    & \leq \norme{b'}_{L^2} + \norme{\Phi(b') - b'}_{L^2} + \norme{\alpha}_{L^2}\\
    & \leq \norme{b'}_{L^2} + C\norme{b'}_{L^2}^2 + C'\norme{\alpha}_{\mC^0} \textrm{ by (\ref{EQ:Inégalité norme Phi(b) - b}) and (\ref{EQ:Norme L2 dominée}),}\\
    & < r'' \textrm{ by (\ref{EQ:Inégalité rayons}).}
\end{align*}
It proves that $f \cdot b = fbf^{-1} \in B''$.
\end{proof}

\begin{proposition}\label{PRO:Zéro dans la tranche déformée}
    Up to shrinking $B$ and $\mU_d$, there is an open neighbourhood $\mB_d$ of $\dbar_0$ in $\mX_d$ such that $\dbar_0 + \tilde{\Phi}(\mU_d,B) \subset \mB_d$. Reciprocally, there is a smooth map
    $$
    \beta : \mB_d \rightarrow B_{3,0},
    $$
    such that for all $b \in B$ and all $\dbar \in \G^\C(E) \cdot \dbar_b \cap \mB_d$, $\beta(\dbar) \in G \cdot b$. If moreover, $\mu_{\infty,\zeta}(\dbar) = 0$, there is a unitary gauge transform $u \in \G(E,h)$ such that $\dbar = u \cdot (\dbar_0 + \tilde{\Phi}(\zeta,\beta(\dbar)))$ and $\tilde{\mu}_\zeta(\beta(\dbar)) = 0$.
\end{proposition}
\begin{proof}
Consider the germ of smooth map
$$
\Psi : \fonction{(B_{3,0} \times \mC^{d + 1}(\frak{g}^\bot) \times \mC^d(\im(\dbar_0^\bullet)),(0,0,0))}{(\mX_d,\dbar_0)}{(b,s,\alpha)}{\e^{\i\Im(s)}\e^{\Re(s)} \cdot (\dbar_b + \alpha)}.
$$
We have
$$
d\Psi(0,0,0)(v,s,\alpha) = v - \dbar_0s + \alpha,
$$
which is an isomorphism by Hodge decomposition (\ref{EQ:Hodge}). Therefore, by the inverse function theorem, there are a neighbourhood $B''$ of $0$ in $B_{3,0}$, a neighbourhood $\mathcal{M}$ of $0$ in $\mC^{d + 1}(\frak{g}^\bot)$, a neighbourhood $\mathcal{N}$ of $0$ in $\mC^d(\im(\dbar_0^\bullet))$ and a neighbourhood $\mB_d'$ of $\dbar_0$ in $\mX_d$ such that
$$
\Psi : B'' \times \mathcal{M} \times \mathcal{N} \rightarrow \mB_d'
$$
is a smooth diffeomorphism. Let $B' \subset B''$ given by Lemma \ref{LEM:Constuction B' subset B''}. Since the $\mC^0$ norm dominates the $\mC^d$ norm, we may assume up to shrinking $\mathcal{N}$ (thus $\mB_d'$) that all $\alpha \in \mathcal{N}$ satisfy Lemma \ref{LEM:Constuction B' subset B''} with $B'$ and $B''$. We may also assume up to shrining $\mathcal{M}$ (thus $\mB_d'$) that the Hermitian part of all $s \in \mathcal{M}$ verify the third point of Proposition \ref{PRO:Tranche déformée P}. Let
$$
\mB_d = \Psi(B' \times \mathcal{M} \times \mathcal{N}).
$$
Up to shrinking $B$ and $\mU_d$, we may assume that $B \subset B'$ and $\dbar_0 + \tilde{\Phi}(\mU_d,B) \subset \mB_d$ so the first condition is verified. We then define $\beta$ as the smooth map given by the first component of the local inverse of $\Psi$.
$$
\beta = (\Psi^{-1})_1 : \mB_d \rightarrow B' \subset B_{3,0}.
$$
Let $\zeta \in \mU_d$, $b \in B \subset B'$ and $g \in \G^\C(E)$ such that $\dbar = g \cdot \dbar_b \in \mB_d$. Let $(b',s,\alpha) = \Psi^{-1}(\dbar) \in B' \times \mathcal{M} \times \mathcal{N}$ so $b' = \beta(\dbar)$ and $g \cdot \dbar_b = \e^{\i\Im(s)}\e^{\Re(s)} \cdot (\dbar_{b'} + \alpha)$. Let $f = \e^{-\Re(s)}\e^{-\i\Im(s)}g \in \G^\C(E)$ so
$$
f \cdot \dbar_b = \dbar_{b'} + \alpha.
$$
By Lemma \ref{LEM:Constuction B' subset B''}, $f \in G$ and $f \cdot b \in B''$ hence $f \cdot \dbar_b = \dbar_{f \cdot b}$ and since $\Psi$ is injective on $B'' \times \mathcal{M} \times \mathcal{N}$, the equality $\dbar_{f \cdot b} = \dbar_{b'} + \alpha$ implies $f \cdot b = b'$ and $\alpha = 0$. In particular, $b' \in G \cdot b$ as wanted.

Assume now that $\mu_{\infty,\zeta}(\dbar) = 0$. Recall that $\dbar = \e^{\i\Im(s)}\e^{\Re(s)} \cdot (\dbar_{b'} + \alpha) = u\e^{\Re(s)} \cdot \dbar_{b'}$ with $u = \e^{\i\Im(s)} \in \mC^{d + 1}(\G(E,h))$. As $\dbar$ is $P_\zeta$-critical, so is $\e^{\Re(s)} \cdot \dbar_{b'}$. We have $s \in \mathcal{M}$ and $\mu_{\infty,\zeta}(\e^{\Re(s)} \cdot \dbar_b) = 0$ so by Proposition \ref{PRO:Tranche déformée P}, $\Re(s) = \sigma(\zeta,b')$ so $\e^{\Re(s)} \cdot \dbar_{b'} = \dbar_0 + \tilde{\Phi}(\zeta,b')$. Finally, we have
$$
\dbar = u \cdot (\dbar_0 + \tilde{\Phi}(\zeta,b')), \qquad b' = f \cdot b \in G \cdot b \cap B_{3,0}.
$$
Moreover, $\tilde{\mu}_\zeta(b') = \mu_{\infty,\zeta}(\dbar_0 + \tilde{\Phi}(\zeta,b')) = 0$.
\end{proof}

\subsection{Gradient flows of the moment map squared, existence and convergence}

In this sub-section, we define and study the negative gradient flows of the moment maps. In particular, we show that the flows associated with $\mu_{\zeta_0}$ and $\tilde{\mu}_{\zeta_0}$ are well-defined at all time if the starting point is close enough to $0$. It relies on \Lo\ gradient inequality and Duistermaat's theorem. In the Kähler case (\textit{i.e.} for $\mu_{\zeta_0}$), this result is due to Ortu \cite[Proposition 2.4]{Ortu}. In \cite{Ortu}, Ortu studies cscK like equations but her results \cite[Sub-section 2.1]{Ortu} apply for general moment map frameworks. In the non-Kähler case (\textit{i.e.} for $\tilde{\mu}_{\zeta_0}$), the proof is similar but we include it for completeness.

Let $\zeta \in \mU_d$ and $\phi_\zeta$ (resp. $\tilde{\phi}_\zeta$) be the flows defined by
$$
\left\{
\begin{array}{rcl}
    \phi_\zeta(b,0) & = & b,\\
    \frac{\partial\phi_\zeta}{\partial t}(b,t) & = & -\i L_{\phi_\zeta(b,t)}\mu_\zeta(\phi_\zeta(b,t)).
\end{array}
\right.
\left\{
\begin{array}{rcl}
    \tilde{\phi}_\zeta(b,0) & = & b,\\
    \frac{\partial\tilde{\phi}_\zeta}{\partial t}(b,t) & = & -\i L_{\tilde{\phi}_\zeta(b,t)}\tilde{\mu}_\zeta(\tilde{\phi}_\zeta(b,t)).
\end{array}
\right..
$$
Recall that when $b \in B$ and $\xi \in \frak{g}$, $L_b\xi$ is the infinitesimal action of $\xi$ at $b$ given by
$$
L_b\xi = \frac{\partial}{\partial t}|_{t = 0}(\e^{t\xi} \cdot b) = [\xi,b] \in T_bB.
$$
Let for all $\zeta$,
$$
f_\zeta : b \mapsto \frac{1}{2}\norme{\mu_\zeta(b)}^2, \qquad \tilde{f}_\zeta : b \mapsto \frac{1}{2}\norme{\tilde{\mu}_\zeta(b)}^2,
$$
be (half) the squares of the norms of the moment maps. The following lemmas are \cite[Lemma 3.1, Lemma 3.2]{GRS}, the proofs are similar.

\begin{lemma}\label{LEM:Flux de gradient}
    Let $b \in B$, $v \in T_bB$ and $\zeta \in \mU_d$. The flows $\phi_\zeta$ and $\tilde{\phi}_\zeta$ verify, where they are defined,
    $$
    \Omega_{\zeta,\phi_\zeta(b,t)}\left(\i\frac{\partial\phi_\zeta}{\partial t}(b,t),v\right) = df_\zeta(\phi_\zeta(b,t))v, \qquad \tilde{\Omega}_{\zeta,\tilde{\phi}_\zeta(b,t)}\left(\i\frac{\partial\tilde{\phi}_\zeta}{\partial t}(b,t),v\right) = d\tilde{f}_\zeta(\tilde{\phi}_\zeta(b,t))v.
    $$
\end{lemma}

\begin{lemma}\label{LEM:Orbite préservée}
    Let $b \in B$ and $\zeta \in \mU_d$. Let $g$ (resp. $\tilde{g}$) in $G$ be defined by
    $$
    \left\{
    \begin{array}{rcl}
        g(0) & = & \Id_E,\\
        g(t)^{-1}g'(t) & = & \i\mu_\zeta(\phi_\zeta(b,t)).
    \end{array}
    \right.
    \qquad
    \left\{
    \begin{array}{rcl}
        \tilde{g}(0) & = & \Id_E,\\
        \tilde{g}(t)^{-1}\tilde{g}'(t) & = & \i\tilde{\mu}_\zeta(\tilde{\phi}_\zeta(b,t)).
    \end{array}
    \right..
    $$
    They verify, where they are defined,
    $$
    \phi_\zeta(b,t) = g(t)^{-1} \cdot b, \qquad \tilde{\phi}_\zeta(b,t) = \tilde{g}(t)^{-1} \cdot b.
    $$
    In particular, $G$-orbits are preserved by $\phi_\zeta$ and $\tilde{\phi}_\zeta$.
\end{lemma}

The first equality in Lemma \ref{LEM:Flux de gradient} can be rewritten as
$$
\frac{\partial\phi_\zeta}{\partial t}(b,t) = -\nabla_{\Omega_\zeta} f_\zeta(\phi_\zeta(b,t)),
$$
where $\nabla_\Omega$ is the gradient with respect to the scalar product induced by $\Omega_\zeta$. Similarly, a consequence of Lemma \ref{LEM:Omega tilde positive} is that $\tilde{\Omega}_\zeta$ is a symplectic form. In particular, we can define
$$
-\tilde{\Omega}_{\zeta,b}\left(\i\nabla_{\tilde{\Omega}_\zeta}\tilde{f}_\zeta(b),v\right) = d\tilde{f}_\zeta(b)v.
$$
In this case, the second equality in Lemma \ref{LEM:Flux de gradient} can be rewritten as
\begin{equation}\label{EQ:Flux de gradient}
    \frac{\partial\tilde{\phi}_\zeta}{\partial t}(b,t) = -\nabla_{\tilde{\Omega}_\zeta}\tilde{f}_\zeta(\tilde{\phi}_\zeta(b,t)).
\end{equation}
Notice that $\nabla_{\tilde{\Omega}_\zeta}\tilde{f}_\zeta$ shares properties with the usual gradients of $\tilde{f}$. Clearly, it vanishes at some point $b$ if and only if $b$ is a critical point of $\tilde{f}_\zeta$. Moreover, by Lemma \ref{LEM:Omega tilde positive}, there is a positive constant $c$ such that
\begin{equation}\label{EQ:Norme gradient tilde}
    d\tilde{f}_\zeta(b)\nabla_{\tilde{\Omega}_\zeta}\tilde{f}_\zeta(b) = -\tilde{\Omega}_{\zeta,b}(\i\nabla_{\tilde{\Omega}_\zeta}\tilde{f}_\zeta(b),\nabla_{\tilde{\Omega}_\zeta}\tilde{f}_\zeta(b)) \geq c\norme{\nabla_{\tilde{\Omega}_\zeta}\tilde{f}_\zeta(b)}_{L^2}^2 \geq 0
\end{equation}
with equality if and only if $\nabla_{\tilde{\Omega}_\zeta}\tilde{f}_\zeta(b) = 0$. We easily see that it implies that $t \mapsto \tilde{f}_\zeta(\tilde{\phi}_\zeta(b,t))$ is non-increasing (similarly with $t \mapsto f_\zeta(\phi_\zeta(b,t))$). Using this, if $\norme{\cdot}_{L^2(V^\lor)}$ is the operator norm on the dual space $V^\lor$ of $V$ induced by $\norme{\cdot}_{L^2}$, we have
\begin{equation}\label{EQ:Equivalence normes gradients}
    \norme{d\tilde{f}_\zeta(b)}_{L^2(V^\lor)} \geq \frac{\abs{d\tilde{f}_\zeta(b)\nabla_{\tilde{\Omega}_\zeta}\tilde{f}_\zeta(b)}}{\norme{\nabla_{\tilde{\Omega}_\zeta}\tilde{f}_\zeta(b)}_{L^2}} \geq c\norme{\nabla_{\tilde{\Omega}_\zeta}\tilde{f}_\zeta(b)}_{L^2}.
\end{equation}
However, $\nabla_{\tilde{\Omega}_\zeta}\tilde{f}_\zeta(b)$ is \textit{a priori} not a gradient of $\tilde{f}_\zeta$ because the non-degenerate positive forms $(v,w) \mapsto \tilde{\Omega}_{\zeta,b}(\i v,w)$ need not be symmetric as $\tilde{\Omega}_\zeta$ is not compatible with the complex structure on $B$.

We now study the existence of these flows at $\zeta = \zeta_0$. A key tool is the \Lo\ gradient inequality.

\begin{proposition}[\Lo\ gradient inequality, {\cite[Chapter 18, Proposition 1]{Lojasiewicz}}]\label{PRO:Lojasiewicz}
    If $U \subset \R^d$ is open and $g : U \rightarrow \R$ is real analytic, then for all $x \in U$ critical point of $g$, there is a neighbourhood $V$ of $x$ in $U$ and constants $a > 0$ and $\frac{1}{2} \leq \theta < 1$ such that, for all $y \in V$,
    $$
    a\abs{g(x) - g(y)}^\theta \leq \norme{\nabla g(y)}.
    $$
\end{proposition}

Using the Marle--Guillemin--Sternberg form \cite[Theorem 2.1]{Lerman}, we can show that the norms squared of the moment maps are locally real analytic up to a composition on the right with a smooth local diffeomorphism. It means that we can apply Proposition \ref{PRO:Lojasiewicz} to the functions $f_{\zeta_0}$ and $\tilde{f}_{\zeta_0}$ at $b = 0$. Indeed, $\mu_{\zeta_0}(0) = \tilde{\mu}_{\zeta_0}(0) = 0$ thus $f_{\zeta_0}$ and $\tilde{f}_{\zeta_0}$ have a critical point at $0$. Moreover, by (\ref{EQ:Equivalence normes gradients}), we can replace the gradient of $\tilde{f}_{\zeta_0}$ by $\nabla_{\tilde{\Omega}_{\zeta_0}}\tilde{f}_{\zeta_0}$. Thus, up to shrinking $B$, there are constants $a > 0$ and $\frac{1}{2} \leq \theta < 1$ such that for all $b \in B$,
\begin{equation}\label{EQ:Lojasiewicz}
    af_{\zeta_0}(b)^\theta \leq \norme{\nabla_{\Omega_{\zeta_0}}f_{\zeta_0}(b)}_{L^2}, \qquad a\tilde{f}_{\zeta_0}(b)^\theta \leq \norme{\nabla_{\tilde{\Omega}_{\zeta_0}}\tilde{f}_{\zeta_0}(b)}_{L^2}.
\end{equation}
In particular, all critical points of $f_{\zeta_0}$ or $\tilde{f}_{\zeta_0}$ in $B$ are zeroes of the associated moment maps.

\Lo\ gradient inequality is the last assumption we needed on the elements of $b$. Let $B_{2,0} = B$ which satisfies Propositions \ref{PRO:Zéro dans la tranche déformée} and \ref{PRO:Lojasiewicz}. Let $B_2 \subset B_{2,0}$. We will now construct $B_1 \subset B_2$ in function of $B_2$ which satisfies the assumptions of Theorem \ref{THE:Déformation P-critique}. In particular, we want the flows $\phi_{\zeta_0}$ and $\tilde{\phi}_{\zeta_0}$ starting at points of $B_1$ to stay in $B_2$. For now, set $B_1 = B_2$, but we allow ourselves to shrink it.

\begin{proposition}[{\cite[Proposition 2.4]{Ortu}}]\label{PRO:Définition et convergence des flux zeta_0}
    Up to shrinking $B_1$, $\phi_{\zeta_0}(b,t)$ (resp. $\tilde{\phi}_{\zeta_0}(b,t)$) are well-defined in $B_2$ for all non-negative time as long as $b \in B_1$. Moreover, they stay in some compact subset of $B_2$ (independent from $b \in B_1$).
\end{proposition}
\begin{proof}
The proof is the same for $\phi_{\zeta_0}$ and $\tilde{\phi}_{\zeta_0}$ so we prove it for $\tilde{\phi}_{\zeta_0}$ to emphasize the fact that the non-holomorphy of $\tilde{\Phi}(\zeta_0,\cdot)$ is not a problem here. The proof is similar to the proof of \cite[Theorem 3.3]{GRS}. Let $b \in B_1$. If $\tilde{f}_{\zeta_0}(b) = 0$, then $\tilde{\phi}_{\zeta_0}(b,\cdot) = b$ is constant and defined for all $t \geq 0$. Else, for all $t$ where $\tilde{\phi}_{\zeta_0}(b,t)$ is defined, $\tilde{f}_{\zeta_0}(\tilde{\phi}_{\zeta_0}(b,t)) \neq 0$ because zeroes of $\tilde{f}_{\zeta_0}$ are critical points and
\begin{align*}
    -\frac{\partial}{\partial t}(\tilde{f}_{\zeta_0}(\tilde{\phi}_{\zeta_0}(b,t))^{1 - \theta}) & = -(1 - \theta)\tilde{f}_{\zeta_0}(\tilde{\phi}_{\zeta_0}(b,t))^{-\theta}d\tilde{f}_{\zeta_0}(\tilde{\phi}_{\zeta_0}(b,t))\frac{\partial \tilde{\phi}_{\zeta_0}}{\partial t}(b,t)\\
    & = (1 - \theta)\tilde{f}_{\zeta_0}(\tilde{\phi}_{\zeta_0}(b,t))^{-\theta}d\tilde{f}_{\zeta_0}(\tilde{\phi}_{\zeta_0}(b,t))\nabla_{\tilde{\Omega}_{\zeta_0}}\tilde{f}_{\zeta_0}(\tilde{\phi}_{\zeta_0}(b,t)) \textrm{ by (\ref{EQ:Flux de gradient}),}\\
    & \geq c(1 - \theta)\tilde{f}_{\zeta_0}(\tilde{\phi}_{\zeta_0}(b,t))^{-\theta}\norme{\nabla_{\tilde{\Omega}_{\zeta_0}}\tilde{f}_{\zeta_0}(\tilde{\phi}_{\zeta_0}(b,t))}_{L^2}^2 \textrm{ by (\ref{EQ:Norme gradient tilde}),}\\
    & \geq ac(1 - \theta)\norme{\nabla_{\tilde{\Omega}_{\zeta_0}}\tilde{f}_{\zeta_0}(\tilde{\phi}_{\zeta_0}(b,t))}_{L^2} \textrm{ by (\ref{EQ:Lojasiewicz}),}\\
    & = ac(1 - \theta)\norme{\frac{\partial\tilde{\phi}_{\zeta_0}}{\partial t}(b,t)}_{L^2}.
\end{align*}
When we integrate this on an interval $[t_1,t_2]$, we obtain
\begin{align}\label{EQ:Lojasiewicz 2}
    \norme{\tilde{\phi}_{\zeta_0}(b,t_1) - \tilde{\phi}_{\zeta_0}(b,t_2)}_{L^2} & \leq \int_{t_1}^{t_2} \norme{\frac{\partial\tilde{\phi}_{\zeta_0}}{\partial s}(b,s)}_{L^2} \, ds \nonumber\\
    & \leq -\frac{1}{ac(1 - \theta)}\int_{t_1}^{t_2} \frac{\partial}{\partial s}(\tilde{f}_{\zeta_0}(\tilde{\phi}_{\zeta_0}(b,s))^{1 - \theta}) \, ds \nonumber\\
    & = \frac{-\tilde{f}_{\zeta_0}(\tilde{\phi}_{\zeta_0}(b,t_2))^{1 - \theta} + \tilde{f}_{\zeta_0}(\tilde{\phi}_{\zeta_0}(b,t_1))^{1 - \theta}}{ac(1 - \theta)}.
\end{align}
In particular,
\begin{align}
    \norme{\tilde{\phi}_{\zeta_0}(b,t)}_{L^2} & \leq \norme{\tilde{\phi}_{\zeta_0}(b,t) - \tilde{\phi}_{\zeta_0}(b,0)}_{L^2} + \norme{b}_{L^2} \nonumber\\
    & \leq \frac{\tilde{f}_{\zeta_0}(\tilde{\phi}_{\zeta_0}(b,0))^{1 - \theta} - \tilde{f}_{\zeta_0}(\tilde{\phi}_{\zeta_0}(b,t))^{1 - \theta}}{ac(1 - \theta)} + \norme{b}_{L^2} \nonumber\\
    & \leq \frac{\tilde{f}_{\zeta_0}(b)^{1 - \theta}}{ac(1 - \theta)} + \norme{b}_{L^2}. \label{EQ:Borne distance}
\end{align}
Let $r > 0$ be the radius of $B_2$ and chose $B_1$ small enough so
$$
\sup_{b \in B_1} \frac{\tilde{f}_{\zeta_0}(b)^{1 - \theta}}{ac(1 - \theta)} + \norme{b}_{L^2} \leq \frac{r}{2}.
$$
In this case, if $b \in B_1$, by (\ref{EQ:Borne distance}), $\tilde{\phi}_{\zeta_0}(b,t)$ stays in the closed ball of centre $0$ and radius $\frac{r}{2}$, which is compact. By the finite time explosion theorem, this flow is defined for all $t \geq 0$ in $B_2$.
\end{proof}

We end this sub-section with a classical convergence result (see \cite[Theorem 3.3]{GRS}) and a continuity result which is a weak version of the Duistermaat's theorem (see \cite[Theorem 1.1]{Lerman}). They both rely on the \Lo\ gradient inequality.

\begin{proposition}\label{PRO:Convergence flux}
    Let $b \in B_2$ and $\zeta \in \mU_d$. Assume that the flow $\phi_\zeta(b,\cdot)$ (resp. $\tilde{\phi}_\zeta(b,\cdot)$) exists for all $t$ and stays in a compact subset of $B_2$. Then, it converges toward a critical point of $f_\zeta$ (resp. $\tilde{f}_\zeta$). Moreover, there are constants $C > 0$ and $\epsilon > 0$ (depending on $b$ and $\zeta$) such that for all $t > 0$,
    $$
    \int_t^{+\infty} \norme{\frac{\partial\phi_\zeta(b,s)}{\partial s}}_{L^2} \, ds \leq \frac{C}{t^\epsilon} \qquad \left(\textrm{resp. } \int_t^{+\infty} \norme{\frac{\partial\tilde{\phi}_\zeta(b,s)}{\partial s}}_{L^2} \, ds \leq \frac{C}{t^\epsilon}\right).
    $$
\end{proposition}
\begin{proof}
Once again, we only prove it for $\tilde{\phi}_\zeta$ because the proof is the same for $\phi_\zeta$. Let $K \subset B_2$ a compact set where the flow stays at all time. Set $l = \lim_{t \rightarrow +\infty} \tilde{f}_\zeta(\tilde{\phi}_\zeta(b,t))$ which exists because this function is non-negative and non-increasing. Let $b_0$ be a critical point of $\tilde{f}_\zeta$ such that $\tilde{f}_\zeta(b_0) = l$. We use Proposition \ref{PRO:Lojasiewicz} near $b_0$ so there is an open neighbourhood $U_0$ of $b_0$ in $B_2$ such that, for all $b_1 \in U_0$,
$$
a(b_0)\abs{\tilde{f}_\zeta(b_1) - l}^{\theta(b_0)} \leq \norme{\nabla_{\tilde{\Omega}_\zeta}\tilde{f}_\zeta(b_1)}_{L^2},
$$
where $\frac{1}{2} \leq \theta(b_0) < 1$ and $a(b_0) > 0$. Of course, such an inequality is also true near any point $b_0$ where $\tilde{f}_\zeta(b_0) = l$. The level set $K_l = \{b_0 \in K|\tilde{f}_\zeta(b_0) = l\}$ is closed in $K$ thus compact. By compactness, we can find a neighbourhood $U$ of $K_l$ in $K$ and uniform constants $\frac{1}{2} \leq \theta' < 1$ and $a' > 0$ such that, for all $b_1 \in U$,
\begin{equation}\label{EQ:Lojasiewicz 3}
    a'\abs{\tilde{f}_\zeta(b_1) - l}^{\theta'} \leq \norme{\nabla_{\tilde{\Omega}_\zeta}\tilde{f}_\zeta(b_1)}_{L^2}.
\end{equation}
By properness of $\tilde{f}_\zeta$ on $K$, we may choose $U$ of the form
$$
U_\delta = \left\{b \in K|\abs{\tilde{f}_\zeta(b) - l} < \delta\right\},
$$
for some $\delta > 0$. In particular, $\tilde{f}_\zeta(\tilde{\phi}_\zeta(b,t)) \tend{t}{+\infty} l$ implies that there is a $T \geq 0$ such that for all $t \geq T$, $\tilde{\phi}_\zeta(b,t) \in U_\delta$. Similarly to the proof of (\ref{EQ:Lojasiewicz 2}), we show that, for all $t_2 \geq t_1 \geq T$,
\begin{align}
    \norme{\tilde{\phi}_\zeta(b,t_1) - \tilde{\phi}_\zeta(b,t_2)}_{L^2} & \leq \int_{t_1}^{t_2} \norme{\frac{\partial\tilde{\phi}_\zeta(b,s)}{\partial s}}_{L^2} \, ds \nonumber\\
    & \leq \frac{-(\tilde{f}_\zeta(\tilde{\phi}_\zeta(b,t_2)) - l)^{1 - \theta'} + (\tilde{f}_\zeta(\tilde{\phi}_\zeta(b,t_1)) - l)^{1 - \theta'}}{a'c(1 - \theta')} \nonumber\\
    & \leq \frac{(\tilde{f}_\zeta(\tilde{\phi}_\zeta(b,t_1)) - l)^{1 - \theta'}}{a'c(1 - \theta')}. \label{EQ:Lojasiewicz 4}
\end{align}
As $\tilde{\phi}_\zeta(b,t)$ stays in $K$ which is complete, we can use Cauchy sequences like arguments to say that $\lim_{t \rightarrow +\infty} \tilde{\phi}_\zeta(b,t)$ exists. This point must be a critical point of $\tilde{f}_\zeta$.

For the last inequality, assume without loss of generality that $\theta' > \frac{1}{2}$. We have, for all $t \geq T$,
\begin{align*}
    \frac{\partial}{\partial t}((\tilde{f}_\zeta(\tilde{\phi}_\zeta(b,t)) - l)^{1 - 2\theta'}) & = (1 - 2\theta')(\tilde{f}_\zeta(\tilde{\phi}_\zeta(b,t)) - l)^{-2\theta'}d\tilde{f}_\zeta(\tilde{\phi}_\zeta(b,t))\frac{\partial \tilde{\phi}_\zeta}{\partial t}(b,t)\\
    & = (2\theta' - 1)(\tilde{f}_\zeta(\tilde{\phi}_\zeta(b,t)) - l)^{-2\theta'}d\tilde{f}_\zeta(\tilde{\phi}_\zeta(b,t))\nabla_{\tilde{\Omega}_\zeta}\tilde{f}_\zeta(\tilde{\phi}_\zeta(b,t)) \textrm{ by (\ref{EQ:Flux de gradient}),}\\
    & \geq c(2\theta' - 1)(\tilde{f}_\zeta(\tilde{\phi}_\zeta(b,t)) - l)^{-2\theta'}\norme{\nabla_{\tilde{\Omega}_\zeta}\tilde{f}_\zeta(\tilde{\phi}_\zeta(b,t))}_{L^2}^2 \textrm{ by (\ref{EQ:Norme gradient tilde}),}\\
    & \geq {a'}^2c(2\theta' - 1) \textrm{ by (\ref{EQ:Lojasiewicz 3}).}
\end{align*}
Therefore, $\tilde{f}_\zeta(\tilde{\phi}_\zeta(b,t) - l)^{1 - 2\theta'} \geq {a'}^2c(2\theta' - 1)(t - T)$. Finally, for all $t \geq T$, by (\ref{EQ:Lojasiewicz 4}),
$$
\int_t^{+\infty} \norme{\frac{\partial\tilde{\phi}_\zeta(b,s)}{\partial s}}_{L^2} \, ds \leq \frac{(\tilde{f}_\zeta(\tilde{\phi}_\zeta(b,t)) - l)^{1 - \theta'}}{a'c(1 - \theta')} \leq \frac{({a'}^2c(2\theta' - 1)(t - T))^{\frac{1 - \theta'}{1 - 2\theta'}}}{a'c(1 - \theta')}.
$$
We deduce the wanted inequality with $\epsilon = \frac{1 - \theta'}{2\theta' - 1} > 0$ when $t \geq 2T$ for some $C > 0$. Up to adjusting the constant $C$, we may assume it is also true on the compact interval $[0,2T]$.
\end{proof}

\begin{proposition}\label{PRO:Duistermaat}
    Let $(\zeta,b_0)$ such that for all $b$ close enough to $b_0$ the flow $\phi_\zeta(b,t)$ (resp. $\tilde{\phi}_\zeta(b,t)$) exists at all time and stays in a compact subset of $B_2$ independent of $b$. Assume moreover that $\lim_{t \mapsto +\infty} \phi_\zeta(b_0,t)$ is a zero of $\mu_\zeta$ (resp. $\lim_{t \mapsto +\infty} \tilde{\phi}_\zeta(b_0,t)$ is a zero of $\tilde{\mu}_\zeta$). Then,
    $$
    b \mapsto \lim_{t \rightarrow +\infty} \phi_\zeta(b,t) \qquad \left(\textrm{resp. } b \mapsto \lim_{t \rightarrow +\infty} \tilde{\phi}_\zeta(b,t)\right)
    $$
    is continuous at $b_0$.
\end{proposition}
\begin{proof}
As for Proposition \ref{PRO:Convergence flux}, we only prove it for $\tilde{\phi}_\zeta$. Let $K \subset B_2$ be a compact set in which the flows $\tilde{\phi}_\zeta(b,t)$ stays at all time (when $b$ is close enough to $b_0$). As in the proof of Proposition \ref{PRO:Convergence flux}, there is a $\delta > 0$ such that for all $b$ in
$$
U_\delta = \left\{b \in K|\tilde{f}_\zeta(b) < \delta\right\},
$$
we have the \Lo\ gradient inequality,
$$
a'f_\zeta(b)^{1 - \theta'} \leq \norme{\nabla_{\tilde{\Omega}_\zeta}\tilde{f}_\zeta(b)}_{L^2}
$$
for some $a' > 0$ and $\frac{1}{2} \leq \theta' < 1$. Let $\varepsilon > 0$ that we assume to be smaller than $\delta$. Let $T > 0$ such that $\tilde{f}_\zeta(\tilde{\phi}_\zeta(b_0,T)) \leq \frac{\varepsilon}{2}$. By continuity of $b \mapsto \tilde{\phi}_\zeta(b,T)$, for all $b$ close enough to $b_0$ (in function of $\varepsilon$), $\tilde{f}_\zeta(\tilde{\phi}_\zeta(b,T)) \leq \varepsilon$. Then, since $\tilde{f}_\zeta$ non-increases with the flow, for all $b$ close enough to $b_0$ and for all $t \geq T$,
$$
\tilde{f}_\zeta(\tilde{\phi}_\zeta(b,t)) < \varepsilon.
$$
In particular, $\tilde{\phi}_\zeta(b,t) \in U_\delta$. As in the proof of Proposition \ref{PRO:Convergence flux}, we have, for all $b$ close enough to $b_0$ and for all $t \geq T$,
$$
\norme{\tilde{\phi}_\zeta(b,T) - \tilde{\phi}_\zeta(b,t)}_{L^2} \leq \frac{\tilde{f}_\zeta(\tilde{\phi}_\zeta(b,T))^{1 - \theta'}}{a'c(1 - \theta')} \leq \frac{\varepsilon^{1 - \theta'}}{a'c(1 - \theta')}.
$$
In particular,
\begin{align*}
    & \norme{\lim_{t \rightarrow +\infty} \tilde{\phi}_\zeta(b,t) - \lim_{t \rightarrow +\infty} \tilde{\phi}_\zeta(b_0,t)}_{L^2}\\
    \leq\ & \norme{\lim_{t \rightarrow +\infty} \tilde{\phi}_\zeta(b,t) - \tilde{\phi}_\zeta(b,T)}_{L^2} + \norme{\tilde{\phi}_\zeta(b,T) - \tilde{\phi}_\zeta(b_0,T)}_{L^2} + \norme{\tilde{\phi}_\zeta(b_0,T) - \lim_{t \rightarrow +\infty} \tilde{\phi}_\zeta(b_0,t)}_{L^2}\\
    \leq\ & \frac{\varepsilon^{1 - \theta'}}{a'c(1 - \theta')} + \varepsilon + \frac{\varepsilon^{1 - \theta'}}{a'c(1 - \theta')}\\
    \tend{\varepsilon}{0}\ & 0.
\end{align*}
We deduce the wanted continuity result.
\end{proof}

\subsection{Closure of an orbit and equivariant arc}\label{SEC:Arc équivariant semi-stable}

The following proposition gives sufficient conditions for a point in the closure of some $G$-orbit in $B$ to be reachable by a $1$-parameter sub-group and will be useful to construct locally $P$-destabilising sub-bundles thanks to Proposition \ref{PRO:Convergence exp(txi)b}. The proof shares similarities with the one of \cite[Theorem 12.5]{GRS}.

\begin{proposition}\label{PRO:Existence arc équivariant semi-stable}
    Let $(b,b_\infty) \in B_2^2$ and $g : \R_+ \rightarrow G$ a $\mC^1$ path in $G$ such that $g(t) \cdot b \tend{t}{+\infty} b_\infty$. Assume that $\Stab(b_\infty) = (\Stab(b_\infty) \cap K)^\C$. Then, there is a $\xi \in \i\frak{k}$ and a point $b' \in G \cdot b \cap B_2$ such that
    $$
    \e^{t\xi} \cdot b' \tend{t}{+\infty} b_\infty.
    $$
\end{proposition}
\begin{proof}
Let us define
$$
G_{b_\infty} = \Stab(b_\infty), \qquad K_{b_\infty} = G_{b_\infty} \cap K.
$$
Their Lie algebras are
$$
\frak{g}_{b_\infty} = \Lie(G_{b_\infty}) = \ker(L_{b_\infty}), \qquad \frak{k}_{b_\infty} = \Lie(K_{b_\infty}) = \frak{g}_{b_\infty} \cap \frak{k}.
$$
By assumption, $G_{b_\infty}$ is a reductive group with maximal compact sub-group $K_{b_\infty}$ so $\frak{g}_{b_\infty}$ is the complexification of $\frak{k}_{b_\infty}$. Let
$$
W = V \cap \im(L_{b_\infty})^\bot \subset V,
$$
and the germ of smooth map
$$
\Psi : \fonction{(\frak{g} \cap \frak{g}_{b_\infty}^\bot \times W,(0,0))}{(V,b_\infty)}{(\xi,c)}{\e^{\xi} \cdot (b_\infty + c)}.
$$
It verifies $d\Psi(0,0)(\xi,c) = L_{b_\infty}\xi + c$ which is an isomorphism. Therefore, we may apply the inverse function theorem. There is a $t_0 \in \R_+$ such that for all $t \geq t_0$,
$$
g(t) \cdot b = \Psi(\xi(t),c(t)) = \e^{\xi(t)} \cdot (b_\infty + c(t)),
$$
where $t \mapsto (\xi(t),c(t))$ is $\mC^1$ and converges toward $(0,0)$. For all $s,t \geq t_0$, we have
\begin{equation}\label{EQ:b_infty + c meme orbite}
    b_\infty + c(t) = f_s(t) \cdot (b_\infty + c(s)),
\end{equation}
where $f_s(t) = \e^{-\xi(t)}g(t)g(s)^{-1}\e^{\xi(s)} \in G$ is $\mC^1$.

Let us introduce, for all $c_0 \in W$,
$$
\psi_{c_0} : \fonction{\frak{g} \times W}{V}{(\eta,c)}{L_{b_\infty + c_0}\eta - c}, \qquad K_{c_0} = \{(\eta,L_{b_\infty + c_0}\eta)|\eta \in \frak{g}_{b_\infty}\}.
$$
The $\psi_{c_0}$ are a complex linear maps which vary continuously with $c_0$. $\psi_0$ is surjective and its kernel is $\ker(\psi_0) = \frak{g}_{b_\infty} \times \{0\} = K_0$, which has dimension $m = \dim(\frak{g}_{b_\infty})$. By openness of surjectivity, if $c_0$ is close enough to $0$, $\psi_{c_0}$ is surjective so $\dim(\ker(\psi_{c_0})) = \dim(\ker(\psi_0)) = m$.

Moreover, for all $c_0 \in W$ and all $\eta \in \frak{g}_{b_\infty}$, $L_{b_\infty + c_0}\eta = L_{c_0}\eta \in W$ because $W$ is preserved by the action of $G_{b_\infty}$. We deduce that $K_{c_0} \subset \ker(\psi_{c_0})$. By equality of dimensions (if $c_0$ is close enough to $0$), $\ker(\psi_{c_0}) = K_{c_0}$.

Choose $s \geq t_0$ large enough so for all $t \geq s$, $\ker(\psi_{c(t)}) = K_{c(t)}$. When we differentiate (\ref{EQ:b_infty + c meme orbite}), we obtain
$$
c'(t) = L_{f_s(t) \cdot (b_\infty + c(s))}(f_s'(t)f_s(t)^{-1}) = L_{b_\infty + c(t)}(f_s'(t)f_s(t)^{-1}).
$$
In other words, $(f_s'(t)f_s(t)^{-1},c'(t)) \in \ker(\psi_{c(t)}) = K_{c(t)}$. It implies that $f_s'(t)f_s(t)^{-1} \in \frak{g}_{b_\infty}$. Since $f_s(s) = \Id_E \in G_{b_\infty}$, we deduce that for all $t \geq s$, $f_s(t) \in G_{b_\infty}$. Therefore, (\ref{EQ:b_infty + c meme orbite}) becomes
$$
c(t) = f_s(t) \cdot c(s).
$$
It means that all the $c(t)$ lie in the same $G_{b_\infty}$-orbit and $c(t) \tend{t}{+\infty} 0$. By \cite[Theorem 12.4]{GRS}, there is a $\xi$ in $\i\frak{k}_{b_\infty} \subset \i\frak{k}$ and a $c = g \cdot c(s) \in G_{b_\infty} \cdot c(s)$ such that
$$
\e^{t\xi} \cdot c \tend{t}{+\infty} 0.
$$
Finally,
$$
b_\infty + c = g\e^{-\xi(s)}g(s) \cdot b \in G \cdot b, \qquad \e^{t\xi} \cdot (b_\infty + c) = b_\infty + \e^{t\xi} \cdot c \tend{t}{+\infty} b_\infty.
$$
\end{proof}

Proposition \ref{PRO:Existence arc équivariant semi-stable} will be useful in the semi-stable case \textit{i.e.} when $\lim_{t \rightarrow +\infty} \mu_\zeta(\phi_\zeta(b,t)) = 0$ (resp. $\lim_{t \rightarrow +\infty} \tilde{\mu}_\zeta(\tilde{\phi}_\zeta(b,t)) = 0$). However, in the unstable case, $\Stab(b_\infty)$ may not be the complexification of $\Stab(b_\infty) \cap K$ so we can't use the same argument. It may not even by reductive. We can still show an analogous result that will help us to construct destabilising sub-bundles, whose proof is widely inspired form \cite[Theorem 10.4]{GRS}.

\begin{figure}[H]
\begin{center}
\definecolor{ffqqqq}{rgb}{1.,0.,0.}
\definecolor{qqffqq}{rgb}{0.,1.,0.}
\definecolor{qqqqff}{rgb}{0.,0.,1.}
\begin{tikzpicture}[line cap=round,line join=round,>=triangle 45,x=2.0cm,y=2.0cm]
    \clip(3.4,2.708974100859924) rectangle (8.650268580288428,6.856040060418605);
    \draw [shift={(9.239457752051297,1.3529883221142884)},line width=1.2pt]  plot[domain=1.7767876634707285:2.7842273963268935,variable=\t]({1.*4.394950727636741*cos(\t r)+0.*4.394950727636741*sin(\t r)},{0.*4.394950727636741*cos(\t r)+1.*4.394950727636741*sin(\t r)});
    \draw [rotate around={-24.28582515565124:(4.669965877256986,3.590621540539471)},dash pattern=on 4pt off 4pt] (4.669965877256986,3.590621540539471) ellipse (1.394050025168909cm and 0.6784537665439491cm);
    \draw [->] (5.812788335551193,4.104992793109486) -- (6.0489002789461095,4.375559232641841);
    \draw [->] (6.6935863459367795,4.93546381845522) -- (6.991416839906913,5.129481928050854);
    \draw [->] (5.235243831249737,3.1645795494097055) -- (5.451406628850812,3.58149916713485);
    \draw [rotate around={-30.434235753443367:(5.6311028326411146,4.9172612817738965)},dash pattern=on 4pt off 4pt] (5.6311028326411146,4.9172612817738965) ellipse (1.4224871627243896cm and 0.6677084839304603cm);
    \draw [rotate around={-59.39359296849049:(7.079576835825635,5.891935377374708)},dash pattern=on 4pt off 4pt] (7.079576835825635,5.891935377374708) ellipse (1.277545266072021cm and 0.6205459792507906cm);
    \draw [->] (4.6882777246219804,3.9479617946362167) -- (4.68913886231099,5.3189808973181085);
    \draw [line width=1.2pt] (4.6882777246219804,3.9479617946362167)-- (4.69,6.69);
    \draw [->] (6.088540860088356,4.981931444916878) -- (6.089270430044178,5.71096572245844);
    \draw [->] (7.500383892824351,5.569997470433965) -- (7.500191946412175,5.784998735216982);
    \draw [line width=1.2pt] (6.088540860088356,4.981931444916878)-- (6.09,6.44);
    \draw [line width=1.2pt] (7.500383892824351,5.569997470433965)-- (7.5,6.);
    \begin{scriptsize}
    \draw [fill=black] (4.69,6.69) circle (1.5pt);
    \draw [fill=ffqqqq] (5.1216783528808785,2.8905580665344637) circle (1.5pt);
    \draw[color=ffqqqq] (5.2,2.85) node {$b$};
    \draw [fill=qqffqq] (8.340524907666067,5.655024077386459) circle (1.5pt);
    \draw[color=qqffqq] (8.5,5.6) node {$b_\infty$};
    \draw[color=black] (7.1,4.843136699931058) node {$t \mapsto \phi_\zeta(b,t)$};
    \draw [fill=ffqqqq] (5.314816586654761,3.3310640590944542) circle (1.5pt);
    \draw[color=ffqqqq] (5.6,3.3) node {$\phi_\zeta(b,s_0)$};
    \draw [fill=qqqqff] (4.6882777246219804,3.9479617946362167) circle (1.5pt);
    \draw[color=qqqqff] (4.810392691808178,4.01048992832376) node {$b_0$};
    \draw [fill=ffqqqq] (6.255251471898585,4.5794577200556095) circle (1.5pt);
    \draw[color=ffqqqq] (6.55,4.5) node {$\phi_\zeta(b,s_1)$};
    \draw [fill=ffqqqq] (7.423533365134435,5.355239003837678) circle (1.5pt);
    \draw[color=ffqqqq] (7.7,5.28) node {$\phi_\zeta(b,s_2)$};
    \draw [fill=qqqqff] (6.088540860088356,4.981931444916878) circle (1.5pt);
    \draw[color=qqqqff] (6.216999859377785,5.029067532425892) node {$b_1$};
    \draw[color=black] (4.05,5.5) node {$t \mapsto \mathbf{\mathrm{e}}^{- \mathbf{\mathrm{i}}t\mu_\zeta(b_\infty)} \cdot b_0$};
    \draw [fill=qqqqff] (7.500383892824351,5.569997470433965) circle (1.5pt);
    \draw[color=qqqqff] (7.623607026947393,5.635363725343828) node {$b_2$};
    \draw [fill=black] (6.09,6.44) circle (1.5pt);
    \draw [fill=black] (7.5,6.) circle (1.5pt);
    \draw[color=black] (5.45,5.9) node {$t \mapsto \mathbf{\mathrm{e}}^{- \mathbf{\mathrm{i}}t\mu_\zeta(b_\infty)} \cdot b_1$};
    \draw[color=black] (6.85,5.95) node {$t \mapsto \mathbf{\mathrm{e}}^{- \mathbf{\mathrm{i}}t\mu_\zeta(b_\infty)} \cdot b_2$};
    \draw [fill=qqqqff] (7.803196053448802,5.613993803138661) circle (1.5pt);
    \draw [fill=qqqqff] (8.026323516353221,5.649694197203368) circle (1.5pt);
    \draw [fill=qqqqff] (8.196005128979262,5.658619295719555) circle (1.5pt);
    \draw [fill=black] (7.8,5.92) circle (1.5pt);
    \draw [fill=black] (8.028708012946604,5.855011562366587) circle (1.5pt);
    \draw [fill=black] (8.200854320748324,5.765301514638931) circle (1.5pt);
    \draw [fill=ffqqqq] (7.827958759240762,5.515110676352515) circle (1.5pt);
    \draw [fill=ffqqqq] (8.043244684375145,5.582014947128078) circle (1.5pt);
    \draw [fill=ffqqqq] (8.194630749681167,5.6219376581304665) circle (1.5pt);
    \end{scriptsize}
\end{tikzpicture}
\end{center}
\caption{Illustration of Proposition \ref{PRO:Existence arc équivariant instable}. The dashed ellipses represent $K$-orbits.}
\end{figure}
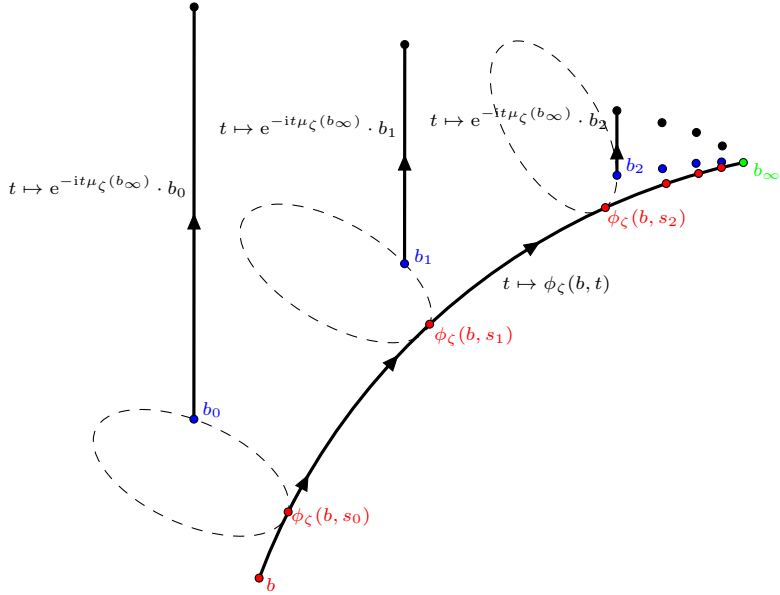

\begin{proposition}\label{PRO:Existence arc équivariant instable}
    Let $b \in B_2$ and $\zeta \in \mU_d$ such that $\phi_\zeta(b,t)$ (resp. $\tilde{\phi}_\zeta(b,t)$) exists in $B_2$ for all $t$ and converges toward some point $b_\infty \in B_2$ (resp. $\tilde{b}_\infty \in B_2$). Then there are sequences $s_k \rightarrow +\infty$ (resp. $\tilde{s}_k \rightarrow +\infty$) and $b_k \in K \cdot \phi_\zeta(b,s_k)$ (resp. $\tilde{b}_k \in K \cdot \tilde{\phi}_\zeta(b,\tilde{s}_k)$) such that
    $$
    \forall t \geq 0, \e^{-\i t\mu_\zeta(b_\infty)} \cdot b_k \in B_2 \qquad \left(\textrm{resp. } \e^{-\i t\tilde{\mu}_\zeta(b_\infty)} \cdot \tilde{b}_k \in B_2\right).
    $$
    Moreover, it converges in $B_2$ when $t \rightarrow +\infty$ and $b_k \rightarrow b_\infty$ (resp. $\tilde{b}_k \rightarrow \tilde{b}_\infty$).
\end{proposition}
\begin{proof}
As before, we only prove it for $\tilde{\phi}_\zeta$ because the proofs are rigorously the same.\\

\noindent\textbf{Step 1 :} Introducing notations.

Let $M = G/K$ and $\pi : G \rightarrow M$ the canonical projection. As $\scal{\cdot}{\cdot}$ on $\frak{k}$ is $K$-invariant, it induces a Riemannian structure $\scal{\cdot}{\cdot}_M$ on $M$, making it a homogeneous space with non-positive sectional curvature. We refer to \cite[Appendix A, Appendix C]{GRS} for more details. We call $\nabla_\LC$ the Levi--Civita connection on $M$.

Let $\tilde{g} : \R_+ \rightarrow G$ defined at Lemma \ref{LEM:Orbite préservée}. Recall that it satisfies
$$
\left\{
    \begin{array}{rcl}
        \tilde{g}(0) & = & \Id_E,\\
        \tilde{g}(t)^{-1}\tilde{g}'(t) & = & \i\tilde{\mu}_\zeta(\tilde{\phi}_\zeta(b,t)).
    \end{array}
\right.
$$
Moreover, for all $t$, $\tilde{\phi}_\zeta(b,t) = \tilde{g}(t)^{-1} \cdot b$.

For all $g \in G$, there is a unique $\xi \in \i\frak{k}$ such that $\pi(g) = \pi(\e^\xi)$. In this case, $\e^\xi$ is given by the Hermitian definite positive part of $g$ (on the left) in the polar decomposition of $g$. For all $t$, we write,
$$
\gamma(t) = \pi(\tilde{g}(t)) \in M
$$
and, when $0 \leq s \leq t$,
\begin{equation}\label{EQ:Définition xi(s,t) et u(s,t)}
    \tilde{g}(s)^{-1}\tilde{g}(t) = \e^{-\xi(s,t)}u(s,t), \qquad \textrm{so } \pi(\tilde{g}(s)^{-1}\tilde{g}(t)) = \pi(\e^{-\xi(s,t)}).
\end{equation}
When $0 \leq s \leq t$, we also denote by $\gamma_{s,t} : [s,t] \rightarrow M$ the unique geodesic such that $\gamma_{s,t}(s) = \gamma(s)$ and $\gamma_{s,t}(t) = \gamma(t)$. For all $s \leq r \leq t$, we have
\begin{equation}\label{EQ:Expression gamma_st}
    \gamma_{s,t}(r) = \pi\left(\tilde{g}(s)\exp\left(-\frac{r - s}{t - s}\xi(s,t)\right)\right).
\end{equation}
Let for all $0 \leq s \leq r \leq t$, $\rho_{s,t}(r) = d_M(\gamma_{s,t}(r),\gamma(r))$ where $d_M$ is the Riemannian distance on $M$.\\

\noindent\textbf{Step 2 :} There are constants $C > 0$ and $0 < \epsilon < 1$ such that for all $t > 0$,
$$
\int_t^{+\infty} \norme{\nabla_\LC\gamma'(s)}_M \, ds \leq \frac{C}{t^\epsilon}.
$$

Indeed, we have for all $t$, $\tilde{g}(t)^{-1}\tilde{g}'(t) = \i\tilde{\mu}_\zeta(\tilde{\phi}_\zeta(b,t))$. By \cite[Theorem C.1]{GRS}, it implies that
$$
\nabla_{\LC}\gamma'(t) = d\pi(\tilde{g}(t))\tilde{g}(t)\i d\tilde{\mu}_\zeta(\tilde{\phi}_\zeta(b,t))\frac{\partial\tilde{\phi}_\zeta}{\partial t}(b,t).
$$
Let $C_1 > 0$ be a constant that bounds $\norme{d\tilde{\mu}_\zeta(\tilde{\phi}_\zeta(b,t))}$ for all $t$, where $\norme{\cdot}$ is the operator norm from $(T_{\tilde{\phi}_\zeta(b,t)}B,L^2)$ to $(\frak{k},\scal{\cdot}{\cdot})$. $C_1$ exists because the flow converges thus stays in a compact set. We have
$$
\norme{\nabla_\LC\gamma'(t)}_M = \norme{\i d\tilde{\mu}_\zeta(\tilde{\phi}_\zeta(b,t))\frac{\partial\tilde{\phi}_\zeta}{\partial t}(b,t)} \leq C_1\norme{\frac{\partial\tilde{\phi}_\zeta}{\partial t}(b,t)}_{L^2}.
$$
The assumption of Step 2 follows from the inequality in Proposition \ref{PRO:Convergence flux}. We may always assume $\epsilon < 1$ up to increasing the constant $C$.\\

\noindent\textbf{Step 3 :} For all $0 < s < r < t$, if $\rho_{s,t}(r) \neq 0$, then $-\frac{C}{s^\epsilon} \leq \rho_{s,t}'(r) \leq \frac{C}{r^\epsilon}$.

Since $\rho_{s,t}(s) = 0$, there is a $s \leq r_1 < r$ such that $\rho_{s,t}(r_1) = 0$ and $\rho_{s,t}$ does not vanish on $]r_1,r]$. By \cite[Lemma A.3]{GRS}, $\rho_{s,t}$ is continuous and smooth at every point where it does not vanish, differentiable on the left and on the right at every point where it vanishes, and for all $u \in ]r_1,r]$,
$$
\rho_{s,t}''(u) \geq -\norme{\nabla_\LC\gamma'(u)}_M, \qquad \frac{d\rho_{s,t}}{dr^+}(r_1) \geq 0.
$$
Here, $\frac{d\rho_{s,t}}{dr^+}(r_1)$ denotes the derivative on the right at $r_1$ of $\rho_{s,t}$. Therefore, by integrating and by Step 2,
$$
\rho_{s,t}'(r) = \frac{d\rho_{s,t}}{dr^+}(r_1) + \int_{r_1}^r \rho_{s,t}''(u) \, du \geq -\int_{r_1}^r \norme{\nabla_\LC\gamma'(u)}_M \, du \geq -\frac{C}{r_1^\epsilon} \geq -\frac{C}{s^\epsilon}.
$$
Similarly, we can find a $r < r_2 \leq t$ such that $\rho_{s,t}(r_2) = 0$ and $\rho_{s,t}$ does not vanish on $[r,r_2[$. By \cite[Lemma A.3]{GRS}, $\frac{d\rho_{s,t}}{dr^-}(r_2) \leq 0$, so with the same arguments,
$$
\rho_{s,t}'(r) = \frac{d\rho_{s,t}}{dr^-}(r_2) - \int_r^{r_2} \rho_{s,t}''(u) \, du \leq \int_r^{r_2} \norme{\nabla_\LC\gamma'(u)}_M \, du \leq \frac{C}{r^\epsilon}.
$$

\noindent\textbf{Step 4 :} There is a constant $C' > 0$ such that for all $0 \leq s \leq r \leq t$, $\rho_{s,t}(r) \leq C'(r^{1 - \epsilon} - s^{1 - \epsilon})$.

It is clearly true for any $C'$ if $\rho_{s,t}(r) = 0$. Else, set $s \leq r_0 < r$ such that $\rho_{s,t}(r_0) = 0$ and $\rho_{s,t}$ does not vanish on $]r_0,r]$. Once more, it exists because $\rho_{s,t}(s) = 0$. By integrating the inequality of Step 3 over $]r_0,r]$, we obtain
$$
\rho_{s,t}(r) = \int_{r_0}^r \rho_{s,t}'(u) \, du \leq \int_{r_0}^r \frac{C}{u^\epsilon} \, du = \frac{C}{1 - \epsilon}(r^{1 - \epsilon} - r_0^{1 - \epsilon}) \leq C'(r^{1 - \epsilon} - s^{1 - \epsilon})
$$
with $C' = \frac{C}{1 - \epsilon}$.\\

\noindent\textbf{Step 5 :} For all $0 \leq s \leq t$, we set $\eta(s,t) = u(s,t)^\dagger\xi(s,t)u(s,t) \in \i\frak{k}$. We have
$$
\gamma_{s,t}'(s) = d\pi(\tilde{g}(s))\tilde{g}(s)\left(-\frac{\xi(s,t)}{t - s}\right), \qquad \gamma_{s,t}'(t) = d\pi(\tilde{g}(t))\tilde{g}(t)\left(-\frac{\eta(s,t)}{t - s}\right).
$$
The first equality comes from (\ref{EQ:Définition xi(s,t) et u(s,t)}). Using (\ref{EQ:Expression gamma_st}) and (\ref{EQ:Définition xi(s,t) et u(s,t)}), we have
\begin{align*}
    \gamma_{s,t}(r) & = \pi\left(\tilde{g}(s)\exp\left(-\frac{r - s}{t - s}\xi(s,t)\right)\right)\\
    & = \pi\left(\tilde{g}(t)u(s,t)^\dagger\e^{\xi(s,t)}\exp\left(-\frac{r - s}{t - s}\xi(s,t)\right)\right)\\
    & = \pi\left(\tilde{g}(t)\exp\left(-\frac{r - t}{t - s}u(s,t)^\dagger\xi(s,t)u(s,t)\right)u(s,t)^\dagger\right)\\
    & = \pi\left(\tilde{g}(t)\exp\left(-\frac{r - t}{t - s}\eta(s,t)\right)\right).
\end{align*}
We deduce the second equality.\\

\noindent\textbf{Step 6 :} For all $s' \geq s \geq 0$ and all $t \geq s' + 1$,
$$
\norme{\frac{\eta(s',t)}{t - s'} - \frac{\eta(s,t)}{t - s}} \leq \frac{C'({s'}^{1 - \epsilon} - s^{1 - \epsilon})}{t - s'}.
$$

Indeed,
\begin{align*}
    \norme{\frac{\eta(s',t)}{t - s'} - \frac{\eta(s,t)}{t - s}} & = \norme{d\pi(\tilde{g}(t))\tilde{g}(t)\frac{\eta(s',t)}{t - s'} - d\pi(\tilde{g}(t))\tilde{g}(t)\frac{\eta(s,t)}{t - s}}_M\\
    & = \norme{-\gamma_{s',t}'(t) + \gamma_{s,t}'(t)}_M \textrm{ by Step 5,}\\
    & \leq \frac{d_M(\gamma_{s',t}(s'),\gamma_{s,t}(s'))}{t - s'} \textrm{ by \cite[Lemma A.4]{GRS},}\\
    & = \frac{\rho_{s,t}(s')}{t - s'}\\
    & \leq \frac{C'({s'}^{1 - \epsilon} - s^{1 - \epsilon})}{t - s'} \textrm{ by Step 4.}
\end{align*}

\noindent\textbf{Step 7 :} For all $0 \leq s \leq t$,
$$
\norme{\frac{\eta(s,t)}{t - s} + \i\tilde{\mu}_\zeta(\tilde{\phi}_\zeta(b,t))} \leq \frac{C}{s^\epsilon}.
$$

Indeed,
\begin{align*}
    \norme{\frac{\eta(s,t)}{t - s} + \i\tilde{\mu}_\zeta(\tilde{\phi}_\zeta(b,t))} & = \norme{d\pi(\tilde{g}(t))\tilde{g}(t)\frac{\eta(s,t)}{t - s} + d\pi(\tilde{g}(t))\tilde{g}(t)\i\tilde{\mu}_\zeta(\tilde{\phi}_\zeta(b,t))}_M\\
    & = \norme{-\gamma_{s,t}'(t) + \gamma'(t)}_M \textrm{ by Step 5.}
\end{align*}
Then, the inequality clearly holds if $\gamma_{s,t}'(t) = \gamma'(t)$. Else, by \cite[Lemma A.3]{GRS}, we have
$$
\frac{d}{dr^-}\rho_{s,t}(t) = -\norme{-\gamma_{s,t}'(t) + \gamma'(t)}_M < 0.
$$
We deduce that in a left punctured neighbourhood of $t$, $\rho_{s,t} > 0$ and
$$
\norme{-\gamma_{s,t}'(t) + \gamma'(t)}_M = -\frac{d\rho_{s,t}}{dr^-}(t) = -\lim_{r \rightarrow t^-} \rho_{s,t}'(r).
$$
By Step 3, for all $r < t$ sufficiently close to $t$, $\rho_{s,t}'(r) \geq -\frac{C}{s^\epsilon}$ because $\rho_{s,t}(r) \neq 0$. We deduce the wanted inequality.\\

\noindent\textbf{Step 8 :} For all fixed $s \geq 0$, $\frac{\eta(s,t)}{t - s} \tend{t}{+\infty} -\i\tilde{\mu}(\tilde{b}_\infty)$.

We take $t > \max\{1,s^2\}$ and we apply Step 7 with $\sqrt{t} \rightarrow s$ and Step 6 with $\sqrt{t} \rightarrow s'$.
\begin{align*}
    \norme{\frac{\eta(s,t)}{t - s} + \i\tilde{\mu}_\zeta(\tilde{\phi}_\zeta(b,t))} & \leq \norme{\frac{\eta(s,t)}{t - s} - \frac{\eta(\sqrt{t},t)}{t - \sqrt{t}}} + \norme{\frac{\eta(\sqrt{t},t)}{t - \sqrt{t}} + \i\tilde{\mu}_\zeta(\tilde{\phi}_\zeta(b,t))}\\
    & \leq \frac{C'(\sqrt{t}^{1 - \epsilon} - s^{1 - \epsilon})}{t - \sqrt{t}} + \frac{C}{\sqrt{t}^\epsilon}\\
    & \tend{t}{+\infty} 0.
\end{align*}
Since $\tilde{\mu}_\zeta(\tilde{\phi}_\zeta(b,t)) \tend{t}{+\infty} \tilde{\mu}(\tilde{b}_\infty)$, we deduce the wanted result.\\

\noindent\textbf{Step 9 :} For all $s \geq 0$ and all $t' \geq t \geq s + 1$,
$$
\norme{\frac{\xi(s,t')}{t' - s} - \frac{\xi(s,t)}{t - s}} \leq \frac{C'}{t^\epsilon}.
$$
In particular, by completeness of $\i\frak{k}$, $\xi_\infty(s) = \lim_{t \rightarrow +\infty} \frac{\xi(s,t)}{t - s}$ exists.

Indeed,
\begin{align*}
    \norme{\frac{\xi(s,t')}{t' - s} - \frac{\xi(s,t)}{t - s}} & = \norme{d\pi(\tilde{g}(s))\pi(\tilde{g}(s))\frac{\xi(s,t')}{t' - s} - d\pi(\tilde{g}(s))\pi(\tilde{g}(s))\frac{\xi(s,t)}{t - s}}_M\\
    & = \norme{-\gamma_{s,t'}'(s) + \gamma_{s,t}'(s)}_M \textrm{ by Step 5,}\\
    & \leq \frac{d_M(\gamma_{s,t}(t),\gamma_{s,t'}(t))}{t - s} \textrm{ by \cite[Lemma A.4]{GRS},}\\
    & = \frac{\rho_{s,t'}(t)}{t - s}\\
    & \leq \frac{C'(t^{1 - \epsilon} - s^{1 - \epsilon})}{t - s} \textrm{ by Step 4,}\\
    & \leq \frac{C'}{t^\epsilon}.
\end{align*}

\noindent\textbf{Step 10 :} $\xi_\infty(s) \tend{s}{+\infty} -\i\tilde{\mu}_\zeta(\tilde{b}_\infty)$.

By Step 9 applied at $t = s + 1$ and $t' \rightarrow +\infty$, we have, for all $s \geq 0$
$$
\norme{\xi_\infty(s) - \xi(s,s + 1)} \leq \frac{C'}{(s + 1)^\epsilon}.
$$
Then,
\begin{align*}
    \norme{\xi(s,s + 1) + \i\tilde{\mu}_\zeta(\tilde{\phi}_\zeta(b,s))} & = \norme{d\pi(\tilde{g}(s))\tilde{g}(s)\xi(s,s + 1) + d\pi(\tilde{g}(s))\tilde{g}(s)\i\tilde{\mu}_\zeta(\tilde{\phi}_\zeta(b,s))}_M\\
    & = \norme{-\gamma_{s,s + 1}'(s) + \gamma'(s)}_M \textrm{ by Step 5.}
\end{align*}
Then, using the same reasoning as in Step 7, $\norme{-\gamma_{s,s + 1}'(s) + \gamma'(s)}_M = 0$ or
\begin{align*}
    \norme{-\gamma_{s,s + 1}'(s) + \gamma'(s)}_M & = \lim_{r \rightarrow s^+} \rho_{s,s + 1}'(r) \textrm{ by \cite[Lemma A.3]{GRS},}\\
    & \leq \frac{C}{s^\epsilon} \textrm{ by Step 3.}
\end{align*}
We deduce that $\xi_\infty(s) \tend{s}{+\infty} -\i\tilde{\mu}_\zeta(\tilde{b}_\infty)$.\\

\noindent\textbf{Step 11 :} Conclusion.

Recall that for all $t > s \geq 0$,
$$
\frac{\eta(s,t)}{t - s} = u(s,t)^\dagger\frac{\xi(s,t)}{t - s}u(s,t).
$$
Let $t_{s,p} \tend{p}{+\infty} +\infty$ such that $u_\infty(s) = \lim_{p \rightarrow +\infty} u(s,t_{s,p}) \in K$ exists. We then have, by Steps 8 and 9, $-\i\tilde{\mu}_\zeta(\tilde{b}_\infty) = u_\infty(s)^\dagger\xi_\infty(s)u_\infty(s)$. Let $\tilde{s}_k \rightarrow +\infty$ be a sequence such that $u_{\infty,\infty} = \lim_{k \rightarrow +\infty} u_\infty(\tilde{s}_k) \in K$ exists. By Step 10, we have
$$
-\i\tilde{\mu}_\zeta(\tilde{b}_\infty) = -\i u_{\infty,\infty}^\dagger\tilde{\mu}_\zeta(\tilde{b}_\infty)u_{\infty,\infty}.
$$
Let for all integer $k$, $\tilde{b}_k = u_{\infty,\infty}u_\infty(\tilde{s}_k)^\dagger \cdot \tilde{\phi}_\zeta(b,\tilde{s}_k) \in K \cdot \tilde{\phi}_\zeta(b,\tilde{s}_k)$. We clearly have $\tilde{b}_k \rightarrow \tilde{b}_\infty$. Let us fix $k$ and set for all $p$,
$$
\tilde{b}_{k,p} = u_{\infty,\infty}u(\tilde{s}_k,t_{\tilde{s}_k,p})^\dagger \cdot \tilde{\phi}_\zeta(b,\tilde{s}_k) \tend{p}{+\infty} \tilde{b}_k.
$$
We have
$$
\frac{1}{t_{\tilde{s}_k,p}}u_{\infty,\infty}\eta(\tilde{s}_k,t_{\tilde{s}_k,p})u_{\infty,\infty}^\dagger \tend{p}{+\infty} -\i u_{\infty,\infty}\tilde{\mu}_\zeta(\tilde{b}_\infty)u_{\infty,\infty}^\dagger = -\i\tilde{\mu}_\zeta(\tilde{b}_\infty)
$$
and
\begin{align*}
    \e^{u_{\infty,\infty}\eta(\tilde{s}_k,t_{\tilde{s}_k,p})u_{\infty,\infty}^\dagger} \cdot \tilde{b}_{k,p} & = u_{\infty,\infty}\e^{\eta(\tilde{s}_k,t_{\tilde{s}_k,p})}u_{\infty,\infty}^\dagger u_{\infty,\infty}u(\tilde{s}_k,t_{\tilde{s}_k,p})^\dagger \cdot \tilde{\phi}_\zeta(b,\tilde{s}_k)\\
    & = u_{\infty,\infty}\e^{\eta(\tilde{s}_k,t_{\tilde{s}_k,p})}u(\tilde{s}_k,t_{\tilde{s}_k,p})^\dagger \cdot \tilde{\phi}_\zeta(b,\tilde{s}_k)\\
    & = u_{\infty,\infty}\tilde{\phi}_\zeta(b,t_{\tilde{s}_k,p})\\
    & \tend{p}{+\infty} u_{\infty,\infty}\tilde{b}_\infty \in B_2.
\end{align*}
Therefore, we may apply Proposition \ref{PRO:Convergence exp(txi)b} with the sequences,
$$
\tilde{b}_{k,p} \tend{p}{+\infty} \tilde{b}_k, \qquad \frac{1}{t_{\tilde{s}_k,p}}u_{\infty,\infty}\eta(\tilde{s}_k,t_{\tilde{s}_k,p})u_{\infty,\infty}^\dagger \tend{p}{+\infty} -\i\tilde{\mu}_\zeta(\tilde{b}_\infty),
$$
to conclude that $t \mapsto \e^{-\i t\tilde{\mu}_\zeta(\tilde{b}_\infty)} \cdot \tilde{b}_k$ remains bounded and converges in $B_2$.
\end{proof}

\section{Proof of the main theorem}\label{SEC:Preuve du théorème}

Let $B_{2,0}$, $\mU_d$ and $\mB_d$ as determined in Section \ref{SEC:Résultats GIT locaux}. Let $B_2 \subset B_{2,0}$ and $B_1 \subset B_2$ also determined in Section \ref{SEC:Résultats GIT locaux}. We prove here Theorem \ref{THE:Déformation P-critique} for all $(\zeta,b) \in \mU_d \times B_1$ such that the flows $\phi_\zeta$ and $\tilde{\phi}_\zeta$ are defined at all time and stay in a compact subset of $B_2$.

\subsection{Equivalence of semi-stabilities}\label{SEC:Equivalence semi-stabilités}

We prove here the equivalence of Theorem \ref{THE:Déformation P-critique} about the different notions of semi-stabilities, still under the assumptions that $\phi_\zeta(b,\cdot)$ and $\tilde{\phi}_\zeta(b,\cdot)$ exist at all time and converge. We also show that in these cases, the limit points of these flows are zeroes of the associated moment maps.\\

\noindent\framebox{$1 \Rightarrow 3$} Assume that $\mu_\zeta$ vanishes in $\overline{G \cdot b} \cap B_2$. Let $b' \in G \cdot b \cap B_2$ be any point of the orbit of $b$. Let $F \subset E$ be an admissible sub-bundle of $\E_b$. Since $\E_{b'}$ is isomorphic to $\E_b$, with an isomorphism in $G$, there is an admissible sub-bundle $F' \subset E$ of $\E_{b'}$ such that $(F',\dbar_{b'})$ is isomorphic to $(F,\dbar_b)$. Let $\xi = -\rk({F'}^\bot)\Pi_{F'} + \rk(F')\Pi_{{F'}^\bot}$ where $\Pi_{F'}$ is the orthogonal projection onto $F'$ and $\Pi_{{F'}^\bot} = \Id_E - \Pi_{F'}$ is the orthogonal projection onto ${F'}^\bot$. $F'$ is admissible thus $\xi \in \i\frak{k}$. Let for all $t$,
$$
w(t) = \scal{\mu_\zeta(\e^{t\xi} \cdot b')}{\i\xi}.
$$
By Proposition \ref{PRO:Convergence exp(txi)b}, for all $t$, $\e^{t\xi} \cdot b' \in B_{3,0}$ and it converges toward some $b_\infty \in B_{3,0}$. Moreover, still by Proposition \ref{PRO:Convergence exp(txi)b}, $(F,\dbar_{b_\infty}) \subset \E_{b_\infty}$ and $(F^\bot,\dbar_{b_\infty}) \subset \E_{b_\infty}$ are holomorphic thus, in the orthogonal decomposition $E = F \oplus F^\bot$,
$$
\mP_\zeta(E,\dbar_{b_\infty},h) =
\begin{pmatrix}
    \mP_\zeta(F',\dbar_{b_\infty},h) & 0\\
    0 & \mP_\zeta({F'}^\bot,\dbar_{b_\infty},h)
\end{pmatrix}.
$$
Recall that $\mu_\zeta(b_\infty) = \Pi_{\frak{k}}\mu_{\infty,\zeta}(b_\infty)$ and $\i\xi \in \frak{k}$ so $\scal{\mu_\zeta(b_\infty)}{\i\xi} = \scal{\mu_{\infty,\zeta}(b_\infty)}{\i\xi}$. We deduce that
\begin{align*}
    \lim_{t \rightarrow +\infty} w(t) & = \scal{\mu_\zeta(b_\infty)}{\i\xi}\\
    & = -2\pi\int_X \tr(-\rk({F'}^\bot)\mP_\zeta(F',\dbar_{b_\infty},h) + \rk({F'})\mP_\zeta({F'}^\bot,\dbar_{b_\infty},h))\\
    & = 2\pi\rk({F'}^\bot)P_\zeta(F') - 2\pi\rk(F')P_\zeta({F'}^\bot)\\
    & = 2\pi(\rk(E) - \rk(F))P_\zeta(F) - 2\pi\rk(F)(P_\zeta(E) - P_\zeta(F))\\
    & = 2\pi\rk(E)P_\zeta(F).
\end{align*}
Moreover, for all $t$,
$$
w'(t) = \scal{d\mu_\zeta(\e^{t\xi} \cdot b')L_{\e^{t\xi} \cdot b'}\xi}{\i\xi} = \Omega_{\zeta,\e^{t\xi} \cdot b'}(\i L_{\e^{t\xi} \cdot b'}\xi,L_{\e^{t\xi} \cdot b'}\xi) \leq 0.
$$
We deduce that $2\pi\rk(E)P_\zeta(F) = \lim_{t \rightarrow +\infty} w(t) \leq w(0)$ and
$$
w(0) = \scal{\mu_\zeta(b')}{\i\xi} \leq \norme{\mu_\zeta(b')}\norme{\xi} = \norme{\mu_\zeta(b')}\sqrt{\rk(E)\rk(F)\rk(F^\bot)\Vol(X)},
$$
hence $\norme{\mu_\zeta(b')} \geq 2\pi\sqrt{\frac{\rk(E)}{\rk(F)\rk(F^\bot)\Vol(X)}}P_\zeta(F)$. It is true for all $b' \in G \cdot b \cap B_2$. By assumption, $\norme{\mu_\zeta}$ takes arbitrarily low positive values on this orbit so $P_\zeta(F) \leq 0$, proving the local $P_\zeta$-semi-stability.\\

\noindent\framebox{$2 \Rightarrow 3$} Assume that there is a $\dbar \in \overline{\mC^{d + 1}(\G^\C(E)) \cdot \dbar_b} \cap \mB_d$ such that $\mu_{\infty,\zeta}(\dbar) = 0$. Let $\dbar_m \rightarrow \dbar$ be a sequence of elements of $\mC^{d + 1}(\G^\C(E)) \cdot \dbar_b \cap \mB_d$ which converge toward $\dbar$. Let $\beta : \mB_d \rightarrow B_{3,0}$ be the smooth function given by Proposition \ref{PRO:Zéro dans la tranche déformée}. Recall that $\beta(\dbar_m) \in G \cdot b \cap B_{3,0}$ and by continuity of $\beta$, $b_\infty = \beta(\dbar) \in \overline{G \cdot b} \cap B_{3,0}$. Moreover, as $\mu_{\infty,\zeta}(\dbar) = 0$, we have $\tilde{\mu}_\zeta(b_\infty) = 0$.

It shows that $\tilde{\mu}_\zeta$ vanishes on $\overline{G \cdot b} \cap B_{3,0}$. We then apply the same method as for $1 \Rightarrow 3$ but we replace $\mu_\zeta$ by $\tilde{\mu}_\zeta$.\\

\noindent\framebox{$3 \Rightarrow 1$ and $2$} The proof is the same for $1$ and $2$ so we only prove it for $2$. Assume that $\E_b$ is locally $P_\zeta$-semi-stable. Let $b_\infty = \lim_{t \rightarrow +\infty} \tilde{\phi}_\zeta(b,t) \in B_2$. Let us show that $\tilde{\mu}_\zeta(b_\infty) = 0$.

Assume it is not the case. By Proposition \ref{PRO:Existence arc équivariant instable}, there is a $b' \in G \cdot b \cap B_2$ such that at all non-negative time, $\e^{-\i t\tilde{\mu}_\zeta(b_\infty)} \cdot b' \in B_2$. Let $\lambda_1 < \cdots < \lambda_m$ be the increasing eigenvalues of $-\i\tilde{\mu}_\zeta(b_\infty)$ and
$$
G_k = \ker(\lambda_k\Id_E + \i\tilde{\mu}_\zeta(b_\infty)), \qquad F_k = \bigoplus_{i = 1}^m G_i,
$$
the associated eigenspaces. Since $\tilde{\mu}_\zeta(b_\infty) \neq 0$ and $\tr(\tilde{\mu}_\zeta(b_\infty)) = 0$, we must have $m \geq 2$. By Proposition \ref{PRO:Convergence exp(txi)b}, the $F_k$ form a filtration of $\E_b$ by admissible sub-bundles. Moreover, each $(G_k,\dbar_{b_\infty})$ is a holomorphic sub-bundle of $\E_{b_\infty}$ because $L_{b_\infty}\tilde{\mu}_\zeta(b_\infty) = 0$. It is a consequence of Proposition \ref{PRO:Convergence flux} which says that $b_\infty$ must be a critical point of $\tilde{f}_\zeta$. In particular, we have the orthogonal and holomorphic decomposition
\begin{equation}\label{EQ:Décomposition E_binfty}
    \E_{b_\infty} = \bigoplus_{k = 1}^m (G_k,\dbar_{b_\infty}).
\end{equation}
Let for all $k$, $\Pi_k$ be the orthogonal projection onto $G_k$. We have
\begin{align*}
    \norme{\tilde{\mu}_\zeta(b_\infty)}^2 & = \sum_{k = 1}^m \scal{\tilde{\mu}_\zeta(b_\infty)}{\i\lambda_k\Pi_k}\\
    & = \sum_{k = 1}^m -2\pi\lambda_k\int_X \tr(\mP_\zeta(G_k,\dbar_{b_\infty},h)) \textrm{ by (\ref{EQ:Décomposition E_binfty}),}\\
    & = -2\pi\sum_{k = 1}^m \lambda_kP_\zeta(G_k)\\
    & = -2\pi\sum_{k = 1}^{m - 1} (\lambda_k - \lambda_{k + 1})P_\zeta(F_k) \textrm{ because } P_\zeta(E) = 0,\\
    & \leq 0 \textrm{ by local $P_\zeta$-semi-stability of $\E_{b'}$.}
\end{align*}
Since $\norme{\tilde{\mu}_\zeta(b_\infty)}^2 > 0$, we reach a contradiction. We deduce that $\tilde{\mu}_\zeta(b_\infty) = 0$. In particular, $\dbar_0 + \tilde{\Phi}(\zeta,b_\infty) \in \mC^{d + 1}(\G^\C(E)) \cdot \dbar_{b_\infty} \subset \overline{\mC^{d + 1}(\G^\C(E)) \cdot \dbar_b}$ is a solution to the $P_\zeta$-critical equation and by Proposition \ref{PRO:Zéro dans la tranche déformée}, it belongs to $\mB_d$. Therefore, $\E_b$ is locally $P_\zeta$-semi-stable in $\mB_d$.

\subsection{Equivalence of polystabilities}\label{SEC:Equivalence polystabilités}

We prove here the equivalence of Theorem \ref{THE:Déformation P-critique} about the different notions of polystabilities, still under the assumptions that the flows $\phi_\zeta(b,\cdot)$ and $\tilde{\phi}_\zeta(b,\cdot)$ exist in $B_2$ at all non-negative time and converge in $B_2$.\\

\noindent\framebox{$1 \Rightarrow 3$} Assume that $\mu_\zeta$ vanishes in $G \cdot b \cap B_2$. Let $b' \in G \cdot b \cap B_2 \subset B_{3,0}$ be a zero of $\mu_\zeta$. Since $\E_b$ is isomorphic to $\E_{b'}$ (with the isomorphism in $G$), $\E_b$ is locally $P_\zeta$-polystable if and only if $\E_{b'}$ is.

Let $F \subset E$ be an admissible sub-bundle of $\E_{b'}$ and $\xi = -\rk(F^\bot)\Pi_F + \rk(F)\Pi_{F^\bot}$ where $\Pi_F$ is orthogonal projection onto $F$ and $\Pi_{F^\bot} = \Id_E - \Pi_F$ is the orthogonal projection onto $F^\bot$. $F$ is admissible thus $\xi \in \i\frak{k}$. Let for all $t$,
$$
w(t) = \scal{\mu_\zeta(\e^{t\xi} \cdot b')}{\i\xi}.
$$
We have $w(0) = 0$ and by Proposition \ref{PRO:Convergence exp(txi)b}, for all $t$, $\e^{t\xi} \cdot b' \in B_{3,0}$ and it converges toward some $b_\infty \in B_{3,0}$. Moreover, $(F,\dbar_{b_\infty}) \subset \E_{b_\infty}$ and $(F^\bot,\dbar_{b_\infty}) \subset \E_{b_\infty}$ are holomorphic thus, in the orthogonal decomposition $E = F \oplus F^\bot$,
$$
\mP_\zeta(E,\dbar_{b_\infty},h) =
\begin{pmatrix}
    \mP_\zeta(F,\dbar_{b_\infty},h) & 0\\
    0 & \mP_\zeta(F^\bot,\dbar_{b_\infty},h)
\end{pmatrix}.
$$
Recall that $\mu_\zeta(b_\infty) = \Pi_{\frak{k}}\mu_{\infty,\zeta}(b_\infty)$ and $\i\xi \in \frak{k}$ so $\scal{\mu_\zeta(b_\infty)}{\i\xi} = \scal{\mu_{\infty,\zeta}(b_\infty)}{\i\xi}$. We deduce that
$$
\lim_{t \rightarrow +\infty} w(t) = \scal{\mu_\zeta(b_\infty)}{\i\xi} = 2\pi\rk(E)P_\zeta(F).
$$
Moreover, for all $t$,
$$
w'(t) = \scal{d\mu_\zeta(\e^{t\xi} \cdot b')L_{\e^{t\xi} \cdot b'}\xi}{\i\xi} = \Omega_{\zeta,\e^{t\xi} \cdot b'}(\i L_{\e^{t\xi} \cdot b'}\xi,L_{\e^{t\xi} \cdot b'}\xi) \leq 0
$$
with equality of and only if $L_{\e^{t\xi} \cdot b'}\xi = 0$. We deduce that $\lim_{t \rightarrow +\infty} w(t) \leq w(0) = 0$ with equality if and only if for all $t$, $L_{\e^{t\xi} \cdot b'}\xi = 0$. It implies that $P_\zeta(F) \leq 0$ and in case of equality, $\xi$ is $\dbar_{b'}$-holomorphic so $F^\bot = \ker(-\rk(F^\bot)\Id_E - \xi)$ is an admissible sub-bundle of $\E_{b'}$ hence the orthogonal and holomorphic decomposition $\E_{b'} = (F,\dbar_{b'}) \oplus (F^\bot,\dbar_{b'})$. It shows that $\E_{b'}$, thus $\E_b$ is locally $P_\zeta$-polystable.\\

\noindent\framebox{$2 \Rightarrow 3$} Assume that $\mu_{\infty,\zeta}$ vanishes on $\G^\C(E) \cdot b \cap \mB_d$ and let $\dbar \in \G^\C(E) \cdot b \cap \mB_d$ such that $\mu_{\infty,\zeta}(\dbar) = 0$. By Proposition \ref{PRO:Zéro dans la tranche déformée}, there is a $b' \in G \cdot b \cap B_{3,0}$ such that $\tilde{\mu}_\zeta(b') = 0$. We then apply the same method as for $1 \Rightarrow 3$ but with
$$
w(t) = \scal{\tilde{\mu}_\zeta(\e^{t\xi} \cdot b')}{\i\xi}.
$$

\noindent\framebox{$3 \Rightarrow 1$ and $2$} The proof is the same for $1$ and $2$ so we only prove it for $2$. Assume that $\E_b$ is locally $P_\zeta$-polystable. Let $b_\infty = \lim_{t \rightarrow +\infty} \tilde{\phi}_\zeta(b,t)$. We proved in Sub-section \ref{SEC:Equivalence semi-stabilités} that $\tilde{\mu}_\zeta(b_\infty) = 0$.

In this case, it is a usual result that $\Stab(b_\infty) = (\Stab(b_\infty) \cap K)^\C$. See for instance \cite[Lemma 2.3]{GRS}, we could also reproduce the proof of Lemma \ref{LEM:Aut réductif} and use Corollary \ref{COR:Sections dbar_b holomorphes}. By Proposition \ref{PRO:Existence arc équivariant semi-stable}, there are a $\xi \in \i\frak{k}$ and a point $b' \in G \cdot b \cap B_2$ such that
$$
\e^{t\xi} \cdot b' \tend{t}{+\infty} b_\infty,
$$
while staying in $B_2$ at all non-negative time. Let $\lambda_1 < \cdots < \lambda_m$ be the increasing eigenvalues of $\xi$ and
$$
G_k = \ker(\lambda_k\Id_E - \xi), \qquad F_k = \bigoplus_{i = 1}^k G_i,
$$
the associated eigenspaces. By Proposition \ref{PRO:Convergence exp(txi)b}, each $F_k$ is an admissible sub-bundle of $\E_{b'}$ so, by local $P_\zeta$-polystability, $P_\zeta(F_k) \leq 0$ and in case of equality, $F_k^\bot$ is admissible in $\E_{b'}$. Still by Proposition \ref{PRO:Convergence exp(txi)b}, we have an orthogonal and holomorphic decomposition
\begin{equation}\label{EQ:Décomposition E_binfty 2}
    \E_{b_\infty} = \bigoplus_{k = 1}^m (G_k,\dbar_{b_\infty}).
\end{equation}
Let for all $k$, $\Pi_k$ be the orthogonal projection onto $G_k$. We have
\begin{align*}
    \scal{\tilde{\mu}_\zeta(b_\infty)}{\i\xi} & = \sum_{k = 1}^m \scal{\tilde{\mu}_\zeta(b_\infty)}{\i\lambda_k\Pi_k}\\
    & = \sum_{k = 1}^m -2\pi\lambda_k\int_X \tr(\mP_\zeta(G_k,\dbar_{b_\infty},h)) \textrm{ by (\ref{EQ:Décomposition E_binfty 2}),}\\
    & = -2\pi\sum_{k = 1}^m \lambda_kP_\zeta(G_k)\\
    & = -2\pi\sum_{k = 1}^{m - 1} (\lambda_k - \lambda_{k + 1})P_\zeta(F_k) \textrm{ because } P_\zeta(E) = 0,\\
    & \leq 0 \textrm{ by local $P_\zeta$-semi-stability of $\E_{b'}$.}
\end{align*}
Moreover, this quantity vanishes because $\tilde{\mu}_\zeta(b_\infty) = 0$. We deduce that for all $k$, $P_\zeta(F_k) = 0$. By local $P_\zeta$-polystability of $\E_b$ (thus of $\E_{b'}$), each $(F_k^\bot,\dbar_{b'})$ is a holomorphic sub-bundle of $\E_{b'}$. Therefore, we have an orthogonal and holomorphic decomposition
$$
\E_{b'} = \bigoplus_{k = 1}^m (G_k,\dbar_{b'}).
$$
In particular, we deduce that $\xi$ is $\dbar_{b'}$-holomorphic so it stabilises $b'$ hence $b_\infty = b' \in G \cdot b \cap B_2$. We deduce that $\dbar_0 + \tilde{\Phi}(\zeta,b') \in \mC^{d + 1}(\G^\C(E)) \cdot \dbar_b$ is a solution to the $P_\zeta$-critical equation and by Proposition \ref{PRO:Zéro dans la tranche déformée}, it belongs to $\mB_d$. Therefore, $\E_b$ is locally $P_\zeta$-polystable in $\mB_d$.

\subsection{Equivalence of stabilities}\label{SEC:Equivalence stabilités}

We prove here the equivalence of Theorem \ref{THE:Déformation P-critique} about the different notions of stabilities, still under the assumptions that the flows $\phi_\zeta(b,\cdot)$ and $\tilde{\phi}_\zeta(b,\cdot)$ exist in $B_2$ at all non-negative time and converge in $B_2$.

Since we have already proven the equivalence between polystabilities in Sub-section \ref{SEC:Equivalence polystabilités}, all we have left to prove is that $\Stab(b)$ is trivial if and only if $\E_b$ is simple. It follows from Corollary \ref{COR:Sections dbar_b holomorphes}.

\subsection{Uniqueness of solutions and Jordan--Hölder filtration}

We show here a uniqueness result about solutions to the $P$-critical equation in each $\G^\C(E)$-orbit. The next lemma is a standard GIT result, analogous to \cite[Lemma 2.1]{GRS}.

\begin{lemma}\label{LEM:Unicité K-orbite zéro mu_zeta}
    Any two zeroes of $\mu_\zeta$ (resp. $\tilde{\mu}_\zeta$) in the same $G$-orbit in $B_{3,0}$ are in the same $K$-orbit.
\end{lemma}
\begin{proof}
The proof is the same for $\mu_\zeta$ and $\tilde{\mu}_\zeta$. We prove it for $\tilde{\mu}_\zeta$. Let $b_2 = u\e^\xi \cdot b_1$ with $u \in K$ and $\xi \in \i\frak{k}$ be zeroes of $\tilde{\mu}_\zeta$ in $B_{3,0}$. By Lemma \ref{LEM:Convexité}, all $\e^{t\xi} \cdot b_1$ lie in $B_{4,0}$ for $0 \leq t \leq 1$. Let
$$
a : t \mapsto \scal{\tilde{\mu}_\zeta(\e^{t\xi} \cdot b_1)}{\i\xi}.
$$
We have $a(0) = a(1) = 0$ because $\tilde{\mu}_\zeta(b_1) = \tilde{\mu}_\zeta(b_2) = 0$ and for all $t$,
$$
a'(t) = \scal{d\tilde{\mu}_\zeta(\e^{t\xi} \cdot b_1)L_{\e^{t\xi} \cdot b_1}\xi}{\i\xi} = \tilde{\Omega}_{\zeta,\e^{t\xi} \cdot b_1}(\i L_{\e^{t\xi} \cdot b_1}\xi,L_{\e^{t\xi} \cdot b_1}\xi) \leq 0,
$$
with equality if and only if $L_{\e^{t\xi} \cdot b_1}\xi = 0$. We deduce that $a$ is non-increasing. But $a(0) = a(1) = 0$ so it must be constant. In particular, $a'(0) = 0$ hence $L_{b_1}\xi = 0$, which means that $b_2 = u\e^\xi \cdot b_1 = u \cdot b_1 \in K \cdot b_1$.
\end{proof}

\begin{theorem}[Local uniqueness of solutions modulo $\G(E,h)$]\label{THE:Unicité solutions}
    Let $(\zeta,b) \in \mU_d \times B_2$. Two zeroes of $\mu_{\infty,\zeta}$ in $\mC^{d + 1}(\G^\C(E)) \cdot \dbar_b \cap \mB_d$ are in the same $\G(E,h)$-orbit.
\end{theorem}
\begin{proof}
Assume that $\dbar_1$ and $\dbar_2$ are in $\mC^{d + 1}(\G^\C(E)) \cdot \dbar_b \cap \mB_d$ and are $P_\zeta$-critical. By Proposition \ref{PRO:Zéro dans la tranche déformée}, there are $b_1 \in G \cdot b \cap B_{3,0}$, $b_2 \in G \cdot b \cap B_{3,0}$, $u_1 \in \G(E,h)$ and $u_2 \in \G(E,h)$ such that $\dbar_1 = u_1 \cdot (\dbar_0 + \tilde{\Phi}(\zeta,b_1))$ and $\dbar_2 = u_2 \cdot (\dbar_0 + \tilde{\Phi}(\zeta,b_2))$. Moreover, $\tilde{\mu}_\zeta(b_1) = \tilde{\mu}_\zeta(b_2) = 0$.

By Lemma \ref{LEM:Unicité K-orbite zéro mu_zeta}, there is a $v \in K$ such that $b_2 = v \cdot b_1$ so, by $K$-invariance of $\tilde{\Phi}(\zeta,\cdot)$, $\dbar_2 = u_2vu_1^\dagger \cdot \dbar_1$.
\end{proof}

We now show the existence of admissible Jordan--Hölder filtrations and the existence and uniqueness of the associated graded object in the semi-stable case. We then deduce informations about the $K$-orbits of the limit points of the flows in function of the $G$-orbits of the starting point.

\begin{theorem}[Jordan--Hölder filtration]\label{THE:FJH admissible}
    Up to shrinking $\mU_d$ and $B_1$, for all $b \in B_1$ such that $\E_b$ is locally $P_\zeta$-semi-stable, there is a filtration of $\E_b$ by admissible sub-bundles
    $$
    0 = F_0 \subsetneq F_1 \subsetneq \cdots \subsetneq F_m = E,
    $$
    such that for all $k$, $P_\zeta(F_k) = 0$ and each $\G_k = (F_k,\dbar_b)/(F_{k - 1},\dbar_b) = (F_k \cap F_{k - 1},\dbar_b)$ is locally $P_\zeta$-stable.
    
    Moreover, the graded object $\mathrm{Gr}(\E_b) = \bigoplus_{k = 1}^m \G_k$ is unique up to an isomorphism in $G$, and is isomorphic to both $\E_{b_\infty}$ and $\E_{\tilde{b}_\infty}$ where $b_\infty = \lim_{t \rightarrow +\infty} \phi_\zeta(b,t)$ and $\tilde{b}_\infty = \lim_{t \rightarrow +\infty} \tilde{\phi}_\zeta(b,t)$.
\end{theorem}

We prove this theorem under the assumption that the flows exist at all time and converge.

\begin{lemma}\label{LEM:FJH admissible}
    Theorem \ref{THE:FJH admissible} holds for $(\zeta,b)$ as long as the flows $\phi_\zeta(b,\cdot)$ and $\tilde{\phi}_\zeta(b,\cdot)$ exist at all time and converge in $B_2$.
\end{lemma}
\begin{proof}
Let us start with the existence. If $F_1 \subset F_2 \subset E$ are admissible sub-bundles of $\E_b$ with $\rk(F_1) = \rk(F_2)$, then $(F_1^\bot,\dbar_0)$ and $(F_2,\dbar_0)$ are sub-bundles of $\E_0$. We deduce that $(F_1^\bot \cap F_2,\dbar_0)$ is a sub-sheaf of $\E_0$ of rank $0$ hence it is trivial because $\E_0$ is torsion-free. It proves that in any filtration by admissible sub-bundles of $\E_b$, the ranks of the $F_k$ increase.

Therefore, all such filtration has at most $\rk(E)$ elements. To prove the existence of a Jordan--Hölder filtration by admissible sub-bundles, we can simply take a filtration by admissible sub-bundles $F_k$ such that $P_\zeta(F_k) = 0$ with a maximal number of elements and clearly, the quotients are locally $P_\zeta$-stable.

For the uniqueness of the graded object, the proof is the same as for \cite[Proposition 1.5.2]{Huybrechts_Lehn}. It relies on the fact that all maps between locally $P_\zeta$-stable bundles verifying $P_\zeta = 0$ is either $0$ or an isomorphism, which is true thanks to Lemma \ref{LEM:Aut réductif}, with the same proof that for Proposition \ref{PRO:Polysimplicité}.

Then, let $b_\infty = \lim_{t \rightarrow +\infty} \phi_\zeta(b,t)$ and $\tilde{b}_\infty = \lim_{t \rightarrow +\infty} \tilde{\phi}_\zeta(b,t)$. We show that $\E_{\tilde{b}_\infty}$ is isomorphic to the graded object because the proof is the same as for $\E_{b_\infty}$. By Sub-section \ref{SEC:Equivalence semi-stabilités}, local $P_\zeta$-semi-stability of $\E_b$ implies that $\tilde{\mu}_\zeta(\tilde{b}_\infty) = 0$. Thus, all holomorphic sections of $\E_{\tilde{b}_\infty}$ are parallel so by Proposition \ref{PRO:Existence arc équivariant semi-stable}, there are a $b' \in G \cap b$ and a $\xi \in \i\frak{k}$ such that
$$
\e^{t\xi} \cdot b' \tend{t}{+\infty} 0.
$$
Let $(G_k)$ be the eigenspaces of $\xi$ and $(F_k)$ the associated filtration by admissible sub-bundles of $\E_{b'}$. By Proposition \ref{PRO:Convergence exp(txi)b},
$$
\E_{\tilde{b}_\infty} = \bigoplus_{k = 1}^m (G_k,\dbar_{\tilde{b}_\infty}).
$$
Moreover, by Sub-section \ref{SEC:Equivalence polystabilités}, $\E_{\tilde{b}_\infty}$ is locally $P_\zeta$-polystable. We deduce that $(F_k)$ is a Jordan--Hölder filtration of $\E_{b'}$ and $\E_{\tilde{b}_\infty}$ is isomorphic to its graded object. In particular, it is isomorphic to the graded object of $\E_b$.
\end{proof}

\begin{corollary}\label{COR:Unicité de Ness}
    Let $b$ and $b'$ in $B_2$ in the same $G$-orbit such that $\E_b$ (thus $\E_{b'}$) is locally $P_\zeta$-semi-stable. Assume that the flows $\phi_\zeta(b,\cdot)$, $\tilde{\phi}_\zeta(b,\cdot)$, $\phi_\zeta(b',\cdot)$ and $\tilde{\phi}_\zeta(b',\cdot)$ exist at all time and converge in $B_2$. Let $b_\infty$, $\tilde{b}_\infty$, $b_\infty'$ and $\tilde{b}_\infty'$ their respective limits. Then, $b_\infty' \in K \cdot b_\infty$, $\tilde{b}_\infty' \in K \cdot \tilde{b}_\infty$ and $\tilde{b}_\infty \in G \cdot b_\infty$.
\end{corollary}
\begin{proof}
By Lemma \ref{LEM:FJH admissible}, $\E_{\tilde{b}_\infty}$ is isomorphic to $\E_{b_\infty}$ and the isomorphism between them lies in $G$ thus $\tilde{b}_\infty \in G \cdot b_\infty$. The proof of the two other assumptions are the same so we prove that $\tilde{b}_\infty' \in K \cdot \tilde{b}_\infty$. Indeed, since $\E_b \cong \E_{b'}$ with an isomorphism in $G$, their graded objects are isomorphic with an isomorphism in $G$ because we can find isomorphic Jordan--Hölder filtrations. It implies that $\tilde{b}_\infty' \in G \cdot \tilde{b}_\infty$.

We conclude thanks to Remark \ref{LEM:Unicité K-orbite zéro mu_zeta}.
\end{proof}

\subsection{Existence of flows and end of proofs}

We show in this sub-section the final argument to conclude the proof of Theorem \ref{THE:Déformation P-critique} \textit{i.e.} the fact that the flows $\phi_\zeta$ and $\tilde{\phi}_\zeta$ are defined at all time for all $b$ close enough to $0$ and all $\zeta$ close enough to $\zeta_0$.

Let $\zeta \in \mU_d$ and $b \in B_2$. Let $T \in \R_+^* \cup \{+\infty\}$ the largest positive number such that $\phi_\zeta(b,t)$ is well-defined for all $t < T$. If $T < +\infty$, for all $t \geq T$, we set $\phi_\zeta(b,t) = \bot$ where $\bot$ means "undefined". If $T = +\infty$ and $\phi_\zeta(b,t)$ stays in a compact subset of $B_2$, by Proposition \ref{PRO:Convergence flux}, we can set $\phi_\zeta(b,+\infty) = \lim_{t \rightarrow +\infty} \phi_\zeta(b,t)$ which exists in $B_2$. Finally, if $T = +\infty$ but $\phi_\zeta(b,t)$ does not stay in a compact subset of $B_2$, we set $\phi_\zeta(b,+\infty) = \bot$.

We do the same thing with $\tilde{\phi}_\zeta$. We endow the set $B_2 \cup \{\bot\}$ with the topology that makes $\bot$ an isolated point.

\begin{lemma}\label{LEM:Continuité en 0}
    Let $b_m \rightarrow 0$ and $\zeta_m \rightarrow \zeta_0$. Let $(t_m)$ be any sequence of $\R_+ \cup \{+\infty\}$. We have $\phi_{\zeta_m}(b_m,t_m) \rightarrow 0$ and $\tilde{\phi}_{\zeta_m}(b_m,t_m) \rightarrow 0$.
\end{lemma}
\begin{proof}
The proof is the same for both flows so we only prove it for $\tilde{\phi}_\zeta$. Assume it is not the case so up to an extraction, there is an $\varepsilon > 0$ such that for all $m$, $\tilde{\phi}_{\zeta_m}(b_m,t_m) = \bot$ or $\norme{\tilde{\phi}_{\zeta_m}(b_m,t_m)}_{L^2} > \varepsilon$. Let $r$ be the radius of $B_1$ and assume up to reducing $\varepsilon$ that $\varepsilon < r$. Assume up to an extraction that for all $m$, $\norme{b_m} \leq \varepsilon$. By continuity of $t \mapsto \tilde{\phi}_{\zeta_m}(b_m,t)$ as long as this flow is well-defined, there is a $t_m' \in [0,t_m[$ such that $\tilde{\phi}_{\zeta_m}(b_m,t_m') \neq \bot$ and
$$
\norme{\tilde{\phi}_{\zeta_m}(b_m,t_m')}_{L^2} = \varepsilon.
$$
In particular, $t_m' < +\infty$ and $\tilde{\phi}_{\zeta_m}(b_m,t_m') \in B_1$. By compactness of the sphere, we may also assume up to an extraction that $\tilde{\phi}_{\zeta_m}(b_m,t_m') \rightarrow b_\infty \in B_1$.

By Lemma \ref{LEM:Orbite préservée}, $\tilde{\phi}_{\zeta_m}(b_m,t_m') \in G \cdot b_m$. Moreover, by Propositions \ref{PRO:Définition et convergence des flux zeta_0} and \ref{PRO:Convergence flux}, $\tilde{\phi}_{\zeta_0}(b,t)$ is well-defined in $B_2$ for all $b \in B_1$ and all $t \in \R_+ \cup \{+\infty\}$. Moreover, $\tilde{\phi}_{\zeta_0}(b,+\infty)$ is a zero of $\tilde{\mu}_{\zeta_0}$. Therefore, we may apply Corollary \ref{COR:Unicité de Ness} to deduce that
\begin{equation}\label{EQ:Flux composés}
\tilde{\phi}_{\zeta_0}(\tilde{\phi}_{\zeta_m}(b_m,t_m'),+\infty) = u_m \cdot \tilde{\phi}_{\zeta_0}(b_m,+\infty) \in B_2
\end{equation}
for some $u_m \in K$. Moreover, by Proposition \ref{PRO:Duistermaat}, $\tilde{\phi}_{\zeta_0}(\cdot,+\infty)$ is continuous at $0$ and $b_\infty$ thus, when we take the limit of (\ref{EQ:Flux composés}), we obtain
\begin{equation}\label{EQ:Limit flux b_infty = 0}
\tilde{\phi}_{\zeta_0}(b_\infty,+\infty) = 0.
\end{equation}
Moreover, for all $m$, since $\tilde{f}_{\zeta_m}$ decreases along the flow lines,
$$
\tilde{f}_{\zeta_m}(\tilde{\phi}_{\zeta_m}(b_m,t_m')) \leq \tilde{f}_{\zeta_m}(b_m) \rightarrow \tilde{f}_{\zeta_0}(0) = 0.
$$
We deduce that
\begin{equation}\label{EQ:b_infty zéro de mu_zeta_0}
    \tilde{f}_{\zeta_0}(b_\infty) = \lim_{m \rightarrow +\infty} \tilde{f}_{\zeta_m}(\tilde{\phi}_{\zeta_m}(b_m,t_m')) = 0.
\end{equation}
Combining (\ref{EQ:Limit flux b_infty = 0}) and (\ref{EQ:b_infty zéro de mu_zeta_0}), we deduce that $b_\infty = 0$, hence a contradiction.
\end{proof}

\begin{corollary}\label{COR:Flux défini partout}
    Up to shrinking $\mU_d$ and $B_1$, the flows $\phi_\zeta$ and $\tilde{\phi}_\zeta$ exist at all time and stay in a compact subset $K$ of $B_2$ that we can choose independent of $\zeta$ and $b$.
\end{corollary}
\begin{proof}
Lemma \ref{LEM:Continuité en 0} shows that the functions $(\zeta,b,t) \mapsto \phi_\zeta(b,t),\tilde{\phi}_\zeta(b,t)$ are continuous at each point of $\{(\zeta_0,0)\} \times (\R_+ \cup \{+\infty\})$. Let $B \subsetneq B_2$ be any smaller ball so $K = \overline{B}$ is compact in $B_2$.

By continuity, for all $s \in \R_+ \cup \{+\infty\}$, there is an open neighbourhood $\mU_{d,s}$ of $\zeta_0$ in $\mU_d$, an open ball $B_{1,s}$ around $0$ and an open neighbourhood $I_s$ of $s$ in $\R_+ \cup \{+\infty\}$ such that $\phi_\zeta(b,t)$ and $\tilde{\phi}_\zeta(b,t)$ both lie in $B \subset K$ when $\zeta \in \mU_{d,s}$, $b \in B_{1,s}$ and $t \in I_s$. By compactness of $\R_+ \cup \{+\infty\}$, there is a finite family $s_1,\ldots,s_m$ of elements of this set such that $\R_+ \cup \{+\infty\} = \bigcup_{k = 1}^m I_{s_k}$.

Up to shrinking $\mU_d$ and $B_1$, we may assume that
$$
\mU_d \subset \bigcap_{k = 1}^m \mU_{d,s_k}, \qquad B_1 \subset \bigcap_{k = 1}^m B_{1,s_k}.
$$
It proves the corollary.
\end{proof}

\begin{corollary}
    Theorems \ref{THE:Déformation P-critique} and \ref{THE:FJH admissible} are true.
\end{corollary}
\begin{proof}
It follows from Sub-sections \ref{SEC:Equivalence semi-stabilités}, \ref{SEC:Equivalence polystabilités} and \ref{SEC:Equivalence stabilités} and Corollary \ref{COR:Flux défini partout} for Theorem \ref{THE:Déformation P-critique}. It follows from Lemma \ref{LEM:FJH admissible} and Corollary \ref{COR:Flux défini partout} for Theorem \ref{THE:FJH admissible}.
\end{proof}

Then, we prove Theorem \ref{THE:Cas zeta = zeta_0} and Proposition \ref{PRO:Topologies lieux stables}.

\begin{proof}[Proof of Theorem \ref{THE:Cas zeta = zeta_0}]
Let $b \in B_1$. By Propositions \ref{PRO:Définition et convergence des flux zeta_0} and \ref{PRO:Convergence flux} and (\ref{EQ:Lojasiewicz}), $b_\infty = \lim_{t \rightarrow +\infty} \phi_{\zeta_0}(b,t) \in \overline{G \cdot b} \cap B_2$ is a zero of $\mu_{\zeta_0}$, which proves the $\mu_{\zeta_0}$-semi-stability of $b$.

If $G \cdot b$ is closed, then $b_\infty \in G \cdot b \cap B_2$ so $b$ is $\mu_{\zeta_0}$-polystable. Reciprocally, assume that $b$ is $\mu_{\zeta_0}$-polystable. Let $F \subset E$ such that $(F,\dbar_0) \subset \E_0$ and $(F^\bot,\dbar_0) \subset \E_0$ are sub-bundles. Let $\xi = -\rk(F^\bot)\Pi_F + \rk(F)\Pi_{F^\bot} \i\frak{k}$, where $\Pi_F$ is the orthogonal projection onto $F$ and $\Pi_{F^\bot} = \Id_E - \Pi_F$ is the orthogonal projection onto $F^\bot$. Let $b' \in \R^*b \cap B_1$. If $(F,\dbar_b) \subset \E_b$, then by Proposition \ref{PRO:Convergence exp(txi)b}, $\e^{t\xi} \cdot b$ remains in $B_1$ for all $t$ and converges in it with a decreasing norm. Since the action of $G$ on $V$ is linear, $\e^{t\xi} \cdot b'$ remains in $B_1$ and converges in it with a decreasing norm. By Proposition \ref{PRO:Convergence exp(txi)b}, it implies that $(F,\dbar_{b'}) \subset \E_{b'}$ is holomorphic.

We just showed that all admissible sub-bundles of $\E_b$ are admissible sub-bundles of $\E_{b'}$. Therefore, $\E_{b'}$ is locally $P_{\zeta_0}$-polystable. Notice that $b'$ need not be in the gauge orbit of $b$. Let $b_0 \in \overline{G \cdot b}$. Still by linearity of the action, for all $\varepsilon > 0$, $\varepsilon b_0 \in \overline{G \cdot (\varepsilon b)}$. Choose $\varepsilon$ small enough so $\varepsilon b$ and $\varepsilon b_0$ both lie in $B_1$. As we saw previously, $\E_{\varepsilon b}$ is locally $P_{\zeta_0}$-polystable.

Let $b_{\infty,\varepsilon} = \lim_{t \rightarrow +\infty} \phi_{\zeta_0}(\varepsilon b,t) \in B_2$. By Sub-section \ref{SEC:Equivalence polystabilités}, $b_{\infty,\varepsilon} \in G \cdot (\varepsilon b)$ is a zero of $\mu_{\zeta_0}$. Moreover, by Corollary \ref{COR:Unicité de Ness}, for all $b' \in G \cdot (\varepsilon b) \cap B_1$, $\lim_{t \rightarrow +\infty} \phi_{\zeta_0}(b',t) \in K \cdot b_{\infty,\varepsilon}$. Moreover, by closeness of $K$-orbits and Proposition \ref{PRO:Duistermaat}, it is also true for $b' \in \overline{G \cdot (\varepsilon b)} \cap B_1$. In particular, for $b' = \varepsilon b_0$, $\lim_{t \rightarrow +\infty} \phi_{\zeta_0}(\varepsilon b_0,t) \in K \cdot b_{\infty,\varepsilon}$.

In particular, $\lim_{t \rightarrow +\infty} \phi_{\zeta_0}(\varepsilon b_0,t) \in G \cdot (\varepsilon b)$, which implies that $\varepsilon b \in \overline{G \cdot (\varepsilon b_0)}$. Since $\varepsilon b_0 \in \overline{G \cdot (\varepsilon b)}$, usual results about actions of reductive Lie groups imply that $\varepsilon b_0 \in G \cdot (\varepsilon b)$ and by linearity of the action, $b_0 \in G \cdot b$ proving that this orbit is closed.
\end{proof}

\begin{proof}[Proof of Proposition \ref{PRO:Topologies lieux stables}]
By Theorem \ref{THE:Déformation P-critique}, we can replace "$\mu_\zeta(b)$-(semi-)(poly)stable in $B_2$" by "locally $P_\zeta$-(semi-)(poly)stable" in the definitions of $L_{\mathrm{s}}$, $L_{\mathrm{ps}}$ and $L_{\mathrm{ss}}$. From here, the belonging of $(\zeta,b)$ to any of those sets only depend on $[\zeta]$ and $G \cdot b \cap B_1$ because for any $b' \in G \cdot b \cap B_1$, $\E_{b'}$ is isomorphic to $\E_b$ with an isomorphism in $G$ (so they have the same admissible sub-bundles) and for any bundle $F$, $P_\zeta(F)$ only depends on $[\zeta]$.\\

\noindent\textbf{Claim 1 :} $L_{\mathrm{s}} \subset \mU_d \times B_1$ is open.

Let $(\zeta_1,b_1) \in L_{\mathrm{s}}$ and $g \in G$ such that $g \cdot b_1 \in B_2$ and $\mu_{\zeta_1}(g \cdot b_1) = 0$. Let
$$
\Psi : \fonction{\mU_d \times B_1 \times \i\frak{k}}{\frak{k}}{(\zeta,b,s)}{\mu_{\zeta}(\e^sg \cdot b)}.
$$
We have for all $(v,w) \in \i\frak{k} \times \frak{k}$,
\begin{align*}
    \scal{\frac{\partial}{\partial s}|_{s = 0}\Psi(\zeta_1,b_1,s)v}{w} & = \scal{d\mu_{\zeta_1}(g \cdot b_1)L_{g \cdot b_1}v}{w}\\
    & = \Omega_{\zeta_1,g \cdot b_1}(L_{g \cdot b_1}w,L_{g \cdot b_1}v).
\end{align*}
By assumption, $\Stab(g \cdot b_1) = \Stab(b_1) = \{\Id_E\}$ so $L_{g \cdot b_1} : \frak{g} \rightarrow T_{g \cdot b_1}B_2$ is injective. Since $\Omega_{\zeta_1}$ is a positive form on $B_2$, then $\scal{\frac{\partial}{\partial s}|_{s = 0}\Psi(\zeta_1,b_1,s)v}{w} = 0$ for all $w$ if and only if $v = 0$. It shows the injectivity, thus the bijectivity by equality of dimensions, of the linearisation of $\Psi$ at $(\zeta_1,b_1,0)$ with respect to the third variable. We easily conclude with the implicit functions theorem.\\

\noindent\textbf{Claim 2 :} For all fixed $\zeta \in \mU_d$, $\{b \in B_1|(\zeta,b) \in L_{\mathrm{ss}} \subset B_1$ is open.

Let $(\zeta_1,b_1) \in L_{\mathrm{ss}}$. Let $b_\infty = \lim_{t \rightarrow +\infty} \phi_{\zeta_1}(b_1,t)$, which is a zero of $\mu_{\zeta_1}$. By applying Proposition \ref{PRO:Lojasiewicz}, we see that in an open neighbourhood $U \subset B_2$ of $b_\infty$, all the critical points of $f_{\zeta_1} = \frac{1}{2}\norme{\mu_{\zeta_1}}^2$ are zeroes of this function.

By Proposition \ref{PRO:Duistermaat}, if $b$ is close enough to $b_1$, $\lim_{t \rightarrow +\infty} \phi_{\zeta_1}(b,t) \in U$ and it is a critical point of $f_{\zeta_1}$ thus a zero of $\mu_{\zeta_1}$, proving the semi-stability of such a $b$.\\

\noindent\textbf{Claim 3 :} For all fixed $b \in B_1$, $\{\zeta \in \mU_d|(\zeta,b) \in L_{\mathrm{ss}}\} \subset \mU_d$ is a closed polyhedral cone centred at $0$ defined by a finite number of inequalities on $[\zeta - \zeta_0]$.

It is due to the fact that the number of admissible sub-bundles of each $\E_b$ up to diffeomorphism is finite and the inequalities $P_\zeta(F) \leq 0$ are linear in $[\zeta - \zeta_0]$.\\

\noindent\textbf{Claim 4 :} $\{\zeta \in \mU_d|(\zeta,b) \in L_{\mathrm{s}}\} \subset \mU_d$ is its interior.

It is a standard convex geometry result that the interior of a closed polyhedral cone defined by some linear inequalities $l_1(x),\ldots,l_m(x) \leq 0$ is the polyhedral cone defined by the inequalities $l_1(x),\ldots,l_m(x) < 0$. Here, local $P_\zeta$-semi-stability of $\E_b$ is defined by $P_\zeta(F) \leq 0$ for all $F \subset E$ admissible in $\E_b$ and local $P_\zeta$-stability of $\E_b$ is defined by $P_\zeta(F) < 0$ for all $F \subset E$ admissible in $\E_b$, hence the result.\\

\noindent\textbf{Claim 5 :} For all fixed $b \in B_1$, if $(\zeta_0,b) \in L_{\mathrm{ps}}$, $\{\zeta \in \mU_d|(\zeta,b) \in L_{\mathrm{ps}}\} = \{\zeta \in \mU_d|(\zeta,b) \in L_{\mathrm{ss}}\}$.

It is a consequence of the fact that if $(\zeta_0,b) \in L_{\mathrm{ps}}$, for all admissible sub-bundles $F \subset E$ of $\E_b$, we have $P_{\zeta_0}(F) = 0$ so $(F,\dbar_b)$ destabilises $\E_b$ at $\zeta_0$ thus $F^\bot$ is also admissible in $\E_b$. And for all $\zeta \in \mU_d$, $P_\zeta(F^\bot) = -P_\zeta(F)$.\\

\noindent\textbf{Claim 6 :} For all fixed $b \in B_1$, if $\E_b$ is simple, $\{\zeta \in \mU_d|(\zeta,0) \in L_{\mathrm{ps}}\} = \{\zeta \in \mU_d|(\zeta,0) \in L_{\mathrm{s}}\}$.

Trivial from the definition of local $P_\zeta$-stability.
\end{proof}

To conclude this section, we show a result which is analogous to the second Ness uniqueness theorem \cite[Theorem 6.5]{GRS} in the semi-stable case.

\begin{proposition}\label{PRO:Unicité dégénération polystable}
    Let $b \in B_1$ and $b_\infty \in \overline{G \cdot b} \cap B_1$ a zero of $\mu_\zeta$ or $\tilde{\mu}_\zeta$. Then, $\E_b$ is locally $P_\zeta$-semi-stable and $\E_{b_\infty}$ is isomorphic to the graded object of $\E_b$.
\end{proposition}
\begin{proof}
The local $P_\zeta$-semi-stability of $\E_b$ comes from Proposition \ref{PRO:Topologies lieux stables}. Assume without loss of generality that $\tilde{\mu}_\zeta(b_\infty) = 0$. Let $b_m \rightarrow b_\infty$ a sequence of points of $G \cdot b \cap B_1$. Let $b_\infty' = \lim_{t \rightarrow +\infty} \tilde{\phi}_\zeta(b,t)$. By Corollary \ref{COR:Unicité de Ness}, for all $m$,
$$
\lim_{t \rightarrow +\infty} \tilde{\phi}_\zeta(b_m,t) \in K \cdot b_\infty'.
$$
By Proposition \ref{PRO:Duistermaat} and closeness of $K$-orbits, when $m \rightarrow +\infty$,
$$
\lim_{t \rightarrow +\infty} \tilde{\phi}_\zeta(b_\infty,t) \in K \cdot b_\infty'.
$$
Since $\tilde{\mu}_\zeta(b_\infty) = 0$, we deduce that $b_\infty \in K \cdot b_\infty'$ and by Theorem \ref{THE:FJH admissible}, $\mathrm{Gr}(\E_b) \cong \E_{b_\infty'} \cong \E_{b_\infty}$.
\end{proof}

\section{Harder--Narasimhan filtrations and a local Kempf--Ness theorem}\label{SEC:Kempf--Ness}

\subsection{Harder--Narasimhan filtrations}

We build a local version of the Harder--Narasimhan filtration. We also show a link between this filtration and the limit points of the flows. For this, we need to be able to apply our previous result to all the destabilising sub-bundles of $\E_0$. Let $\F_0 = (F_0,\dbar_0) \subset \E_0$ be a non-zero direct sum of the simple components $\G_{0,k}$ of $\E_0$. In particular, we have an orthogonal and holomorphic decomposition
$$
\E_0 = \F_0 \oplus \E_0/\F_0
$$
and the smooth structure of $\E_0/\F_0$ is $F_0^\bot$. It induces an $L^2$-orthogonal decomposition of $V$,
\begin{align*}
    V & = H^{0,1}(X,\End(\E_0))\\
    & = H^{0,1}(X,\End(\F_0)) \oplus H^{0,1}(X,\Hom(\E_0/\F_0,\F_0))\\
    & \oplus H^{0,1}(X,\Hom(\F_0,\E_0/\F_0)) \oplus H^{0,1}(X,\End(\E_0/\F_0)).
\end{align*}
Let $W = H^{0,1}(X,\End(\F_0))$ so $W^\bot = H^{0,1}(X,\Hom(\E_0/\F_0,\F_0)) \oplus H^{0,1}(X,\Hom(\F_0,\E_0/\F_0)) \oplus H^{0,1}(X,\End(\E_0/\F_0))$. When $\zeta \in \mZ$ (recall that it means that $\zeta$ is closed and $P_\zeta(E) = 0$), we define
$$
\zeta^{F_0} = \zeta - \frac{P_\zeta(F_0)}{\rk(F_0)}\frac{\Vol}{\Vol(X)}.
$$
We have that $P_{\zeta^{F_0}}(F_0) = 0$ and the map $\zeta \mapsto \zeta^{F_0}$ is linear and sends $\zeta_0$ onto itself.

There is a finite number of such sub-bundles up to an isomorphism in $G$ because they are all isomorphic to direct sums of some $\G_{0,k}$. Therefore, if $B_2$ is small enough, we may assume, up to shrinking $\mU_d$ and $B_1$ that for all $(\zeta,b) \in \mU_d \times (B_1 \cap W)$, all the results of the previous sections apply to $\F_0$ with $(\zeta^{F_0},b)$.

\begin{theorem}[Harder--Narasimhan filtration]\label{THE:FHN admissible}
    For all $\zeta \in \mU_d$ and $b \in B_1$, there is a unique filtration of $\E_b$ by admissible sub-bundles
    $$
    0 = F_0 \subsetneq F_1 \subsetneq \cdots \subsetneq F_m = E,
    $$
    such that for all $k$, each $\G_k = (F_k,\dbar_b)/(F_{k - 1},\dbar_b) = (G_k,\dbar_b)$ (where $G_k = F_k \cap F_{k - 1}^\bot$) is locally $P_\zeta$-semi-stable and $\frac{P_\zeta(G_m)}{\rk(G_m)} < \frac{P_\zeta(G_{m - 1})}{\rk(G_{m - 1})} < \cdots < \frac{P_\zeta(G_1)}{\rk(G_1)}$.
    
    Moreover, if we set $b_\infty = \lim_{t \rightarrow +\infty} \phi_\zeta(b,t)$, $\tilde{b}_\infty = \lim_{t \rightarrow +\infty} \tilde{\phi}_\zeta(b,t)$ and $\Pi_k$ the orthogonal projection onto $G_k$, then, there are $u$ and $\tilde{u}$ in $K$ such that
    $$
    \i u\mu_\zeta(b_\infty)u^\dagger = \i\tilde{u}\tilde{\mu}_\zeta(\tilde{b}_\infty)\tilde{u}^\dagger = -\frac{2\pi}{\Vol(X)}\sum_{k = 1}^m \frac{P_\zeta(G_k)}{\rk(G_k)}\Pi_k
    $$
    and the direct sum of the graded objects of the semi-stable quotients $\bigoplus_{k = 1}^m \mathrm{Gr}(\G_k)$ is isomorphic to both $\E_{b_\infty}$ and $\E_{\tilde{b}_\infty}$.
\end{theorem}
\begin{proof}
For the uniqueness, the proof is the same as for \cite[Theorem 1.3.4]{Huybrechts_Lehn}. We could also show the existence the same way, but we will directly prove it thanks to the flows. We prove the statement of the theorem for $\tilde{b}_\infty = \lim_{t \rightarrow +\infty} \tilde{\phi}_\zeta(b,t)$ since the proof for the limit point of the other flow is the same.

Let $\lambda_1 < \cdots < \lambda_m$ be the distinct eigenvalues of $-\i\tilde{\mu}_\zeta(\tilde{b}_\infty)$ and
$$
G_k = \ker(\lambda_k\Id_E + \i\tilde{\mu}_\zeta(\tilde{b}_\infty)), \qquad F_k = \bigoplus_{i = 1}^k G_i,
$$
the associated eigenspaces. Let for all $k$, $\Pi_k$ be the orthogonal projection onto $G_k$ so we have
$$
-\i\tilde{\mu}_\zeta(\tilde{b}_\infty) = \sum_{k = 1}^m \lambda_k\Pi_k.
$$
By Proposition \ref{PRO:Convergence flux}, $\tilde{\mu}_\zeta(\tilde{b}_\infty)$ is $\dbar_{\tilde{b}_\infty}$-holomorphic so we have an orthogonal and holomorphic decomposition
$$
\E_{\tilde{b}_\infty} = \bigoplus_{k = 1}^m (G_k,\dbar_{\tilde{b}_\infty}).
$$
We deduce that for all $k$, on the one hand
$$
\scal{\tilde{\mu}_\zeta(\tilde{b}_\infty)}{\i\Pi_k} = -2\pi\int_X \tr(\mP_\zeta(G_k,\dbar_{\tilde{b}_\infty},h)) = -2\pi P_\zeta(G_k),
$$
and on the other hand
$$
\scal{\tilde{\mu}_\zeta(\tilde{b}_\infty)}{\i\Pi_k} = \sum_{i = 1}^m \lambda_i\scal{\i\Pi_i}{\i\Pi_k} = \lambda_k\norme{\Pi_k}^2 = \lambda_k\rk(G_k)\Vol(X).
$$
Thus $\lambda_k = -\frac{2\pi}{\Vol(X)}\frac{P_\zeta(G_k)}{\rk(G_k)}$. In particular, $\frac{P_\zeta(G_k)}{\rk(G_k)}$ decreases with $k$ and
$$
\i\tilde{\mu}_\zeta(\tilde{b}_\infty) = -\frac{2\pi}{\Vol(X)}\sum_{k = 1}^m \frac{P_\zeta(G_k)}{\rk(G_k)}\Pi_k.
$$
Then, by Proposition \ref{PRO:Existence arc équivariant instable}, there is a sequence $\tilde{b}_i \in G \cdot b$ which converges toward $\tilde{b}_\infty$ and such that $t \mapsto \e^{-t\tilde{\mu}_\zeta(\tilde{b}_\infty)} \cdot \tilde{b}_i$ converges in $B_2$. Let for all $i$, $\tilde{b}_{\infty,i} = \lim_{t \rightarrow +\infty} \e^{-t\tilde{\mu}_\zeta(\tilde{b}_\infty)} \cdot \tilde{b}_i \in B_2$.

The map $b' \mapsto \lim_{t \rightarrow +\infty} \e^{-t\tilde{\mu}_\zeta(\tilde{b}_\infty)} \cdot b'$ is linear where it is defined as a limit of linear actions. In particular, it is continuous and since it fixes $\tilde{b}_\infty$ (because $\tilde{b}_\infty$ is a critical point of $\frac{1}{2}\norme{\tilde{\mu}_\zeta}^2$), we have $\tilde{b}_{\infty,i} \rightarrow \tilde{b}_\infty$.

The fact that $\tilde{\mu}_\zeta(\tilde{b}_\infty)$ acts trivially on each $\tilde{b}_{\infty,i}$ implies that each $\tilde{b}_{\infty,i}$ is block diagonal in the orthogonal decomposition $E = \bigoplus_{k = 1}^m G_k$ and each $\E_{\tilde{b}_{\infty,i}}$ decomposes as an orthogonal and holomorphic direct sum
$$
\E_{\tilde{b}_{\infty,i}} = \bigoplus_{k = 1}^m (G_k,\dbar_{\tilde{b}_{\infty,i}}).
$$
When looking at each block, we see that
\begin{align*}
    \mP_\zeta(G_k,\dbar_0 + \tilde{\Phi}(\zeta,\tilde{b}_{\infty,i}),h) & = -\frac{1}{2\i\pi}\mu_{\infty,\zeta}(\dbar_0 + \tilde{\Phi}(\zeta,\tilde{b}_{\infty,i})_{|\End(G_k,\dbar_0)}\,\Vol\\
    & \tend{i}{+\infty} -\frac{1}{2\i\pi}\mu_{\infty,\zeta}(\dbar_0 + \tilde{\Phi}(\zeta,\tilde{b}_{\infty})_{|\End(G_k,\dbar_0)}\,\Vol\\
    & = -\frac{1}{2\pi}\lambda_k\Id_{G_k}\Vol\\
    & = \frac{P_\zeta(G_k)}{\rk(G_k)}\Id_{G_k}\frac{\Vol}{\Vol(X)}.
\end{align*}
We deduce that $\dbar_0 + \tilde{\Phi}(\zeta,\tilde{b}_{\infty,i})$ converges toward a $P_{\zeta^{G_k}}$-critical operator where we recall that $\zeta^{G_k} = \zeta - \frac{P_\zeta(G_k)}{\rk(G_k)}\frac{\Vol}{\Vol(X)}\Id_{G_k}$.

By Proposition \ref{PRO:Topologies lieux stables}, being $\mu_{\infty,\zeta^{G_k}}$-semi-stable in $\mB_d$ is an open condition so if $i$ is large enough, $(G_k,\dbar_{\tilde{b}_{\infty,i}})$ is $\mu_{\infty,\zeta^{G_k}}$-semi-stable in $\mB_d$. By Theorem \ref{THE:Déformation P-critique}, it is locally $P_{\zeta^{G_k}}$-semi-stable. We easily see that it is equivalent to being locally $P_\zeta$-semi-stable.

Moreover, by Proposition \ref{PRO:Convergence exp(txi)b}, the existence of $\tilde{b}_{\infty,i} = \lim_{t \rightarrow +\infty} \e^{-\i t\tilde{\mu}_\zeta(\tilde{b}_\infty)} \cdot \tilde{b}_i$ implies that the filtration
$$
0 = F_0 \subsetneq \cdots \subsetneq F_m = E
$$
is a filtration by admissible sub-bundles in $\E_{\tilde{b}_i}$ and
$$
(F_k,\dbar_{\tilde{b}_i}) \cap (F_{k - 1},\dbar_{\tilde{b}_i}) = (G_k,\dbar_{\tilde{b}_{\infty,i}})
$$
is locally $P_\zeta$-semi-stable. Moreover, as the $\E_{\tilde{b}_i}$ are all isomorphic (with isomorphisms in $G$) and by uniqueness of the Harder--Narasimhan filtration, the filtrations $(F_k,\dbar_{\tilde{b}_i})$ are isomorphic. In particular, the $(G_k,\dbar_{\tilde{b}_{\infty,i}})$ are isomorphic and they approach $(G_k,\dbar_{\tilde{b}_\infty})$ which is locally $P_{\zeta^{G_k}}$-polystable. We may apply Proposition \ref{PRO:Unicité dégénération polystable} to deduce that
$$
(G_k,\dbar_{\tilde{b}_\infty}) = \mathrm{Gr}(G_k,\dbar_{\tilde{b}_{\infty,i}}).
$$
Let $i$ be any integer. We have $b \in G \cdot \tilde{b}_i$ so the Harder--Narasimhan filtration of $\E_{\tilde{b}_i}$ yields a Harder--Narasimhan filtration $(F_k')_{1 \leq k \leq m}$ of $\E_b$ which is isomorphic to the previous one. In particular, we can build a unitary $\tilde{u} \in K$ such that for all $k$, $F_k' = \tilde{u}(F_k)$. The orthogonal projection onto $G_k' = F_k' \cap {F_{k - 1}'}^\bot$ is $\Pi_k' = \tilde{u}\Pi_k\tilde{u}^\dagger$ so
$$
-\frac{2\pi}{\Vol(X)}\sum_{k = 1}^m \frac{P_\zeta(G_k')}{\rk(G_k')}\Pi_k' = -\frac{2\pi}{\Vol(X)}\sum_{k = 1}^m \frac{P_\zeta(G_k)}{\rk(G_k)}\tilde{u}\Pi_k\tilde{u}^\dagger = -\i\tilde{u}\tilde{\mu}_\zeta(\tilde{b}_\infty)\tilde{u}^\dagger.
$$
\end{proof}

\subsection{A local Kempf--Ness theorem and continuity results}

We end up this section by showing the continuity results about the solutions to the $P$-critical equation and a local version of the Kempf--Ness theorem. Recall that
$$
L_{\mathrm{ss}} = \{(\zeta,b) \in \mU_d \times B_1|b \textrm{ is $\mu_\zeta$-semi-stable in } B_2\}.
$$

\begin{theorem}[Continuity of solutions of the $P$-critical equation modulo unitary transforms]\label{THE:Continuité des opérateurs P-critiques}
    Up to shrinking $\mU_d$ and $B_1$, the maps
    $$
    \Gamma : \fonction{L_{\mathrm{ss}}}{B_2/K}{(\zeta,b)}{K \cdot \left(\lim_{t \rightarrow +\infty} \phi_\zeta(b,t)\right)},
    $$
    $$
    \Gamma_\infty : \fonction{L_{\mathrm{ss}}}{\mB_d/\mC^{d + 1}(\G(E,h))}{(\zeta,b)}{\mC^{d + 1}(\G(E,h)) \cdot \left(\dbar_0 + \tilde{\Phi}(\zeta,\lim_{t \rightarrow +\infty} \tilde{\phi}_\zeta(b,t))\right)}
    $$
    are well-defined, $G$-invariant and continuous.
    
    For each $(\zeta,b)$, $\Gamma(\zeta,b)$ is the orbit of a zero of $\mu_\zeta$ in $\overline{G \cdot b} \cap B_2$ and $\Gamma_\infty(\zeta,b)$ is the orbit of a $P_\zeta$-critical operator in $\overline{\mC^{d + 1}(\G^\C(E)) \cdot \dbar_b} \cap \mB_d$. Moreover, if $(\zeta,b) \in L_{\mathrm{ps}}$, $\Gamma(\zeta,b) \subset G \cdot b$ and $\Gamma_\infty(\zeta,b) \subset \mC^{d + 1}(\G^\C(E)) \cdot \dbar_b$.
\end{theorem}
\begin{proof}
$\Gamma$ and $\tilde{\Gamma}$ are well-defined by Corollary \ref{COR:Flux défini partout} and they are $G$-invariant by Corollary \ref{COR:Unicité de Ness}. By the results of Sub-sections \ref{SEC:Equivalence semi-stabilités} and \ref{SEC:Equivalence polystabilités}, the second part of the theorem is true. All we have left to show is their continuity.

By Corollary \ref{COR:Flux défini partout} applied to $B_1 \subset B_{2,0}$, up to shrinking $\mU_d$, there is a ball $B_1' \subset B_1$ such that for all $\zeta \in \mU_d$, the flows $\phi_\zeta$ and $\tilde{\phi}_\zeta$ are well-defined and converge in a compact subset of $B_1$ (independent of $\zeta$ and $b$) as long as the starting point is in $B_1'$. Let us show the continuity of $\Gamma$ and $\Gamma_\infty$ on $L_{\mathrm{ss}} \cap (\mU_d \times B_1')$. Let
$$
\tilde{\Gamma} : \fonction{L_{\mathrm{ss}}}{B_2/K}{(\zeta,b)}{K \cdot \left(\lim_{t \rightarrow +\infty} \tilde{\phi}_\zeta(b,t)\right)}.
$$
By continuity of $\tilde{\Phi}$ (see Proposition \ref{PRO:Tranche déformée P}), the continuity of $\tilde{\Gamma}$ at some point implies the continuity of $\Gamma_\infty$ at the same point. The proof of the continuity of $\Gamma$ is the same as the proof of the continuity of $\tilde{\Gamma}$ so we only do it for $\tilde{\Gamma}$.

Let $(\zeta_1,b_1) \in L_{\mathrm{ss}}$ with $b_1 \in B_1'$ and $K \cdot b_\infty = \tilde{\Gamma}(\zeta_1,b_1)$ which is in $B_1/K$ by construction of $B_1'$. Let $(\zeta_m,b_m)$ be a sequence of elements of $L_{\mathrm{ss}}$ that converges toward $(\zeta_1,b_1)$. Since the $\tilde{\Gamma}(\zeta_m,b_m)$ stay in a compact subset of $B_1$ independent of $m$, it is enough to show that $K \cdot b_\infty$ is its only limit up to an extraction. Let $K \cdot b_\infty'$ with $b_\infty' \in B_1$ be a limit of this sequence up to an extraction.

Set for all $m$, $b_{\infty,m} = \lim_{t \rightarrow +\infty} \tilde{\phi}_{\zeta_m}(b_m,t) \in B_1$. Up to an extraction, $(b_{\infty,m})$ converges in $\overline{B_1}$ and its limit must lie in $K \cdot b_\infty'$. Up to replacing $b_\infty'$ by an element of its $K$-orbit, we may assume that $b_{\infty,m} \rightarrow b_\infty'$. Let for all $m$, $b_m' = \tilde{\phi}_{\zeta_m}(b_m,t_m)$ with $t_m < +\infty$ chosen large enough so the distance between $b_m'$ and $b_{\infty,m}$ converges toward $0$. In particular, $b_m' \rightarrow b_\infty'$ so for $m$ large enough, $b_m' \in B_1$.

We then apply the same trick as in the proof of Lemma \ref{LEM:Continuité en 0}. By Lemma \ref{LEM:Orbite préservée}, $b_m' \in G \cdot b_m \cap B_1$. Therefore, by Corollary \ref{COR:Unicité de Ness},
\begin{equation}\label{EQ:Flux composés 2}
\lim_{t \rightarrow +\infty} \tilde{\phi}_{\zeta_1}(b_m',t) = u_m \cdot \lim_{t \rightarrow +\infty} \tilde{\phi}_{\zeta_1}(b_m,t)
\end{equation}
for some $u_m \in K$. Up to an extraction, we may assume that $u_m \rightarrow u_\infty \in K$. Moreover, by Proposition \ref{PRO:Duistermaat}, $\lim_{t \rightarrow +\infty} \tilde{\phi}_{\zeta_1}(\cdot,t)$ is continuous at $b_1$ and $b_\infty'$ thus, when we take the limit of (\ref{EQ:Flux composés 2}), we obtain
\begin{equation}\label{EQ:Limit flux b_infty = 0 2}
\lim_{t \rightarrow +\infty} \tilde{\phi}_{\zeta_1}(b_\infty',t) = u_\infty \cdot \lim_{t \rightarrow +\infty} \tilde{\phi}_{\zeta_1}(b_1,t) \in K \cdot b_\infty.
\end{equation}
Moreover, for all $m$, $b_{\infty,m}$ is a zero of $\tilde{\mu}_{\zeta_m}$ by the semi-stability assumption. Therefore, $b_\infty' = \lim_{m \rightarrow +\infty} b_{\infty,m}$ is a zero of $\tilde{\mu}_{\zeta_1}$. It implies that
\begin{equation}\label{EQ:Flux constant}
    \lim_{t \rightarrow +\infty} \tilde{\phi}_{\zeta_1}(b_\infty',t) = \lim_{t \rightarrow +\infty} b_\infty' = b_\infty'.
\end{equation}
Combining (\ref{EQ:Limit flux b_infty = 0 2}) and (\ref{EQ:Flux constant}), we obtain that $K \cdot b_\infty' = K \cdot b_\infty$ hence the wanted result.

Up to replacing $B_1$ by $B_1'$, $\Gamma$ and $\tilde{\Gamma}$ (thus $\Gamma_\infty$) are continuous on their domain.
\end{proof}

\begin{remark}
    We expect the maps $(\zeta,b) \mapsto \lim_{t \rightarrow +\infty} \phi_\zeta(b,t)$ and $(\zeta,b) \mapsto \lim_{t \rightarrow +\infty} \tilde{\phi}_\zeta(b,t)$ to be continuous on the semi-stable locus even without quotienting by $K$. In this case, we would produce a family of $P_\zeta$-critical operators which is continuous with respect to $(\zeta,b)$. However, showing this continuity probably requires to study more deeply the $g(t) \in G$ and $\tilde{g}(t) \in G$ of Lemma \ref{LEM:Orbite préservée} as in the proof of Proposition \ref{PRO:Existence arc équivariant instable}, which is tedious.

    Notice that we already have the wanted continuity result when $\zeta$ is fixed by Proposition \ref{PRO:Duistermaat}.
\end{remark}

The germ of action of $G$ on $B_1$ extends to $\mU_d \times B_1$ where $G$ acts trivially on $\mU_d$. Moreover, it preserves $L_{\mathrm{ss}}$. Let $q : L_{\mathrm{ss}} \rightarrow L_{\mathrm{ss}}/G$ be the associated quotient map. Let for all $\zeta$, the polystable GIT quotient,
$$
B_1/\!/\!_\zeta G = \{q(\zeta,b)|(\zeta,b) \in L_{\mathrm{ps}}\} \subset L_{\mathrm{ss}}/G.
$$
$\mB_d^{\mathrm{int}} \subset \mB_d$ be the set of $\dbar \in \mB_d$ which are integrable \textit{i.e.} $\dbar^2 = 0$. Let $B_1^{\mathrm{int}} \subset B_1$ and $B_2^{\mathrm{int}} \subset B_2$ be the sets of $b \in B_1$ (resp. $b \in B_2$) such that $\dbar_b$ is integrable. By Proposition \ref{PRO:Propriétés Kuranishi}, they are $G$-invariant analytic subsets of $B_1$ and $B_2$ (and clearly, $B_1^{\mathrm{int}} = B_2^{\mathrm{int}} \cap B_1$).

Since $B_1^{\mathrm{int}}$ is $G$-invariant, we can define naturally
$$
B_1^{\mathrm{int}}/\!/\!_\zeta G \subset B_1/\!/\!_\zeta G.
$$

\begin{lemma}\label{LEM:Lemme Kempf--Ness P}
    For all $\zeta \in \mU_d$, the map
    $$
    \varphi_\zeta : \fonction{(\tilde{\mu}_\zeta^{-1}\{0\} \cap B_2)/K}{(\mC^d(\mu_{\infty,\zeta}^{-1}\{0\}) \cap \mB_d)/\mC^{d + 1}(\G(E,h))}{b}{\dbar_0 + \tilde{\Phi}(\zeta,b)}
    $$
    is a homeomorphism into its image. Moreover, its restriction to $B_2^{\mathrm{int}}$
    $$
    \varphi_\zeta^{\mathrm{int}} : (\tilde{\mu}_\zeta^{-1}\{0\} \cap B_2^{\mathrm{int}})/K \longrightarrow \mC^d(\mu_{\infty,\zeta}^{-1}\{0\} \cap \mB_d^{\mathrm{int}})/\mC^{d + 1}(\G(E,h))
    $$
    is also a homeomorphism into its image and it has an open image.
\end{lemma}
\begin{proof}
It is well-defined and continuous by Proposition \ref{PRO:Tranche déformée P}. Let
$$
\mathcal{I} = \left\{\mC^{d + 1}(\G(E,h)) \cdot \dbar\,\middle|\,\dbar \in \mB_d, \mu_{\infty,\zeta}(\dbar) = 0 \textrm{ and } \exists b \in B_2, \dbar \in \mC^{d + 1}(\G(E,h)) \cdot \dbar_b\right\}.
$$
By Propositions \ref{PRO:Tranche déformée P} and \ref{PRO:Zéro dans la tranche déformée}, $\mathcal{I}$ contains the image of $\varphi_\zeta$. Let $\beta : \mB_d \rightarrow B_{3,0}$ be the smooth function given by Proposition \ref{PRO:Zéro dans la tranche déformée}. If $(\dbar_1,\dbar_2) \in \mB_d^2$ are two representatives of a same $\mC^{d + 1}(\G(E,h))$-orbit in $\mathcal{I}$, then by Propositions \ref{PRO:Zéro dans la tranche déformée} and \ref{PRO:Tranche déformée P}, $\dbar_1 \in \G^\C(E) \cdot \dbar_{\beta(\dbar_1)}$ and $\dbar_2 \in \G^\C(E) \cdot \dbar_{\beta(\dbar_2)}$. Thus, $\dbar_{\beta(\dbar_2)} \in \G^\C(E) \cdot \dbar_{\beta(\dbar_1)}$. Moreover, $\beta(\dbar_1)$ and $\beta(\dbar_2)$ are zeroes of $\tilde{\mu}_\zeta$.

By Corollary \ref{COR:Sections dbar_b holomorphes}, $\beta(\dbar_2) \in G \cdot \beta(\dbar_1)$ and by Lemma \ref{LEM:Unicité K-orbite zéro mu_zeta}, $\beta(\dbar_2) \in K \cdot \beta(\dbar_1)$. By the universal property of quotients, $\beta$ factorises as a continuous map
$$
\overline{\beta} : \mathcal{I} \longrightarrow (\tilde{\mu}_\zeta^{-1}\{0\} \cap B_{3,0})/K.
$$
Moreover, Proposition \ref{PRO:Zéro dans la tranche déformée} and Lemma \ref{LEM:Unicité K-orbite zéro mu_zeta} imply that
$$
\overline{\beta} \circ \varphi_\zeta : (\tilde{\mu}_\zeta^{-1}\{0\} \cap B_2)/K \longrightarrow (\tilde{\mu}_\zeta^{-1}\{0\} \cap B_{3,0})/K
$$
is the identity map. Therefore, $\varphi_\zeta$ is injective and is a homeomorphism into its image. It also shows that $\varphi_\zeta$ is open in $\mathcal{I}$. Let
$$
\mathcal{I}^{\mathrm{int}} = \left\{\mC^{d + 1}(\G(E,h)) \cdot \dbar \in \mathcal{I}|\dbar^2 = 0\right\}.
$$
By restriction, $\varphi_\zeta^{\mathrm{int}} : (\tilde{\mu}_\zeta^{-1}\{0\} \cap B_2^{\mathrm{int}})/K \rightarrow \mathcal{I}^{\mathrm{int}}$ is also an open homeomorphism into its image. Finally, by Proposition \ref{PRO:Propriétés Kuranishi},
$$
\mathcal{I}^{\mathrm{int}} \subset \mC^d(\mu_{\infty,\zeta}^{-1}\{0\} \cap \mB_d^{\mathrm{int}})/\mC^{d + 1}(\G(E,h))
$$
is open, which proves the lemma.
\end{proof}

\begin{theorem}[Local Kempf--Ness theorem]\label{THE:Kempf--Ness}
    The maps $\Gamma$ and $\Gamma_\infty$ factorise as continuous maps
    $$
    \overline{\Gamma} : L_{\mathrm{ss}}/G \longrightarrow B_2/K, \qquad \overline{\Gamma}_\infty : L_{\mathrm{ss}}/G \longrightarrow \mB_d/\mC^{d + 1}(\G(E,h)).
    $$
    Moreover, for all $\zeta \in \mU_d$, the restrictions
    $$
    \overline{\Gamma}_\zeta : B_1/\!/\!_\zeta G \longrightarrow (\mu_\zeta^{-1}\{0\} \cap B_2)/K, \qquad \overline{\Gamma}_{\infty,\zeta} : B_1/\!/\!_\zeta G \longrightarrow (\mC^d(\mu_{\infty,\zeta}^{-1}\{0\}) \cap \mB_d)/\mC^{d + 1}(\G(E,h)),
    $$
    $$
    \overline{\Gamma}_\zeta^{\mathrm{int}} : B_1^{\mathrm{int}}/\!/\!_\zeta G \longrightarrow (\mu_\zeta^{-1}\{0\} \cap B_2^{\mathrm{int}})/K, \qquad \overline{\Gamma}_{\infty,\zeta}^{\mathrm{int}} : B_1^{\mathrm{int}}/\!/\!_\zeta G \longrightarrow (\mC^d(\mu_{\infty,\zeta}^{-1}\{0\}) \cap \mB_d^{\mathrm{int}})/\mC^{d + 1}(\G(E,h))
    $$
    are homeomorphisms into their images. Moreover, they have an open image except maybe $\overline{\Gamma}_{\infty,\zeta}$.
\end{theorem}
\begin{proof}
The existence and the continuity of $\overline{\Gamma}$ and $\overline{\Gamma}_\infty$ is simply the fact that $\Gamma$ and $\Gamma_\infty$ are $G$-invariant and continuous by Theorem \ref{THE:Continuité des opérateurs P-critiques}, and the universal property of the quotient topology. Let $\zeta \in \mU_d$. $\overline{\Gamma}_\zeta$, $\overline{\Gamma}_{\infty,\zeta}$, $\overline{\Gamma}_\zeta^{\mathrm{int}}$, $\overline{\Gamma}_{\infty,\zeta}^{\mathrm{int}}$ are continuous as restrictions of continuous functions.

When $(\zeta,b) \in L_{\mathrm{ps}}$, the flow $\phi_\zeta(b,\cdot)$ (resp. $\tilde{\phi}_\zeta(b,\cdot)$) converges in $G \cdot b \cap B_2$ toward a zero of $\mu_\zeta$ (resp. $\tilde{\mu}_\zeta$), and being integrable is a closed condition so these four functions indeed take their respective values in $(\mu_\zeta^{-1}\{0\} \cap B_2)/K$, $(\mC^d(\mu_{\infty,\zeta}^{-1}\{0\}) \cap \mB_d)/\mC^{d + 1}(\G(E,h))$, $(\mu_\zeta^{-1}\{0\} \cap B_2^{\mathrm{int}})/K$ and $(\mC^d(\mu_{\infty,\zeta}^{-1}\{0\}) \cap \mB_d^{\mathrm{int}})/\mC^{d + 1}(\G(E,h))$.

Similarly, if $\tilde{\Gamma}$ is the function introduced in the proof of Theorem \ref{THE:Continuité des opérateurs P-critiques}, we can define $\overline{\tilde{\Gamma}}$ and $\overline{\tilde{\Gamma}}_\zeta$. Notice that $\overline{\Gamma}_{\infty,\zeta} = \varphi_\zeta \circ \overline{\tilde{\Gamma}}$ where $\varphi_\zeta$ is introduced in Lemma \ref{LEM:Lemme Kempf--Ness P}.

Let, $I_\zeta$ be the image of $\overline{\Gamma}_\zeta$ and $\tilde{I}_{\zeta}$ the image of $\overline{\tilde{\Gamma}}_\zeta$. Let
$$
\iota_\zeta : \fonction{I_\zeta}{B_1/\!/\!_\zeta G}{K \cdot b}{G \cdot b \cap B_1}, \qquad \tilde{\iota}_\zeta : \fonction{\tilde{I}_\zeta}{B_1/\!/\!_\zeta G}{K \cdot b}{G \cdot b \cap B_1}.
$$
These maps are well-defined because for all $K \cdot b_\infty$ in the image of $\overline{\Gamma}_\zeta$ or $\overline{\tilde{\Gamma}}_\zeta$, $b_\infty$ is built as a point in the $G$-orbit of some $b \in B_1$ hence $G \cdot b_\infty \cap B_1 \neq \emptyset$. Moreover, they are continuous.

If $(\zeta,b) \in L_{\mathrm{ps}}$, $\lim_{t \rightarrow +\infty} \phi_\zeta(b,t)$ and $\lim_{t \rightarrow +\infty} \tilde{\phi}_\zeta(b,t)$ both lie in $G \cdot b$. Therefore, $\iota_\zeta$ is the left inverse of $\overline{\Gamma}_\zeta$ and $\tilde{\iota}_\zeta$ is the left inverse of $\overline{\tilde{\Gamma}}_\zeta$. It shows that $\overline{\Gamma}_\zeta$ and $\overline{\tilde{\Gamma}}_\zeta$ are homeomorphisms into their respective images. Moreover, we have $\overline{\Gamma}_{\infty,\zeta} = \varphi_\zeta \circ \overline{\tilde{\Gamma}}$, thus by Lemma \ref{LEM:Lemme Kempf--Ness P}, $\overline{\Gamma}_{\infty,\zeta}$ is a homeomorphism into its image.

By restriction, $\overline{\Gamma}_\zeta^{\mathrm{int}}$ and $\overline{\Gamma}_{\infty,\zeta}^{\mathrm{int}}$ are homeomorphisms into their images. All we have left to show is that the images of $\overline{\Gamma}_\zeta$, $\overline{\Gamma}_\zeta^{\mathrm{int}}$ and $\overline{\Gamma}_{\infty,\zeta}^{\mathrm{int}}$ are open. Let $K \cdot b_\infty = \overline{\Gamma}_\zeta(G \cdot b \cap B_1)$ for some $b \in B_1$. By $\mu_\zeta$-polystability of $b$, there is a $g \in G$ such that $b_\infty = g \cdot b$. Let $b_m \rightarrow b_\infty$ be a sequence of zeroes of $\mu_\zeta$ in $B_2$. We have $g^{-1} \cdot b_m \rightarrow b \in B_1$ so for all $m$ large enough, $g^{-1} \cdot b_m \in B_1$. By Corollary \ref{COR:Unicité de Ness}, for all such $m$,
$$
\lim_{t \rightarrow +\infty} \phi_\zeta(g^{-1} \cdot b_m,t) \in K \cdot \left(\lim_{t \rightarrow +\infty} \phi_\zeta(b_m,t)\right) = K \cdot b_m
$$
because $b_m$ is a zero of $\mu_\zeta$. It implies that $b_m = \overline{\Gamma}_\zeta(g^{-1} \cdot b_m)$ lies in the image of $\overline{\Gamma}_\zeta$ if $m$ is large enough. Then, notice that the image of  $\overline{\Gamma}_\zeta^{\mathrm{int}}$ is the image of $\overline{\Gamma}_\zeta$ intersected with $B_2^{\mathrm{int}}$, which is open in $(\mu_\zeta^{-1}\{0\} \cap B_2^{\mathrm{int}})/K$.

With the same proof, we show that the image of $\overline{\tilde{\Gamma}}_\zeta^{\mathrm{int}}$ is open in $(\tilde{\mu}_\zeta^{-1}\{0\} \cap B_2^{\mathrm{int}})/K$. By the second part of Lemma \ref{LEM:Lemme Kempf--Ness P}, the image of
$$
\overline{\Gamma}_{\infty,\zeta}^{\mathrm{int}} = \varphi_\zeta^{\mathrm{int}} \circ \overline{\tilde{\Gamma}}_\zeta^{\mathrm{int}}
$$
is open.
\end{proof}

\section{Application with the $J$-equation and the dHYM equation on a surface}\label{SEC:Exemple avec la J-équation}

In this section, we build explicit examples of simple rank $2$ bundles on a surface that satisfy the $J$-equation and/or the dHYM equation with respect to a varying metric by using the theorems of the above sections. We start with a few general results on sub-solutions on line bundles when the equation has degree at most $2$. In particular, we show their uniqueness modulo the unitary gauge group.

\subsection{General results on surfaces}

\begin{proposition}\label{PRO:Somme de sous-solutions fibrés}
    Assume the $P$-critical equation has degree at most $2$ \textit{i.e.} $\zeta_k = 0$ when $k \geq 3$. If $(L_i,h_i)_{1 \leq i \leq m}$ are Hermitian line bundles and for all $i$, $\dbar_i$ is a sub-solution to the $P$-critical equation then, $\dbar_E = \bigoplus_{i = 1}^m \dbar_i$ on $E = \bigoplus_{i = 1}^m (L_i,h_i)$ is a sub-solution too.
\end{proposition}
\begin{proof}
Let for all $i$, $\hat{F}_i = \hat{F}(L_i,h_i,\dbar_i)$ so
$$
\hat{F}(E,h_E,\dbar_E) =
\begin{pmatrix}
    \hat{F}_1 & 0 & \hdots & 0\\
    0 & \hat{F}_2 & \hdots & 0\\
    \vdots & \vdots & \ddots & \vdots\\
    0 & 0 & \hdots & \hat{F}_m
\end{pmatrix}
$$
Let $x \in X$ and $v \in T_x^{1,0}X \otimes \End(E_x)\backslash\{0\}$ that we decomposes as
$$
v =
\begin{pmatrix}
    v_{11} & v_{12} & \hdots & v_{1m}\\
    v_{21} & v_{22} & \hdots & v_{2m}\\
    \vdots & \vdots & \ddots & \vdots\\
    v_{m1} & v_{m2} & \hdots & v_{mm}
\end{pmatrix},
$$
with $v_{ij} \in T_x^{1,0}X \otimes \Hom(L_{j,x},L_{i,x})$. We have
\begin{align}
    \tr(\hat{F}(E,h_E,\dbar_E) \wedge \i v \wedge v^\dagger) & = \sum_{i,j = 1}^m \hat{F}_i \wedge \i v_{ij} \wedge v_{ij}^\dagger, \label{EQ:Expression trace F wedge i v wedge v dagger 1}\\
    \tr(\i v \wedge \hat{F}(E,h_E,\dbar_E) \wedge v^\dagger) & = \sum_{i,j = 1}^m \hat{F}_j \wedge \i v_{ij} \wedge v_{ij}^\dagger. \label{EQ:Expression trace F wedge i v wedge v dagger 2}
\end{align}
On the other hand, for all $i$, since we are in rank $1$,
$$
\tr([\hat{F}_i \wedge \i v_{ij} \wedge v_{ij}^\dagger]_\sym) = \hat{F}_i \wedge \i v_{ij} \wedge v_{ij}^\dagger.
$$
Then, by using the commutativity inside the trace, the symmetric product $\tr([\hat{F}(E,h_E,\dbar_E),\i v,v^\dagger]_\sym)$ is an average whose half of the terms equals (\ref{EQ:Expression trace F wedge i v wedge v dagger 1}) and the other half of the terms equals (\ref{EQ:Expression trace F wedge i v wedge v dagger 2}). Therefore,
\begin{align*}
    \tr([\hat{F}(E,h_E,\dbar_E),\i v,v^\dagger]_\sym) & = \frac{1}{2}\sum_{i,j = 1}^m \hat{F}_i \wedge \i v_{ij} \wedge v_{ij}^\dagger + \frac{1}{2}\sum_{i,j = 1}^m \hat{F}_j \wedge \i v_{ij} \wedge v_{ij}^\dagger\\
    & = \frac{1}{2}\sum_{i,j = 1}^m \tr([\hat{F}_i \wedge \i v_{ij} \wedge v_{ij}^\dagger]_\sym) + \tr([\hat{F}_j \wedge \i v_{ij} \wedge v_{ij}^\dagger]_\sym).
\end{align*}
We deduce that
\begin{align*}
    \tr([\mP_\zeta'(E,\dbar_E,h_E) \wedge \i v \wedge v^\dagger]_\sym) & = 2\zeta_2 \wedge \tr([\hat{F}(E,h_E,\dbar_E),\i v,v^\dagger]_\sym) + \zeta_1 \wedge \tr([\i v,v^\dagger]_\sym)\\
    & = \zeta_2 \wedge \left(\sum_{i,j = 1}^m \tr([\hat{F}_i \wedge \i v_{ij} \wedge v_{ij}^\dagger]_\sym) + \tr([\hat{F}_j \wedge \i v_{ij} \wedge v_{ij}^\dagger]_\sym)\right)\\
    & + \zeta_1 \wedge \left(\sum_{i,j = 1}^m \i v_{ij} \wedge v_{ij}^\dagger\right)\\
    & = \frac{1}{2}\sum_{i,j = 1}^m 2\zeta_2 \wedge \tr([\hat{F}_i \wedge \i v_{ij} \wedge v_{ij}^\dagger]_\sym) + \zeta_1 \wedge \i v_{ij} \wedge v_{ij}^\dagger\\
    & + \frac{1}{2}\sum_{i,j = 1}^m 2\zeta_2 \wedge \tr([\hat{F}_j \wedge \i v_{ij} \wedge v_{ij}^\dagger]_\sym) + \zeta_1 \wedge \i v_{ij} \wedge v_{ij}^\dagger\\
    & = \frac{1}{2}\sum_{i,j = 1}^m [\mP_\zeta'(L_i,\dbar_i,h_i) \wedge \i v_{ij} \wedge v_{ij}^\dagger] + [\mP_\zeta'(L_j,\dbar_j,h_j) \wedge \i v_{ij} \wedge v_{ij}^\dagger]\\
    & > 0.
\end{align*}
Thus, $\dbar_E$ is a sub-solution on $(E,h_E)$.
\end{proof}

\begin{corollary}
    If a solution to the $P$-critical equation exists on a Hermitian line bundle $(L,h)$ and the equation has degree at most $2$ (in particular, when $X$ is a surface), the solution is unique modulo $\G(L,h)$.
\end{corollary}
\begin{proof}
It is a consequence of Propositions \ref{PRO:Somme de sous-solutions fibrés} and \ref{PRO:caractérisation unicité}.
\end{proof}

\subsection{An example}

Let $X = \mathrm{Bl}_p\P^2$, $\pi : X \rightarrow \P^2$ the projection, $H_{\P^2} \in H^{1,1}(\P^2)$ the hyperplane class, $H = \pi^*H_{\P^2} \in H^{1,1}(X)$ its pull-back and $D \in H^{1,1}(X)$ the class of the exceptional divisor $\mathrm{Ex}$. Recall that $H^2 = 1$, $D^2 = -1$ and $H \cup D = 0$. In particular, for all $a > b > 0$ integers,
\begin{equation}\label{EQ:Nakai Moishezon}
    (aH - bD)^2 = a^2 - b^2 > 0, \qquad (aH - bD) \cup (H - E) = a - b > 0, \qquad (aH - bD) \cup D = b > 0.
\end{equation}
The cone of curves of $X$ is generated by $H - D$ and $D$. Therefore, the three inequalities (\ref{EQ:Nakai Moishezon}) show that $\OX_X(bH - aD)$ satisfies the Nakai--Moishezon criterion.

In particular, there is a Kähler form $\omega \in 2H - D$ and the line bundles $\mL_1 = \OX_X(13H - 11D)$ and $\mL_2 = \OX_X(6H - 2D)$ are ample. We have
$$
\frac{[\omega] \cup \ch_1(\mL_1)}{2\ch_2(\mL_1)} = \frac{(2H - D) \cup (13H - 11D)}{(13H - 11D)^2} = \frac{5}{16},
$$
$$
\frac{[\omega] \cup \ch_1(\mL_2)}{2\ch_2(\mL_2)} = \frac{(2H - D) \cup (6H - 2D)}{(6H - 2D)^2} = \frac{5}{16}.
$$
In other words, the constant for the $J$-equation with respect to $\omega$ of $\mL_1$ and $\mL_2$ is the same. Let $c_J = \frac{5}{16}$. The $J$-equation on these bundles is given by
$$
c_J\hat{F}^2 - \omega \wedge \hat{F} = 0.
$$
Let us show that both of these bundles admit a solution to this equation. We use Song's criterion \cite[Corollary 1.2]{Song} which says that $\mL_i$ ($i \in \{1,2\}$) admits a solution to the $J$-equation if and only if for all curve $C \subset X$,
$$
\frac{[\omega] \cup [C]}{\ch_1(\mL_i) \cup [C]} < 2c_J = \frac{5}{8}.
$$
With the same argument as previously, it is enough to verify it when $[C] = H - D$ and $[C] = D$. We have
$$
\frac{[\omega] \cup (H - D)}{\ch_1(\mL_1) \cup (H - D)} = \frac{(2H - D) \cup (H - D)}{(13H - 11D) \cup (H - D)} = \frac{1}{2} < \frac{5}{8},
$$
$$
\frac{[\omega] \cup D}{\ch_1(\mL_1) \cup D} = \frac{(2H - D) \cup D}{(13H - 11D) \cup D} = \frac{1}{11} < \frac{5}{8},
$$
$$
\frac{[\omega] \cup (H - D)}{\ch_1(\mL_2) \cup (H - D)} = \frac{(2H - D) \cup (H - D)}{(6H - 2D) \cup (H - D)} = \frac{1}{4} < \frac{5}{8},
$$
$$
\frac{[\omega] \cup D}{\ch_1(\mL_2) \cup D} = \frac{(2H - D) \cup D}{(6H - 2D) \cup D} = \frac{1}{2} < \frac{5}{8}.
$$
We deduce that both $\mL_1$ and $\mL_2$ admit a solution to the $J$-equation. Moreover, these respective solutions are automatically $J$-positive (\textit{i.e.} are sub-solutions) by \cite[Proof of Theorem 1.1]{Song_Weinkove}. Let $h_1$ and $h_2$ be metrics on these line bundles such that the curvatures $\hat{F}_1$ of $\mL_1 = (L_1,h_1,\dbar_1)$ and $\hat{F}_2$ of $\mL_2 = (L_2,h_2,\dbar_2)$ satisfy, for $i \in \{1,2\}$,
$$
c_J\hat{F}_i^2 - \omega \wedge \hat{F}_i = 0.
$$
Therefore, $\E_0 = \mL_1 \oplus \mL_2$ satisfies the same equation and by Proposition \ref{PRO:Somme de sous-solutions fibrés}, it is $J$-positive.

Now, let $\alpha$ be a closed $(1,1)$-form in $H - 2D$. Let $\varepsilon_1 > 0$ and $\varepsilon_2 \in \R$. Let
$$
\omega_{\varepsilon_2} = \omega + \varepsilon_2\alpha \in 2H - D + \varepsilon_2(H - 2D).
$$
If $\abs{\varepsilon_2}$ is small enough, it is a positive form. Let us consider the dHYM equation with respect to the Kähler form $\varepsilon_1\omega_{\varepsilon_2}$ on some holomorphic Hermitian bundle $\E = (E,h,\dbar)$ of curvature $\hat{F}$. Assume moreover that $\E$ satisfies Takahashi's condition for the $J$-equation with respect to $[\omega]$ \textit{i.e.} $\ch_2(E) > 0$ and $[\omega] \cup \ch_1(E) > 0$.
\begin{align*}
    & \Im(\overline{Z_{\mathrm{dHYM},\varepsilon_1\omega_{\varepsilon_2}}(E)}(\varepsilon_1\omega_{\varepsilon_2}\Id_E + \i\hat{F})^2)\\
    =\ & \Im\left(\overline{(\i^2\e^{-\i[\varepsilon_1\omega_{\varepsilon_2}]} \cup \ch(E)})_{(2,2)}(\varepsilon_1\omega_{\varepsilon_2}\Id_E + \i\hat{F})^2\right)\\
    =\ & \Im\left(\left(-\ch_2(E) - \i\varepsilon_1[\omega_{\varepsilon_2}] \cup \ch_1(E) + \varepsilon_1^2\frac{[\omega_{\varepsilon_2}]^2}{2}\rk(E)\right)(-\hat{F}^2 + 2\i\varepsilon_1\omega_{\varepsilon_2} \wedge \hat{F} + \varepsilon_1^2\omega_{\varepsilon_2}^2\Id_E)\right)\\
    =\ & \varepsilon_1([\omega_{\varepsilon_2}] \cup \ch_1(E))\hat{F}^2 - 2\varepsilon_1\ch_2(E)\omega_{\varepsilon_2} \wedge \hat{F} + \varepsilon_1^3([\omega_{\varepsilon_2}]^2\rk(E))\omega_{\varepsilon_2} \wedge \hat{F} - \varepsilon_1^3([\omega_{\varepsilon_2}] \cup \ch_1(E))\omega_{\varepsilon_2}^2\Id_E.
\end{align*}
Let us set
$$
c_J(E,\varepsilon_2) = \frac{[\omega_{\varepsilon_2}] \cup \ch_1(E)}{2\ch_2(E)}, \qquad c_{\mathrm{HYM}}(E,\varepsilon_2) = \frac{[\omega_{\varepsilon_2}] \cup \ch_1(E)}{[\omega_{\varepsilon_2}]^2\rk(E)}, \qquad C(E,\varepsilon_2) = \frac{[\omega_{\varepsilon_2}]^2\rk(E)}{2\ch_2(E)}.
$$
After a rescaling by $\frac{1}{2\varepsilon_1\ch_2(E)}$, the dHYM equation for $\E$ with respect to the metric $\varepsilon_1\omega_{\varepsilon_2}$ is given by
\begin{equation}\label{EQ:dHYM epsilon_1 omega epsilon_2}
c_J(E,\varepsilon_2)\hat{F}^2 - \omega_{\varepsilon_2} \wedge \hat{F} + C(E,\varepsilon_2)\varepsilon_1^2(\omega_{\varepsilon_2} \wedge \hat{F} - c_{\mathrm{HYM}}(E,\varepsilon_2)\omega_{\varepsilon_2}^2\Id_E) = 0.
\end{equation}
The positivity conditions on the Chern classes of $E$ and on $\omega_{\varepsilon_2}$ imply that all the terms in (\ref{EQ:dHYM epsilon_1 omega epsilon_2}) are well-defined.

We find here a usual result : the $J$-equation is the low-volume limit of the dHYM equation \textit{i.e.} the $J$-equation with respect to $\omega_{\varepsilon_2}$ is the limit equation of the dHYM equation with respect to $\varepsilon_1\omega_{\varepsilon_2}$ when $\varepsilon_1 \rightarrow 0$. In the large volume limit, we find the HYM equation with respect to $\omega_{\varepsilon_2}$,
$$
\omega_{\varepsilon_2} \wedge \hat{F} - c_{\mathrm{HYM}}(E,\varepsilon_2)\omega_{\varepsilon_2}^2\Id_E = 0.
$$

As the sum of two ample bundles, $\E_0$ satisfies $\ch_2(\E_0) > 0$ and $\ch_1(\E_0) \cup [\omega] > 0$. More generally, any other bundle with the same topology satisfies these conditions. Let us build simple deformations of $\E_0$ and let us determine for which values of $\varepsilon_1$ and $\varepsilon_2$ they admit a solution to (\ref{EQ:dHYM epsilon_1 omega epsilon_2}).

\begin{lemma}\label{LEM:H1 non trivial}
    For all integers $(a,b)$ with $b \geq \max\{a + 2,1\}$, $H^1(X,\OX_X(aH - bD)) \neq 0$.
\end{lemma}
\begin{proof}
Let $\iota : \mathrm{Ex} \hookrightarrow X$ be the inclusion of the exceptional divisor. Consider the ideal sheaf exact sequence,
\begin{equation}\label{EQ:Suite exacte faisceau idéal}
\begin{tikzcd}
    0 \ar{r} & \OX_X(-D) \ar{r} & \OX_X \ar{r} & \iota_*\OX_{\mathrm{Ex}} \ar{r} & 0
\end{tikzcd}.
\end{equation}
Let $(a,b)$ be integers with $b \geq \max\{a + 2,1\}$. By the projection formula, we have
$$
\iota_*\OX_{\mathrm{Ex}} \otimes \OX_X(aH - (b - 1)D) = \iota_*(\OX_{\mathrm{Ex}} \otimes \iota^*\OX_X(aH - (b - 1)D)) = \iota_*\OX_{\mathrm{Ex}}(b - 1)
$$
because $\iota^*D = -H_{\mathrm{Ex}}$ where $H_{\mathrm{Ex}} \in H^{1,1}(\mathrm{Ex},\Z)$ is the hyperplane class of $\mathrm{Ex} \cong \P^1$. Therefore, when we tensorise (\ref{EQ:Suite exacte faisceau idéal}) by $\OX_X(aH - (b - 1)D)$, we obtain the exact sequence
$$
\begin{tikzcd}
    0 \ar{r} & \OX_X(aH - bD) \ar{r} & \OX_X(aH - (b - 1)D) \ar{r} & \iota_*\OX_{\mathrm{Ex}}(b - 1) \ar{r} & 0
\end{tikzcd}.
$$
It gives a long exact sequence of cohomology groups,
\begin{adjustwidth}{-10cm}{-10cm}
$$
\begin{tikzcd}
    0 \ar{r} & H^0(X,\OX_X(aH - bD)) \ar{r} & H^0(X,\OX_X(aH - (b - 1)D)) \ar{r} & H^0(X,\iota_*\OX_{\mathrm{Ex}}(b - 1)) \ar[out = 0,in = 180,looseness = 2]{dll}\\
    & H^1(X,\OX_X(aH - bD)) \ar{r} & H^1(X,\OX_X(aH - (b - 1)D)) \ar{r} & H^1(X,\iota_*\OX_{\mathrm{Ex}}(b - 1))
\end{tikzcd}.
$$
\end{adjustwidth}
Then, the global sections of $\OX_X(aH - (b - 1)D)$ are the global sections of $\OX_{\P^2}(a)$ which vanish at $p \in \P^2$ with multiplicity at least $b - 1$. Since $b - 1 > a$, there are no such sections except to zero section. Similarly with $\OX_X(aH - bD)$. Then, by exactness of $\iota_*$ and because $b > 0$,
$$
H^0(X,\iota_*\OX_{\mathrm{Ex}}(b - 1)) = H^0(\mathrm{Ex},\OX_{\mathrm{Ex}}(b - 1)) \cong \C^b \neq 0,
$$
$$
H^1(X,\iota_*\OX_{\mathrm{Ex}}(b - 1)) = H^1(\mathrm{Ex},\OX_{\mathrm{Ex}}(b - 1)) = 0.
$$
We deduce that we have an exact sequence
$$
\begin{tikzcd}
    0 \ar{r} & \C^b \ar{r} & H^1(X,\OX_X(aH - bD)) \ar{r} & H^1(X,\OX_X(aH - (b - 1)D)) \ar{r} & 0
\end{tikzcd}.
$$
It implies that $\dim(H^1(X,\OX_X(aH - bD))) \geq b > 0$.
\end{proof}

\begin{corollary}\label{COR:Existence E1 et E2}
    There exists simple bundles $\E_1$ and $\E_2$ together with exact sequences
    $$
    \begin{tikzcd}
        0 \ar{r} & \mL_1 \ar{r} & \E_1 \ar{r} & \mL_2 \ar{r} & 0
    \end{tikzcd}, \qquad
    \begin{tikzcd}
        0 \ar{r} & \mL_2 \ar{r} & \E_2 \ar{r} & \mL_1 \ar{r} & 0
    \end{tikzcd}.
    $$
\end{corollary}
\begin{proof}
Let us verify that
$$
\mathrm{Ext}^1(\mL_2,\mL_1) = H^1(X,\mL_1 \otimes \mL_2^\lor) = H^1(X,\OX_X(7H - 9D))
$$
and
$$
\mathrm{Ext}^1(\mL_1,\mL_2) = H^1(X,\mL_2 \otimes \mL_1^\lor) = H^1(X,\OX_X(-7H + 9D))
$$
are non-zero. For the first one, we simply apply Lemma \ref{LEM:H1 non trivial}. For the second one, by \cite[Proposition 2.4.3]{Huybrechts} and \cite[Proposition 2.5.5]{Huybrechts}, the canonical bundle of $X$ is given by $K_X = \OX_X(-3H + D)$ and by Serre duality,
$$
H^1(X,\OX_X(-7H + 9D)) = H^1(X,\OX_X(-7H + 9D)^\lor \otimes K_X) = H^1(X,\OX_X(4H - 8D)).
$$
We conclude with Lemma \ref{LEM:H1 non trivial}. It means that we can choose $\E_1$ and $\E_2$ as wanted such that the associated exact sequences don't split. Since both $\mL_1$ and $\mL_2$ are simple, a sufficient condition for $\E_1$ and $\E_2$ to be simple is $\Hom(\mL_1,\mL_2) = \Hom(\mL_2,\mL_1) = 0$, which is the case by Proposition \ref{PRO:Polysimplicité} (or simply by noticing that neither $\mL_1 \otimes \mL_2^\lor$ nor its dual is nef).
\end{proof}

Let $\dbar_1$ and $\dbar_2$ be any complex structures on $E$ such that for $i \in \{1,2\}$, $(E,\dbar_i) = \E_i$. Since the $\E_i$ are built as extensions of $\mL_2$ by $\mL_1$ and of $\mL_1$ by $\mL_2$, the $\dbar_i$ can be chosen arbitrarily close to $\dbar_0$ for the $\mC^\infty$ topology.

\begin{proposition}
    Let $A(\varepsilon_1,\varepsilon_2) = \varepsilon_2 + \frac{1}{128}(1 - \varepsilon_2^2)(5 - 11\varepsilon_2)\varepsilon_1^2$. There are $\varepsilon_0,\epsilon_0 > 0$ such that for all $0 \leq \varepsilon_1 \leq \varepsilon_0$ and $-\varepsilon_0 \leq \varepsilon_2 \leq \varepsilon_0$,
    \begin{enumerate}
        \item $\E_0 = \mL_1 \oplus \mL_2$ admits a solution $\dbar \in \G^\C(E) \cdot \dbar_0$ to (\ref{EQ:dHYM epsilon_1 omega epsilon_2}) with $\norme{\dbar - \dbar_0}_{\mC^1} \leq \epsilon_0$ if and only if $A(\varepsilon_1,\varepsilon_2) = 0$,
        \item $\E_1$ admits a solution $\dbar \in \G^\C(E) \cdot \dbar_1$ to (\ref{EQ:dHYM epsilon_1 omega epsilon_2}) with $\norme{\dbar - \dbar_0}_{\mC^1} \leq \epsilon_0$ if and only if $A(\varepsilon_1,\varepsilon_2) < 0$,
        \item $\E_2$ admits a solution $\dbar \in \G^\C(E) \cdot \dbar_2$ to (\ref{EQ:dHYM epsilon_1 omega epsilon_2}) with $\norme{\dbar - \dbar_0}_{\mC^1} \leq \epsilon_0$ if and only if $A(\varepsilon_1,\varepsilon_2) > 0$.
    \end{enumerate}
\end{proposition}
\begin{proof}
Let $L_1$ and $L_2$ be the respective smooth structures of $\mL_1$ and $\mL_2$. Let $E = L_1 \oplus L_2$ be the common smooth structure of $\E_0$, $\E_1$ and $\E_2$. We have, for all $\varepsilon_2$,
\begin{align*}
    [\omega_{\varepsilon_2}] \cup \ch_1(L_1) & = (2H - D + \varepsilon_2(H - 2D)) \cup (13H - 11D) = 15 - 9\varepsilon_2,\\
    [\omega_{\varepsilon_2}] \cup \ch_1(L_2) & = (2H - D + \varepsilon_2(H - 2D)) \cup (6H - 2D) = 10 + 2\varepsilon_2,\\
    [\omega_{\varepsilon_2}] \cup \ch_1(E) & = [\omega_{\varepsilon_2}] \cup \ch_1(L_1) + [\omega_{\varepsilon_2}] \cup \ch_1(L_2) = 25 - 7\varepsilon_2,\\
    [\omega_{\varepsilon_2}]^2 & = (2H - D + \varepsilon_2(H - 2D))^2 = 3 - 3\varepsilon_2^2.
\end{align*}
And
\begin{align*}
    2\ch_2(L_1) & = (13H - 11D)^2 = 48,\\
    2\ch_2(L_2) & = (6H - 2D)^2 = 32,\\
    2\ch_2(E) & = 2\ch_2(L_1) + 2\ch_2(L_2) = 80.
\end{align*}
Let $\zeta_{\varepsilon_1,\varepsilon_2} \in H^{*,*}(X,\End(E))$ such that the $P_{\zeta_{\varepsilon_1,\varepsilon_2}}$-critical equation for $\E_0$, $\E_1$ and $\E_2$ is (\ref{EQ:dHYM epsilon_1 omega epsilon_2}). We have $P_{\zeta_{\varepsilon_1,\varepsilon_2}}(E) = 0$ and
\begin{align*}
    & P_{\zeta_{\varepsilon_1,\varepsilon_2}}(L_1)\\
    =\ & c_J(E,\varepsilon_2)2\ch_2(L_1) - [\omega_{\varepsilon_2}] \cup \ch_1(L_1) + C(E,\varepsilon_2)\varepsilon_1^2([\omega_{\varepsilon_2}] \cup \ch_1(L_1) - c_{\mathrm{HYM}}(E,\varepsilon_2)[\omega_{\varepsilon_2}]^2\rk(L_1))\\
    =\ & \frac{[\omega_{\varepsilon_2}] \cup \ch_1(E)}{2\ch_2(E)}2\ch_2(L_1) - [\omega_{\varepsilon_2}] \cup \ch_1(L_1) + \frac{[\omega_{\varepsilon_2}]^2\rk(E)}{2\ch_2(E)}\varepsilon_1^2\left([\omega_{\varepsilon_2}] \cup \ch_1(L_1) - [\omega_{\varepsilon_2}] \cup \ch_1(E)\frac{\rk(L_1)}{\rk(E)}\right)\\
    =\ & \frac{25 - 7\varepsilon_2}{80}48 - (15 - 9\varepsilon_2) + \frac{(3 - 3\varepsilon_2^2)2}{80}\varepsilon_1^2\left(15 - 9\varepsilon_2 - (25 - 7\varepsilon_2)\frac{1}{2}\right)\\
    =\ & \frac{24}{5}\varepsilon_2 + \frac{3}{80}(1 - \varepsilon_2^2)(5 - 11\varepsilon_2)\varepsilon_1^2.
\end{align*}
We deduce that
\begin{equation}\label{EQ:P_zeta(L_1) exemple}
    P_{\zeta_{\varepsilon_1,\varepsilon_2}}(L_1) < (\leq)\ 0 \qquad \Longleftrightarrow \qquad A(\varepsilon_1,\varepsilon_2) = \varepsilon_2 + \frac{1}{128}(1 - \varepsilon_2^2)(5 - 11\varepsilon_2)\varepsilon_1^2 < (\leq)\ 0.
\end{equation}
Let $V = H^{0,1}(X,\End(\E_0))$ and $B_2 \subset V$ a small enough ball centred at $0$ so we can apply Theorem \ref{THE:Déformation P-critique} at $\E_0$. Let $\mB_1$ be the neighbourhood of $\dbar_0$ for the $\mC^1$ topology, $\mU_1$ the neighbourhood of $\zeta_{0,0}$ for the $\mC^0$ topology and $B_1 \subset B_2$ an other ball centred at $0$ given by this theorem.

We may assume that $\mB_1$ is the ball of centre $\dbar_0$ and radius $\epsilon_0 > 0$. Let $\varepsilon_0 > 0$ such that, when $0 \leq \varepsilon_1 \leq \varepsilon_0$ and $-\varepsilon_0 \leq \varepsilon_2 \leq \varepsilon_0$, $\zeta_{\varepsilon_1,\varepsilon_2} \in \mU_1$. Since $\dbar_1$ and $\dbar_2$ can be chosen arbitrarily close to $\dbar_0$ and are integrable, by Proposition \ref{PRO:Propriétés Kuranishi}, there are $b_1 \in B_1$ and $b_2 \in B_1$ such that for each $i$, $\dbar_i \in \G^\C(E) \cdot \dbar_{b_i}$.

The sub-bundles $\F$ of $\E_0$ such that $P_{\zeta_{0,0}}(\F) = 0$ are exactly $\mL_1$ and $\mL_2$. Therefore, the admissible sub-bundles of $\E_{b_1}$ are $0$, $L_1$ and $E$ and the admissible sub-bundles of $\E_{b_2}$ are $0$, $L_2$ and $E$. We deduce the result of the Proposition by Theorem \ref{THE:Déformation P-critique}, (\ref{EQ:P_zeta(L_1) exemple}) and the fact that for all $\varepsilon_1,\varepsilon_2$, $P_{\zeta_{\varepsilon_1,\varepsilon_2}}(L_2) = P_{\zeta_{\varepsilon_1,\varepsilon_2}}(E) - P_{\zeta_{\varepsilon_1,\varepsilon_2}}(L_1) = -P_{\zeta_{\varepsilon_1,\varepsilon_2}}(L_1)$.
\end{proof}

\bibliographystyle{plainurl}

\end{document}